\documentclass[11pt]{amsart}
\usepackage{color}

\long\def\red#1{\textcolor {black}{#1}} \long\def\blue#1{\textcolor
{blue}{#1}} 

\usepackage{graphicx}
\usepackage{amssymb,amsfonts,amsthm}
\usepackage{amsmath,amsthm,amssymb}
\usepackage{amscd}
\usepackage{amsfonts}

\def\z{{\mathfrak z}}
\def\h{{\mathfrak h}}
\def\MM{\mathcal {M}}
\def\Q{\mathbb{Q}}
\def\d{{\mathrm{dist}}}
\def\DG{{\mathrm{DG}}}

\def\cc{{\mathcal C}}

\def\ZZ{{\mathbb Z}}
\def\QQ{{\mathbb Q}}
\def\RR{{\mathbb R}}

\def\HH{{\mathbb H}}

\def\CC{{\mathbb C}}
\def\k{{\bf k}}
\def\O{{\mathcal O}}
\def\calS{{\mathcal S}}
\def\OS{{\O_\calS}}
\def\OST{{\O^{2}_\calS}}
\def\OSS{{\O^{*}_\calS}}
\def\kS{{\k_{\calS}}}

\def\A{{\mathcal A}}
\def\F{{\mathcal F}}
\def\p{{\mathfrak{p}}}

\newcommand{\base}{\operatorname{base}}
\newcommand{\uas}{{$\omega$-almost surely}}
\newcommand{\uass}{$\omega$-almost surely }

\def\MCG{\mathcal{MCG}}
\newcommand{\SL}{\operatorname{SL}}

\newcommand{\tr}{\operatorname{tr}}
\newcommand{\st}{\;\: : \;\:}         

\newcommand{\Cay}{\mathrm{Cay}}
\newcommand{\length}{\mathrm{length}}

\theoremstyle{plain}

\newcommand{\mt}[9]{\left(\begin{array}{ccc}
#1 &  #2 & #3 \\
#4  & #5 & #6 \\
#7  &  #8  & #9
\end{array}\right)}
\newcommand{\mat}[9]{\left(\begin{array}{ccc}
#1 &  #2 & #3 \\
#4  & #5 & #6 \\
#7  &  #8  & #9
\end{array}\right)}

\newcommand{\Umat}[1]{\left(\begin{array}{cc} 1 & \begin{array}{cc} 0 & 0 \end{array}\\
#1 & \begin{array}{cc}1 & 0 \\
     0 & 1
     \end{array}
  \end{array}
  \right)}

\newcommand{\LRMat}[1]{\left(\begin{array}{cc} 1 & \begin{array}{cc} 0 & 0 \end{array}\\
\begin{array}{c} 0 \\ 0\end{array} & #1\end{array}\right)}

\newcommand{\rank}{\mathrm{rank}}

\newcommand{\la}{\langle}
\newcommand{\ra}{\rangle}

\newtheorem{theorem}{Theorem}[section]
\newtheorem{lemma}[theorem]{Lemma}
\newtheorem{cor}[theorem]{Corollary}
\newtheorem{proposition}[theorem]{Proposition}
\theoremstyle{definition}
\newtheorem{definition}[theorem]{Definition}

\theoremstyle{remark}
\newtheorem{remark}[theorem]{Remark}
\newtheorem{remarks}[theorem]{Remarks}

\newtheorem{convention}[theorem]{Convention}

\newcommand{\mathcald}{{\mathcal D}}

\newcommand{\calgg}{{\mathcal G}}

\newcommand{\bbi}{\partial_\infty}
\newcommand{\Con}{{\mathrm{Con}}}
\newcommand{\lm}{{\lim}}
\newcommand{\dist}{{\mathrm{dist}}}

\newcommand{\Hbo}{\mathrm{Hbo}}

\newcommand{\dv}{{\mathrm{div}}}
\newcommand{\Dv}{{\mathrm{Div}}}
\newcommand{\Ball}{{\mathrm{B}}}

\newcommand{\cgot}{\mathfrak{c}}
\newcommand {\ft}{\mathfrak t}

\newcommand{\pp}{{\mathcal P}}

\newcommand{\nn}{{\mathcal N}}
\newcommand{\onn}{\overline{{\mathcal N}}}

\newcommand{\ww}{{\mathcal W}}

\newcommand{\Hb}{{\mathrm{Hb}}}
\def\calr{\mathcal{R}}   
\def\calc{\mathcal{C}}   

\long\def\blue#1{\textcolor {blue}{#1}}

\long\def\red#1{\textcolor {black}{#1}}

\newcommand {\N}{\mathbb{N}} 
\newcommand{\Out}{\mathrm{Out}}

\newcommand {\Z}{\mathbb{Z}}            
\newcommand {\free}{\mathbb{F}} 
\newcommand {\q}{\mathfrak q} 
\newcommand {\g}{{\mathfrak g}} 
\newcommand {\pgot}{{\mathfrak p}}

\newcommand {\qgot} {{\mathfrak q}}

\newcommand {\me}{\medskip}

\newcommand {\iv}{^{-1}}
\newcommand{\lio}[1]{\lm^\omega\left(#1\right)}
\newcommand{\co}[1]{\Con^\omega\left( #1 \right)}
\newcommand{\arct}[1]{\ft_{#1}}
\newcommand{\proj}{{\hbox{proj}}}
\newcommand{\ttt}{{\mathcal T}}

\newcommand {\fn}{\footnote}
\newcommand {\Notat}{\noindent {\it{Notation}}:} 
\begin{document}

\title{Divergence in lattices in semisimple Lie groups and graphs of groups}
\author{Cornelia Dru\c{t}u, Shahar Mozes and Mark Sapir}\thanks{{The work of the first author was supported in part by
the ANR project ``Groupe de recherche de G\'eom\'etrie et Probabilit\'es dans
les Groupes''. The work of
the second author was supported by an ISF grant. The work of
the second and the third authors was supported by a BSF (US-Israeli)
grant. The work of the third author
was  supported in part by the NSF grant DMS 0700811.}}
\date{\today}

\begin{abstract} Divergence functions of a metric space estimate the length of a path connecting two points $A$, $B$ at distance $\le n$ avoiding a large enough ball around a third point $C$. We characterize groups with non-linear divergence functions as groups having cut-points in their asymptotic cones. That property is weaker than the property of having Morse (rank 1) quasi-geodesics. Using our characterization of Morse
quasi-geodesics, we give a new proof of the theorem of Farb-Kaimanovich-Masur {which} states that mapping class groups cannot contain copies of irreducible lattices in semi-simple Lie groups of higher ranks. {We also deduce} a generalization of the result of  Birman-Lubotzky-McCarthy about solvable subgroups of mapping class groups not covered by the Tits alternative of Ivanov and McCarthy.

We show that any group acting acylindrically on a simplicial tree or a locally compact hyperbolic graph always has ``many" periodic Morse quasi-geodesics (i.e. Morse elements), so its divergence functions are never linear. We also show that the same result holds in many cases when the hyperbolic graph  satisfies Bowditch's properties that are weaker than local compactness. This gives a new proof of Behrstock's result that every pseudo-Anosov element in a mapping class group is Morse.

On the other hand, we conjecture that lattices in semi-simple Lie groups of higher rank always have linear divergence. We prove it in the case when the $\Q$-rank is 1 and when the lattice is  $\SL_n(\OS)$ where $n\ge 3$, $S$ is a finite
set of valuations of a number field $K$ including all infinite valuations, and $\OS$ is the corresponding ring of $S$-integers.
\end{abstract}

\maketitle \tableofcontents

\section{Introduction}

\subsection{The divergence}

Roughly speaking, the divergence of a pair of points $(a,b)$ in a metric space $X$ relative to a point
$c\not\in \{a,b\}$ is the length of the shortest path from $a$ to $b$ avoiding a ball around $c$ of
radius $\delta$ times the distance $\dist(c,\{a,b\})$ {minus $\gamma $} for some $\delta\in (0,1)$ {and $\gamma \ge 0$} fixed in advance. If
no such path exists, then we define the divergence to be infinity. The divergence of a pair $(a, b)$ is
the supremum of the divergences of $a,b$ relative to all $c\in X$. The divergence function
$\Dv_{{\gamma }}(n;\delta)$ is the maximum of all divergencies of pairs $(a,b)$ with $\dist(a,b)\le n$. As usual
with asymptotic invariants of metric spaces and groups, we consider divergence functions up to the
natural equivalence: $$f\equiv g \quad \hbox{   if    }\quad \frac1C\, g\left( \frac nC \right)-Cn-C
\le f(n) \le Cg(Cn)+Cn+C$$ for some $C>1$ and all $n$ (a similar equivalence is used for functions in
more than one variable). Then the divergence function is {a quasi-isometry invariant of a
metric space, under some mild conditions on the metric space (Lemma \ref{qi} and Corollary \ref{cdiv})}.

For example, the divergence function of the plane $\RR^2$ or the Cayley graph of $\ZZ^2$ is linear (for
every $\delta$), the divergence function of a tree or of a group with infinitely many ends is infinity
for all $n>0$, the divergence function of any hyperbolic group is at least exponential
\cite{Gromov:hyperbolic, Alonso}, the divergence function of any mapping class group is at least
quadratic \cite{Behrstock:asymptotic}. There are in fact several other definitions of divergence in the
literature: one can restrict the choice of $c$ in a different way. For example, one can only consider
the case when $c$ is on a geodesic $[a,b]$. The choice of $c$ can be further restricted by assuming
that it is the midpoint of a geodesic $[a,b]$. The divergence of pairs of rays can also be defined
\cite{Alonso}, and so on. S. Gersten (\cite{Gersten:divergence}, \cite{Gersten:divergence3}) used a
version of the divergence function to study Haken manifolds. It is proved in\cite{Gersten:divergence},
\cite{Gersten:divergence3} and \cite{KapovichLeeb} that for every fundamental group of a Haken
3-manifold the divergence is linear, quadratic or exponential, the quadratic divergence occurring
precisely for graph manifolds and exponential divergence for manifolds with at least one hyperbolic
geometric component. It is also proved in \cite{Gersten:divergence} that the divergence function of a
semidirect product $H \rtimes \Z$ is at most the distorsion function of $H$ in $H \rtimes \Z$. The
equality can be strict as for instance in the Heisenberg group $\Z^2 \rtimes \Z$ which has linear
divergence while $\Z^2$ is quadratically distorted.

In Section \ref{sec3}, we show that different definitions of divergence give equivalent functions, so
we can speak of the divergence of a metric space.

\subsection{Super-linear divergence and cut-points in asymptotic cones}

For applications (see below), it is important to distinguish cases when the divergence is/is not
linear. Since the (equivalence class of) divergence function is a quasi-isometry invariant, we can say
that, for example,  a group has linear or superlinear divergence without specifying a generating set.
Note that there are finitely generated groups whose divergence is not linear but is arbitrarily close
to being linear (and in fact is bounded by a linear function on arbitrary long intervals) \cite{OOS}.

In Section 3, we characterize groups with superlinear divergence as groups whose asymptotic cones have
global cut-points. In particular we prove

\begin{proposition}[See Lemma \ref{lm02} {and Corollary \ref{cdiv}}] All asymptotic cones of a finitely generated group $G$ have no cut-points if and only if the divergence
function $\Dv_{{2}}\left( n,{\frac{1}{2} }\right)$ is linear.
\end{proposition}

We also characterize groups one of whose asymptotic cones has cut-points. Note that there are groups
with some asymptotic cones having cut-points and some asymptotic cones having no cut-points {\cite{OOS}}.

The importance of having cut-points in asymptotic cones has been shown in \cite{DrutuSapir:TreeGraded}
and \cite{DrutuSapir:splitting}. In particular it is proved in \cite{DrutuSapir:TreeGraded} that if a
non-virtually cyclic finitely generated group $G$ has cut-points in one of its asymptotic cones, then a
direct power of $G$ contains a free non-Abelian subgroup. Hence $G$ does not satisfy a non-trivial law
(cannot be bounded torsion or solvable, etc.). Also if one asymptotic cone of $G$ has cut-points, then
$G$ cannot have an infinite cyclic subgroup in its center, unless $G$ is virtually cyclic
\cite{DrutuSapir:TreeGraded}. Note that $G$ can still have infinite locally finite center \cite{OOS}.

Subgroups of groups with cut-points in all their asymptotic cones display some form of rigidity: if a
subgroup $H$ has infinitely many homomorphisms into $G$ that are pairwise non-conjugate in $G$, then
$H$ acts non-trivially on an asymptotic cone of $G$ which in many cases implies that $H$ acts
non-trivially on a simplicial tree \cite{DrutuSapir:splitting}, and so $H$ splits non-trivially into an
amalgamated product or an HNN extension.

\subsection{Morse quasi-geodesics}

A formally stronger property than superlinear divergence is the existence of the so called {\em Morse
quasi-geodesics} (also called {\em stable quasi-geodesics} or {\em rank 1 quasi-geodesic}). A bi-infinite
quasi-geodesic $\q$ in $X$ is called {\em Morse} if every $(L,C)$-quasi-geodesic with endpoints on $\q$
is at bounded distance from $\q$ (the bound depends only on $L,C$). Proposition \ref{prop3} below
provides several other equivalent definitions of Morse quasi-geodesics (\blue{see the corrected version in Section \ref{s:err}}). It turns out that a
quasi-geodesic $\q$ is Morse if and only if every point $x$ on the limit of $\q$ in every asymptotic
cone $\calc$ separates the two halves (before the point and after the point) of the limit, i.e. the two
halves are in different connected components of $\calc\setminus \{x\}$. This implies in particular that
every point on the limit of $\q$ is a cut-point of $\calc$. The converse implication (i.e. existence of
cut-points in every asymptotic cone implies existence of a Morse quasi-geodesic) is not known in
general (but it does hold in particular cases, for instance for non-positively curved manifolds, see
following paragraph). We show (Proposition \ref{lm3}) that if an asymptotic cone $\calc$ of $G$ has
cut-points then some asymptotic cone $\calc'$ of $G$ has a non-trivial geodesic with the above property
(we call it a {\em transverse geodesic}). By Remark \ref{Osin}, if the Continuum Hypothesis is true,
one can take $\calc'=\calc$.

In a Gromov hyperbolic space, every bi-infinite quasi-geodesic is Morse, a property which is crucial in
the proof of Mostow rigidity in the rank 1 case. A similar property is true for relatively hyperbolic
spaces \cite{DrutuSapir:TreeGraded}.

Suppose that a finitely generated group $G$ acts on a space $X$ by isometries. Recall that a
quasi-geodesic $\q$ in $X$ is called {\em periodic} if $h\cdot \q=\q$ for some $h\in G$, and $\q/\la
h\ra$ is bounded. The most common source of such quasi-geodesics are just orbits $\{h^n\cdot x, n\in
\ZZ\}$ of $h\in G$ in $X$. Note that if an orbit of $h$ in $X$ is quasi-geodesic, then the sequence
$\{h^n, n\in \ZZ\}$ is also a periodic bi-infinite quasi-geodesic in (any Cayley graph of) $G$. W.
Ballmann \cite{Ball} (see also Kapovich-Kleiner-Leeb \cite[Proposition 4.5]{KKL:QI}) proved that in a
locally compact complete CAT(0)-space $X$ with a co-compact group action, a periodic quasi-geodesic is
Morse if an only if it does not bound a half-plane. Moreover \cite{KKL:QI} if $X$ is a non-flat de Rham
irreducible manifold on which a discrete group acts cocompactly by isometries, then $X$ is either a
symmetric space of non-compact type and rank at least two or $X$ contains a periodic Morse geodesic.
Consequently in this particular case existence of cut-points in asymptotic cones implies existence of
Morse quasi-geodesics.

Existence of Morse geodesics in finite dimensional locally compact CAT(0)-spaces with co-compact group
action implies existence of free non-cyclic subgroups in the group \cite{Ball}. One cannot drop the
CAT(0) assumption in this statement because by Olshanskii-Osin-Sapir \cite{OOS}, there exist Tarski
monsters (non-virtually cyclic groups where all proper subgroups are cyclic), where every non-trivial
cyclic subgroup is a Morse quasi-geodesic.

\subsection{Morse elements}

We shall call elements $g\in G$ such that $\{g^n, n\in \ZZ\}$ is a Morse quasi-geodesic {\em Morse
elements} (these elements are also sometimes called {\em elements of rank 1}). For example every
element of infinite order in a hyperbolic group is Morse \cite{Gromov:hyperbolic}. In relatively
hyperbolic groups, every element of infinite order that is not in a parabolic subgroup is Morse
\cite{DrutuSapir:TreeGraded}. In the mapping class group $\hbox{MCG}(S_g)$, every pseudo-Anosov element
is Morse \cite{Behrstock:asymptotic}. The fundamental group of a compact irreducible manifold of
non-positive sectional curvature either is a lattice in a higher rank semi-simple Lie group or it
contains a Morse element (\cite{Ball}, \cite{KKL:QI}). This dichotomy has been extended to fundamental
groups of locally CAT(0) spaces (complexes), with lattices of isometries of Euclidean buildings added
to the list of possibilities (\cite{BallBrin1},\cite{BallBrin2},\cite{BallBu}).

Existence of Morse elements immediately implies some algebraic properties of the group $G$. In
particular, suppose that a finitely generated subgroup $H\le G$ contains a Morse element $g$. Then, by
Lemma \ref{lmeasy},  in the word metric of $H$, $\{g^n, n\in \ZZ\}$ is also a Morse quasi-geodesic. In
particular, all the asymptotic cones of the group $H$ (considered as a metric space with its own word
metric) must have cut-points. Thus we have the following statement

\begin{proposition}\label{prop001} [See Proposition \ref{propsub} below] If $H<G$ does not have its own Morse elements (say, $H$ is torsion or satisfies a non-trivial law or, more generally, does not have cut-points in its asymptotic cones), then $H$ cannot contain any Morse elements of $G$.
\end{proposition}

This proposition has applications to subgroups of relatively hyperbolic groups and mapping class groups
(see below).

\subsection{The main results of the paper}

Proposition \ref{prop001} shows that it is useful to study both the class of groups with Morse elements
and the class of groups without Morse elements. In Section \ref{sec4}, we study the former and in
Sections \ref{sec5}, \ref{sec6}, we study the latter.

Existence of Morse elements in non-trivial free products of groups follows (in particular) from the
fact that free products are hyperbolic relative to their free factors and from
\cite{DrutuSapir:TreeGraded}. We generalize this fact by proving the following theorem (the terms used
in the theorem are explained after the formulation).

\begin{theorem} [See Theorems \ref{th31}, \ref{th32}]\label{th002}
Let $X$ be a simplicial tree or a {uniformly} locally finite hyperbolic graph, and let $G$ be any finitely generated
group acting on $X$ acylindrically. Then any loxodromic element of $G$ is Morse in $G$.
\end{theorem}

Recall that an action of $G$ on $X$ is {\em acylindrical} if for some $l>0$ the stabilizers in $G$ of
pairs of points  in $X$ at distance $\ge l$ are finite of uniformly bounded sizes (in this case we say
that the action is $l$-acylindrical).

It is known \cite{Bowditch:tightgeod} that if $X$ is a tree or a locally finite hyperbolic graph
then {for every isometry $\alpha$ of $X$} some power $\alpha^k$ of $\alpha$ either fixes a point in $X$ (i.e. $\alpha$ is {\em elliptic}) or stabilizes a bi-infinite
geodesic $\pgot$ and $\la \alpha^k \ra$ acts on $\pgot$ co-compactly (i.e. $\alpha$ is {\em loxodromic}).
If a group $G$ {acts} on $X$ {by isometries and every element of $G$ is elliptic, then $G$ has a bounded orbit}.

In Section \ref{sec43}, we generalize Theorem \ref{th002} to groups acting on the so called {\em
Bowditch graphs}; our result implies the theorem of Behrstock stating that every pseudo-Anosov element
in a mapping class group of a surface is Morse. The notion of Bowditch graph (see a formal definition
in Section \ref{sec43}) is in a sense an abstract version of the curve complex of a surface. We
postulate existence of a set of {\em tight geodesics} which is invariant under $G$. The set should be
large enough so that every two vertices in $X$ are connected by a tight geodesic. On the other hand the
set of tight geodesics should be small enough so that, for example, for every pair of points $a, b\in
X$ and every point $c$ on a tight geodesic $[a,b]$, far enough from $a,b$, every ball $\Ball(c,r)$,
contains only finitely many points from tight geodesics connecting $a$ and $b$. Note that every
simplicial tree is a Bowditch graph where every finite geodesic is considered tight.


In Sections \ref{sec5} and \ref{sec6}, we study lattices in higher rank semi-simple Lie groups. We
conjecture that every such lattice has linear divergence. The conjecture is true for uniform lattices
because an asymptotic cone of a uniform lattice $\Gamma$ in a Lie group $L$ is bi-Lipschitz equivalent
to an asymptotic cone of $L$. If $L$ is semi-simple of higher rank and non-compact type, then every
asymptotic cone of $L$ is a Euclidean building of rank $\ge 2$ \cite{KleinerLeeb:buildings} and every
two points in the asymptotic cone belong to a 2-dimensional flat. Hence there are no cut-points in the
asymptotic cones of $L$. For non-uniform lattices the question is still open. We prove

\begin{theorem}[See Corollary \ref{c1} and Theorem \ref{thm:SLd:NoCutPt}]\label{th003}
Let $\Gamma$ be an irreducible lattice in a semi-simple Lie group of $\RR$-rank  $\ge 2$. Suppose that
$\Gamma$ is either of $\QQ$-rank 1 or {is} of the form $\SL_n(\OS)$ where $n\ge 3$, $S$ is a finite
set of valuations of a number field $K$ including all infinite valuations, and $\OS$ is the
corresponding ring of $S$-integers. Then $\Gamma$ has linear divergence.
\end{theorem}

In the proof of Theorem \ref{th003}, we heavily use the theorem of Lubotzky-Mozes-Raghunathan
\cite{LMR:Cyclic, LMR:metrics} which says that the word metric on an irreducible lattice in a
semi-simple Lie group of higher rank is quasi-isometric to the restriction of any left-invariant
Riemannian metric of the Lie group itself. In the case of $\QQ$-rank 1 we use the structure of
asymptotic cones. In that case every asymptotic cone is a product of Euclidean buildings
\cite{KleinerLeeb:buildings}. We use results from Dru\c tu \cite{Drutu:Nondistorsion,
Drutu:Remplissage, Drutu:Filling} and results about buildings from Kleiner-Leeb
\cite{KleinerLeeb:buildings}. In the case of $\SL_n(\OS)$ we explicitly construct a path connecting two
given matrices and avoiding a ball centered at a third matrix.

\subsection{Applications to subgroups of mapping class groups}

One of the most important results about subgroups of the mapping class groups is the Tits alternative
proved by McCarthy \cite{McCarthy:Tits} and Ivanov \cite{Ivanov:algebraic}: {\em every subgroup of a
mapping class group $\MCG (S)$ either contains a free non-Abelian subgroup or it contains a free
Abelian subgroup of rank at most $\xi (S)$ and of index at most $N=N(S)$.} Thus every subgroup of a
mapping class group not containing a free non-Abelian subgroup is virtually Abelian. This improved a
previous result of Birman-Lubotzky-McCarthy \cite{BLM} that a solvable subgroup of a mapping class
group must be virtually Abelian.

Proposition \ref{prop001} and Theorem \ref{th002} immediately imply the following new version of Tits
alternative.

\begin{theorem}\label{versionTits}
If a group $H$ does not have Morse elements and is a subgroup of the mapping class group $\MCG(S)$
(where $S$ is a surface with possible punctures), then $H$ stabilizes a (finite) collection of pairwise
disjoint simple closed curves on $S$.
\end{theorem}

\proof Indeed, the fact that the curve complex of a surface $S$ is a Bowditch graph is proved in
\cite{Bowditch:tightgeod}. By Proposition \ref{prop001} and Theorem \ref{th002}, $H$ cannot contain any
loxodromic elements for the action of $\MCG(S)$ on the curve complex of $S$. Thus some power of every
element in $H$ must fix a curve on $S$. By Ivanov and McCarthy \cite{Ivanov:subgroups} then $H$ stabilizes a
collection of pairwise disjoint simple closed curves on $S$.
\endproof

Note that this theorem can be used to give a proof of the Birman-Lubotzky-McCarthy theorem \cite{BLM}
cited above. Indeed, non-virtually cyclic solvable groups do not have cut-points in their asymptotic
cones \cite{DrutuSapir:TreeGraded}, so if $H$ is a solvable non-virtually cyclic subgroup of $\MCG(S)$,
then it must stabilize a collection of pairwise disjoint simple closed curves on $S$. Hence up to a
finite index, it must fix a curve $\gamma$ on $S$. Then up to finite index, $H$ is a subgroup of the
direct product of the cyclic subgroup generated by the Dehn twist about $\gamma$ and the restriction of
$H$ onto $S\setminus \gamma$, we get a solvable subgroup of the mapping class group of a surface of
smaller complexity (the complexity of $S$ is, by definition, $3g+p-3$ where $g$ is the genus, $p$ is
the number of punctures). Using an induction on the complexity, we deduce that up to a finite index,
$H$ is in the subgroup generated by Dehn twists about a collection of pairwise disjoint simple closed
curves, so $H$ is virtually Abelian (in fact the rank of the free Abelian subgroup of $H$ does not
exceed $3g+p-3$, and the index does not exceed the maximal size of a finite subgroup of $\MCG(S)$).

Theorem \ref{versionTits} implies, in particular, the theorem of
Farb-Kaimanovich-Masur \cite{FM,KaimanovichMasur}: {\em the mapping class group
of a surface does not contain lattices of semi-simple Lie groups of higher ranks.}
Indeed for
irreducible non-uniform lattices this {is so} because such lattices
contain distorted cyclic subgroups \cite{LMR:Cyclic}, while the
mapping class groups do not, according to Farb, Lubotzky
and Minsky \cite{FLM}. On the other hand, uniform higher rank lattices do not
have cut-points in their asymptotic cones (see above).
By Theorem \ref{versionTits}, if such a lattice is a subgroup of a mapping class group, then it must stabilize a multi-curve.\fn{Note that initially in \cite{KaimanovichMasur} a similar conclusion was obtained as a corollary of a description of the Poisson boundaries of mapping class groups. Another proof, using quasi-morphisms, was found by Bestvina and Fujiwara in \cite{BestvinaFujiwara}.}

The rest of the proof follows \cite{FM}.  Suppose that
such a lattice $\Gamma$ stabilizes a curve on $S$.
Then a finite index subgroup of $\Gamma$
would have a homomorphism with infinite image into the mapping class group of a surface
of smaller complexity (the surface cut along the multi-curve). By Selberg's theorem, we can assume that
$\Gamma$ is torsion-free. By Margulis' normal subgroup theorem, the
homomorphism must be injective, and we can proceed by induction on
the complexity of $S$. {When the complexity is zero, i.e. when the surface is the sphere with three holes, the mapping class group is finite.}

\subsection{Applications to subgroups of relatively hyperbolic groups}

Similar results are true for relatively hyperbolic groups. {In this paper, when speaking about relatively hyperbolic groups we always mean strongly relatively hyperbolic groups, in the sense of \cite{Gromov:hyperbolic}.}

Recall that in such groups as well a Tits
alternative holds: a subgroup in a relatively hyperbolic group is either parabolic or it contains a
free non-Abelian subgroup \cite{Koubi:croissance}. Using the fact that every non-parabolic element in a
relatively hyperbolic group is Morse (\cite{DrutuSapir:TreeGraded}, \cite{Osin:RelHyp}), we immediately
deduce

\begin{theorem} \label{relhyp} If $H$ is an infinite finitely generated group without Morse elements, then every isomorphic copy of $H$ in a finitely generated (strongly) relatively hyperbolic group is inside a parabolic subgroup.
\end{theorem}

Note that this does not follow from the quasi-isometry rigidity result of \cite{DrutuSapir:TreeGraded}
(stating that any quasi-isometric embedding of a subgroup without cut-points in some asymptotic cones
stays in a tubular neighborhood of a parabolic subgroup) since we do not assume that the subgroup is
undistorted, and since non-existence of Morse elements is weaker than non-existence of cut-points in
some asymptotic cones \cite{OOS}.

\subsection{Potential applications to $\Out(F_n)$}

{Recently Yael Algom-Kfir \cite{A-K} proved that any fully irreducible element of the outer automorphism group $\Out(F_n)$ is Morse. Since the role of fully irreducible automorphisms in $\Out(F_n)$ is similar to the role of pseudo-Anosov elements in mapping class groups,  this can potentially imply that $\Out(F_n)$ satisfies the same restrictions on subgroups as mapping class groups above. Unfortunately, a sufficient analog of the Ivanov-McCarthy theorem for subgroups of the mapping class group without pseudo-Anosov elements is not known yet for subgroups of $\Out(F_n)$ without fully irreducible elements.}

{\bf Acknowledgement.} {We thank the MPIM in Bonn for its
hospitality during the summer of 2006, when part of this work was
done.}

We are grateful to Denis Osin for helpful discussions. {We are also grateful to Curt Kent for pointing out some inaccuracies in the previous
version of Section 3.1. }

\section{General preliminaries}

\subsection{Definitions and general results}\label{defgenr}

{Throughout the paper we work with various discrete versions of paths and arcs. Let $(X, \dist )$ be a metric space. A \emph{finite $C$-path} is a sequence of points $a_1,a_2,...,a_n$ in $X$ with $\dist(a_i,a_{i+1})\le C$. \emph{Infinite} and \emph{bi-infinite $C$-paths} are defined similarly, with sets of indices $\pm \N$, respectively $\Z$.}

{Let $L\geq 1$ and $C\geq 0$ be two constants. An $(L,C)$-{\em quasi-geodesic} is a map $\pgot\colon I\to X$, where $I$ is an interval of the real line, such
that
$$
\frac{1}{L}|s-t|-C\leq \dist (\pgot (s),\pgot (t))\leq
L|s-t|+C, \hbox{ for all } s,t\in I.
$$ For a quasi-geodesic $\pgot : [0,d ]\to X$ we call $d$ its \emph{quasi-length}. For a concatenation
of quasi-geodesics, its \emph{quasi-length} is the sum of the quasi-lengths of its components. If $I=[a,\infty)$ then $\pgot$ (or its image $\pgot (I)$) is called either $(L,C)$-\emph{quasi-geodesic ray} or \emph{infinite $(L,C)$-quasi-geodesic}. If $I=\mathbb{R}$ then  $\pgot $ (or its image) is called \emph{bi-infinite $(L,C)$-quasi-geodesic}. When the constants $L,C$ are irrelevant they are not mentioned.}

 {Quasi-geodesics may not be continuous, for instance $C$-paths may be (images of) quasi-geodesics. Since we tacitly identify finitely generated groups with sets of vertices in their Cayley graphs, we shall sometimes refer to sequences of elements in a group as ``quasi-geodesics'', meaning that they compose a $C$-path which is image of a quasi-geodesic.}

 {We call $(L,0)$-quasi-isometries (quasi-geodesics) $L$-\emph{bi-Lipschitz maps
(paths)}, or simply \emph{bi-Lipschitz maps (paths)}.}

 {Let $\omega$ be an ultrafilter. All ultrafilters we use are assumed to be non-principal. A statement $P(n)$ depending on $n\in \N$ \emph{holds $\omega$-almost surely} ($\omega$-a.s.)  if the set of $n$'s where $P(n)$ holds belongs to the ultrafilter. The $\omega$-\emph{limit}  $\lim_\omega x_n$ of a sequence of numbers $(x_n)$ is the number $x$ (possibly $\pm \infty$) such that for every neighborhood $U$ of $x$, $x_n$ is in $U$ $\omega$-almost surely.}

 {An \emph{asymptotic cone} $\Con^\omega(X,(o_n),(d_n))$ of a metric space $X$, corresponding to a non-principal ultrafilter $\omega$, a sequence of observation points $(o_n)$ in $X$, and a sequence of positive scaling constants $d_n$ such that $\lim_\omega d_n=\infty$, is the quotient space of the set of sequences $\Pi_b X=\{(x_n)\mid \lim_\omega\left( \dist(x_n,o_n)/d_n \right) \mbox{ finite }\}$ by the equivalence relation $(x_n)\equiv(y_n)$ if $\lim_\omega \left( \dist(x_n,y_n)/d_n \right)=0$. The equivalence classes composing the cone are denoted by $(x_n)^\omega$, and the distance function on the cone is defined
by $\dist\left((x_n)^\omega,(y_n)^\omega\right)=\lim_\omega\left( \dist(x_n,y_n)/d_n \right)$.} For details on asymptotic cones we refer the reader to \cite{DriesWilkie}, \cite{Gromov:Asymptotic},
\cite{Drutu:survey}.

For every sequence of subsets $A_n$ in a metric space $X$, and every asymptotic
cone $\calc=\Con^\omega(X,(o_n),(d_n))$, the $\omega$\emph{-limit} $\lio{A_n}$ is defined as the set of all
elements $(a_n)^\omega\in\calc$ where $a_n\in A_n$. It is not difficult to check that the
$\omega$-limit of any sequence of geodesics $[a_n,b_n]$ of $X$ is always either

\begin{itemize}
\item empty (if no point $(c_n)^\omega$, $c_n\in [a_n,b_n]$, is in $\calc$), or
\item a finite geodesic $[(a_n)^\omega, (b_n)^\omega]$ (if both $(a_n)^\omega$, $(b_n)^\omega$ are in $\calc$), or
\item a bi-infinite geodesic (if neither $(a_n)^\omega$ nor $(b_n)^\omega$ is in $\calc$, but some point $(c_n)^\omega$, $c_n\in [a_n,b_n]$, is in $\calc$), or
\item a geodesic ray (if exactly one of the points $(a_n)^\omega, (b_n)^\omega$ is in $\calc$).
\end{itemize}

A similar statement is true for quasi-geodesics. Only the $\omega$-limit of a sequence of (finite)
quasi-geodesics is either empty or a bi-Lipschitz embedded interval, ray or line.

In an asymptotic cone $\calc$, a geodesic which is equal to the limit of a sequence of geodesics in $X$
is called a \emph{limit geodesic}.

\begin{remark} \label{remlim1} Note that there exist examples of groups (see \cite{Drutu:RelHyp}) such that ``most"
geodesics in their asymptotic cones are not limit geodesics.
\end{remark}

For every subset $A$ in a metric space, denote by $\nn_\delta (A)$ the open $\delta$-tubular
neighborhood of $A$, i.e. the set of points $x$ satisfying $\dist(x,A) < \delta$; denote by
$\onn_\delta (A)$ the closed $\delta$-tubular neighborhood of $A$, i.e. the set of points $x$
satisfying $\dist(x,A)\leq \delta$.

\begin{convention}\label{cvttt}
In what follows, the setting is that of a geodesic metric space, and it is assumed that for some
$\lambda \geq 1$ and $\kappa \geq 0$ a collection $\ttt$ of $(\lambda , \kappa )$-quasi-geodesics is
chosen, so that:
\begin{enumerate}
    \item  any two points in the metric space are joined by at least
one quasi-geodesic in $\ttt$;
    \item any sub-quasi-geodesic of a quasi-geodesic in $\ttt$ is also in $\ttt$.
\end{enumerate}
\end{convention}

A concatenation of $k$ quasi-geodesics in $\ttt$ is called a $k$-\emph{piecewise $\ttt$ quasi-path.} If
a quasi-path $\pgot$ is obtained as the concatenation of finitely many quasi-geodesics
$\pgot_1,...\pgot_k$ then we write $\pgot = \pgot_1\sqcup \cdots \sqcup \pgot_k$. It is easily seen
that when $\kappa=0$ such a quasi-path is $\lambda$-Lipschitz. Therefore in this case we call $\pgot$ a
piecewise $\ttt$ \emph{path.}

\begin{lemma}\label{piecewiseg}
Suppose that $\kappa=0$, i.e. $\ttt$ consists of $\lambda$-bi-Lipschitz paths. Let $Y$ be a geodesic
metric space, let $B$ be a closed set in $Y$ and let $x, y$ be in the same connected component of $Y
\setminus B$. Then there exists a piecewise $\ttt$ path $\pgot$ in $Y$ connecting $x$ and $y$ and not
intersecting $B$.
\end{lemma}

\proof The set $Y \setminus B$ is open. Hence for every point $a$ is in $Y\setminus B$ there exists an
open ball $\Ball (a,\epsilon)$ around $a$ which is disjoint from $B$. Any two points inside
$\Ball (a,\epsilon/\lambda^2)$ are connected by a piecewise $\ttt$ path (with at most two pieces) inside
$\Ball (a,\epsilon)$. Therefore the set $Z_x$ of all points in $Y\setminus B$ reachable from $x$ via
piecewise $\ttt$ paths is open. The set $Z_x$ is also closed since for every point $z\in Y\setminus B$
outside $Z_x$, there exists an open ball $\Ball (z,\delta)$ around $z$ that does not intersect $B$; then the
ball $\Ball (z,\delta /\lambda^2)$ cannot intersect $Z_x$ for otherwise there would be a piecewise $\ttt$
path connecting $x$ and $z$. Since $Z_x$ is obviously connected, it coincides with the connected
component of $x$ in $Y\setminus B$. Hence $y\in Z_x$.
\endproof


\begin{lemma} \label{lemlimprelim}
Let $(\calc ,\dist )$ be a geodesic metric space, assume that $\ttt$ is a collection of geodesics (i.e.
$\lambda =1$ and $\kappa =0$) and let $\pgot=[a,b]\cup [b,c]$ be a piecewise $\ttt$ simple path in
$\calc$ joining points $a$ and $c$. For every $x,y\in \pgot$ denote by $\ell (x,y)$ the length of the
shortest sub-arc of $\pgot $ of endpoints $x,y$. Let $b'\in [a,b)$, $b''\in (b,c]$ and $C>0$ be such
that
$$
\dist (x,y)\geq \frac{1}{C} \ell (x,y)\mbox{ for every }x\in [a,b']\mbox{ and } y\in [b'',c]\, .
$$

If $d' \in [a,b']$ and $d''\in [b'',c]$ minimize the distance then $\pgot'=[a,d']\cup [d',d''] \cup
[d'', c]$ is a $C'$-bi-Lipschitz path, where $C'=\max (C,3)$ and $[d',d'']$ is a geodesic in $\ttt$
joining $d'$ and $d''$.
\end{lemma}

\proof Given $x,y\in \pgot'$ denote by $\ell' (x,y)$ the length of the shortest sub-arc of $\pgot' $ of
endpoints $x,y$. If $x\in [a,d']$ and $y\in [d'',c]$ then by hypothesis
$$
\dist (x,y)\geq \frac{1}{C} \ell (x,y) \geq \frac{1}{C} \ell' (x,y)\, .
$$

Hence it remains to study the case when one of the two points $x$ and $y$ is on $[d',d'']$. Assume it
is $y$. Likewise assume that $x\in [a,d']$ (the other case is similar).

If $\dist (x,y) < \dist (d',y)$ then $\dist (x,d'')< \dist (d',d'')$, contradicting the choice of
$d',d''$. Thus, $\dist (d',y)\leq \dist (x,y)$, hence $\dist (x,d')\leq 2\dist (x,y)$ and
$\ell'(x,y)=\dist (x,d')+\dist (d',y)\leq 3\dist (x,y)\leq C' \dist (x,y)$.
\endproof

\begin{lemma} \label{lemlim0}
Let $(\calc ,\dist )$ be a geodesic metric space, assume that $\ttt$ is a collection of geodesics, and
let $\pgot$ be a piecewise $\ttt$ simple path in $\calc$ joining points $A$ and $B$. Then for every
$\delta$ small enough there exists a constant $C$ and a piecewise $\ttt$ simple path $\pgot'$ at
Hausdorff distance at most $\delta$ from $\pgot$ and such that $\pgot'$ is a $C$-bi-Lipschitz path.
\end{lemma}

\proof


We denote the vertices of $\pgot$ in consecutive order by $v_0=A, v_1,...., v_k=B$, and by $[v_i,
v_{i+1}]$ the consecutive edges of $\pgot$. For any two points $x,y$ on $\pgot$ we denote by $\ell
(x,y)$ the length of the shortest sub-arc of $\pgot$ of endpoints $x,y$.

Let $\phi :\pgot \times \pgot \setminus \Delta \to \mathbb{R}_+$ be the map defined by $\phi (x,y)=
\frac{\dist (x,y)}{\ell (x,y)}$, where $\Delta = \{(x,x)\mid x\in \pgot \}$. Note that the maximal
value of $\phi $ is $1$. If the infimum of $\phi$ is $1/K>0$ then $\pgot $ is a $K$-bi-Lipschitz path.
We therefore assume that the minimal value of $\phi$ is zero.

As $\phi$ is a continuous function, if $x\in [v_i, v_{i+1}]$ and $y\in [v_j, v_{j+1}]$ with $\{v_i,
v_{i+1}\}\cap \{v_j, v_{j+1}\}=\emptyset$ then $\phi (x,y)\geq C_{ij}$ for some constant $C_{ij}>0$.
Hence there must exist some $i\in \{1,2,...,k-1\}$ such that the infimum of $\phi$ on $[v_{i-1},
v_{i}]\times [v_i, v_{i+1}]$ is zero. We now show how $\pgot$ can be slightly modified between
$v_{i-1}$ and $v_{i+1}$ so that between these two vertices the function $\phi$ has a positive infimum.
 In what follows we always assume that $x\in [v_{i-1},
v_{i}]$ and $y\in [v_i, v_{i+1}]$.

Let $\delta >0$ be such that the distance from every vertex $v_i$ to $\pgot\setminus \{ [v_{i-1},
v_{i}]\cup [v_i, v_{i+1}]\}$ is at least $2\delta $. Consider the distance from $[v_{i-1}, v_{i}]
\setminus \Ball (v_i,\delta )$ to $[v_i, v_{i+1}]$, the distance from $[v_{i-1}, v_{i}]$ to $[v_i,
v_{i+1}]\setminus \Ball (v_i,\delta )$, and let $\tau >0$ be the minimum between the two distances.

There exist $x_0\in [v_{i-1}, v_{i})$ and $y_0\in (v_i, v_{i+1}]$ such that $\dist
(x_0,y_0)=\frac{\tau}{2}$. Clearly both $x_0$ and $y_0$ are in $\Ball (v_i, \delta )$. Now let $x_1\in
[v_{i-1},x_0]$ and $y_1\in [y_0, v_{i+1}]$ be a pair of points minimizing the distance. As $\dist
(x_1,y_1)\leq \dist (x_0,y_0)=\frac{\tau}{2}$ it follows that both points are again in $\Ball (v_i, \delta
)$. Consider $\pgot'$ the piecewise $\ttt$ path obtained by replacing in $\pgot$ the sub-arc $[v_{i-1},
v_{i}]\cup [v_i, v_{i+1}]$ with $[v_{i-1}, x_1]\cup [x_1,y_1] \cup [y_1, v_{i+1}]$, where $[x_1,y_1]$
is in $\ttt$. The fact that $x_1,y_1\in \Ball (v_i, \delta )$ implies that $\pgot$ and $\pgot'$ are at
Hausdorff distance at most $\delta$. The fact that $\pgot'$ is also simple follows from the choice of
$\delta$.

Let $\phi':\pgot' \times \pgot' \setminus \Delta' \to \RR_+$ be the function similar to $\phi$ defined
with respect to $\pgot'$ (it is distinct from $\phi$ even on common points, as the arc-lengths
changed). Lemma \ref{lemlimprelim} implies that the infimum of $\phi'$ is positive between $v_{i-1}$
and $v_{i+1}$.

By eventually repeating the same modification in all vertices near which $\phi$ approaches the zero
value, we obtain in the end a $C$-bi-Lipschitz path piecewise $\ttt$, and at Hausdorff distance
$\delta$ from $\pgot$.

\endproof


\medskip

\begin{lemma} \label{lemlim1}
Let $X$ be a geodesic metric space, let $\lambda\geq 1,\, \kappa \ge 0$, and let $L$ be a collection of
$(\lambda , \kappa )$-quasi-geodesics in $X$ satisfying the conditions in Convention \ref{cvttt}.

Let $\calc=\Con^\omega(X,(o_n),(d_n))$ be an asymptotic cone of $X$, and let $L_\omega$ be the
collection of limits of sequences of quasi-geodesics in $L$.
\begin{enumerate}
    \item\label{quasi} Let $\pgot$ be a piecewise $L_\omega$ path in $\calc$ joining distinct points $A=(a_n)^\omega$ and
$B=(b_n)^\omega$. Then there exists $k$ and $D$ such that $\pgot = \lio{\pgot_n}$, where each $\pgot_n$
is a $k$-piecewise $L$ quasi-path joining $a_n$ and $b_n$, moreover each $\pgot_n$ is of quasi-length
$\leq D\dist(a_n,b_n)$.
    \item\label{geo} Assume that $L$ is a collection of geodesics, and let $\pgot$ be as in (\ref{quasi}),
moreover $\pgot$ a $C$-bi-Lipschitz path. Then there exists $C'=C'(C)$ and a natural number $k\ge 1$
such that $\pgot = \lio{\pgot_n}$, where each $\pgot_n$ is a $C'$-bi-Lipschitz path joining $a_n$ and
$b_n$ in $X$, moreover  $\pgot_n$ is a  $k$-piecewise $L$ path.
\end{enumerate}
\end{lemma}

\proof (\ref{quasi}) If $\pgot$ has $m$ edges in $L_\omega$ then a sequence $\pgot_n$ with at most $2m$
edges in $L$ can be easily constructed. The other properties of $\pgot_n$ follow immediately.

\medskip

(\ref{geo})\quad The path $\pgot$ can be written as a limit $\lio{\pgot_n}$ of piecewise $L$ paths
$\pgot_n$ with the same number $k$ of edges. We now modify $\pgot_n$ so that they become
$C'$-bi-Lipschitz paths joining $a_n$ and $b_n$ in $X$.

As before $\ell (x,y)$ denotes the length distance on $\pgot$ between two points $x,y$. Denote the
consecutive vertices of $\pgot$ by $v_0=A, v_1,..., v_k=B$, where $v_i=\left( v^{i}_n\right)^\omega $,
and denote by $[v^{i-1}_n, v^i_n]$ the geodesics in $L$ whose limits compose $\pgot$.

Assume that for some $i$, both $[v^{i-1}_n, v^i_n]$ and $[v^i_n, v^{i+1}_n]$ have lengths of order
$d_n$. Consider $[v^{i-1}_n, \bar{v}^i_n]\subset [v^{i-1}_n, v^i_n]$ and $[\tilde{v}^i_n,
v^{i+1}_n]\subset [v^i_n, v^{i+1}_n]$ maximal so that for any $x\in [v^{i-1}_n, \bar{v}^i_n]$ and $y\in
[\tilde{v}^i_n, v^{i+1}_n]$ we have $\dist (x,y)\geq \frac{1}{2C}\ell (x,y)$. By maximality we have
that $\dist (\bar{v}^i_n,\tilde{v}^i_n)= \frac{1}{2C}\ell (\bar{v}^i_n,\tilde{v}^i_n)$. This also
implies that $\bar{v}^i_n,\tilde{v}^i_n$ are the only points to realize the distance between
$[v^{i-1}_n, \bar{v}^i_n]$ and $[\tilde{v}^i_n, v^{i+1}_n]$, as any other pair of points on the two
sub-segments are at a larger $\ell$-distance, hence at a larger distance. The hypothesis that
$\lio{\pgot_n}$ is a $C$-bi-Lipschitz path also implies that $\bar{v}^i_n,\tilde{v}^i_n$ are at
distance $o(d_n)$ from $v^i_n$. We then modify $\pgot_n$ by replacing $[\bar{v}^i_n,v^i_n]\sqcup
[v^i_n, \tilde{v}^i_n]$ with a geodesic $[\bar{v}^i_n,\tilde{v}^i_n]$ in $L$.

Assume now that some edge $[v^i_n, v^{i+1}_n]$ has length $o(d_n)$. Then consider $[v^{i-1}_n,
\bar{v}^i_n]\subset [v^{i-1}_n, v^i_n]$ and $[\tilde{v}^{i+1}_n, v^{i+2}_n]\subset [v^{i+1}_n,
v^{i+2}_n]$ maximal so that for any $x\in [v^{i-1}_n, \bar{v}^i_n]$ and $y\in [\tilde{v}^{i+1}_n,
v^{i+2}_n]$ we have $\dist (x,y)\geq \frac{1}{2C}\ell (x,y)$. Then modify $\pgot_n$ by replacing
$[\bar{v}^i_n,v^i_n]\sqcup [v^i_n, v^{i+1}_n]\sqcup [v^{i+1}_n, \tilde{v}^{i+1}_n]$ with a geodesic
$[\bar{v}^i_n,\tilde{v}^{i+1}_n]$ in $L$.

The piecewise $L$ path $\pgot_n$ thus modified is \uass a $C'$-bi-Lipschitz path according to a slight
modification of the argument in Lemma \ref{lemlimprelim}, and clearly $\lio{\pgot_n}=\pgot$.\endproof


\section{Characterization of asymptotic cut{-}points and Morse geodesics}
\label{sec3}

In this section we shall give internal characterizations of spaces all (some) of whose asymptotic cones
have cut-points. The characterization is in terms of divergence functions. There are several possible
definitions of divergence. Each of them estimates the ``cost" of going from a point $a$ to a point $b$
while staying away from a ball around a point $c$. The difference between various definitions is in the
allowed position of $c$ (how close can $c$ be to $a$ or $b$ and whether $c$ belongs to a geodesic
$[a,b]$). We show that these definitions give equivalent functions, in particular  in the case of
Cayley graphs of finitely generated {one-ended} groups. We also show that for finitely generated one-ended groups
these functions are equivalent to {the} Gersten divergence function.

In the second part of the section, we characterize Morse geodesics in terms of divergence.


\subsection{Divergence and asymptotic cut-points}

Typically the spaces that we have in mind when defining divergence are finitely generated groups,
geodesic metric spaces $X$ quasi-isometric to finitely generated groups, and geodesic metric spaces $X$
such that the action of their group of isometries is co-bounded, that is the orbit of a ball under
$Isom (X)$ covers $X$ (for simplicity we call such spaces \emph{periodic}).

We consider the usual relation on the set of functions $\RR_+\to \RR_+\, $, $f\preceq_C g$ if
$$f(n) \le Cg(Cn)+Cn+C$$ for some $C>1$ and all $x$. This defines the known equivalence relation on the set of
functions $\RR_+\to \RR_+\, $, $f\equiv_C g$ if $f\preceq_C g$ and $g\preceq_C f$.

Most of the time we obliterate the constant $C$ from the index and simply write  $f\preceq g$ and
$f\equiv g$.

We do not distinguish equivalent functions in this paper, so, for instance, all linear functions
(including all constants) are equivalent.

\begin{definition}\label{divabc} Let $(X, \dist)$ be a
geodesic metric space (one can formulate a similar definition for arbitrary length spaces), and let
$0<\delta<1$ {and $\gamma \geq 0$}. Let $a,b,c\in X$ with $\dist(c,\{a,b\})=r>0$, where $\dist(c,\{a,b\})$ is the minimum of
$\dist(c,a)$ and $\dist(c,b)\, $. Define $\dv_{{\gamma }}(a,b,c;\delta)$ as the infimum of the lengths of paths
connecting $a, b$ and avoiding the ball $\Ball(c,\delta r {-\gamma })$ {(note that by definition a ball of non-positive radius is empty)}. If no such path exists, take
$\dv_{{\gamma }}(a,b,c;\delta)=\infty$.
\end{definition}

The behavior of the function $\dv_{{\gamma }}(a,b,c;\delta)$ with respect to quasi-isometry is easily checked.

\begin{lemma}\label{lqi}
Let $q : X \to Y$ be an $(L,C)$-quasi-isometry between two geodesic metric spaces. Then {for every $0< \delta < 1$ and $\gamma \geq 0$ there exists $\gamma_1= \gamma_1(\delta , \gamma , L,C )\geq \gamma$ such that} for any three
points $a,b,c$ in $X$,
\begin{equation}\label{qi}
\dv_{\red{\gamma }}(q(a),q(b),q(c);\delta)\geq \frac{1}{2L} \dv_{\red{\gamma_1 }}(a,b,c;\delta \red{/L^2}) -C\, .
\end{equation}
\end{lemma}

\me


\begin{definition}\label{Div}
The {\em divergence function} $\Dv_{\red{\gamma }}(n ,\delta)$ of the space $X$ is defined as the supremum of all
numbers $\dv_{\red{\gamma }}(a,b,c;\delta)$ with $\dist(a,b)\le n$.
\end{definition}

\me

Clearly if $\delta\le \delta'$ \red{and $\gamma \geq \gamma'$} then  $\Dv_{\red{\gamma }}(n;\delta)\le\Dv_{\red{\gamma' }}(n;\delta')$ for every $n$.

\me

\begin{lemma} \label{rem6}
If $X$ is one-ended, \red{proper, periodic, and every point is at distance less than $\kappa$ from a
bi-infinite bi-Lipschitz path then there exists $\delta_0$ such that for every $\gamma \geq 4 \kappa $} the function $\Dv_{\red{\gamma }}(n ,\delta_0)$ takes only finite values.

In particular this holds if $X$ is a Cayley graph of a finitely generated one-ended group, and one can
take \red{$\delta_0 = \frac12$ and $\kappa=\frac12$} in this case.
\end{lemma}

\proof \red{Since the space is periodic and proper we may assume that there exist finitely many bi-infinite bi-Lipschitz paths $\pgot_1,...,\pgot_n ,$ such that every point in $X$ is at distance less than $\kappa$ from a path $g\pgot_i$ with $g\in Isom (X)$ and $i\in \{1,2,...,n\}$. Let $C \geq 1$ be such that $\pgot_i$ is $C$-bi-Lipschitz for every $i\in \{1,2,...,n\}$. Let $\delta_0=\frac1{C^2+1}$, let $a$ be a point on a path $\pgot_a=g\pgot_i$ and $b$ a point on a path $\pgot_b=g'\pgot_j$, where $g,g'\in Isom (X)$ and $i,j\in \{1,2,...,n\}$.  We prove that the value $\dv_{\red{0}}(a,b,c;\delta_0)$ is always
finite.}

\red{Let
$r=\dist(c,\{a,b\})>0$.} We claim that one of the two connected components of $\pgot_a\setminus \{a\}$
does not intersect $\Ball(c,\delta_0 r)$. Indeed, otherwise there would be two points on $\pgot_a$ at
distance at most \red{$u=2\delta_0 r$} from each other in $X$ but at distance \red{$v > \frac{2}{C}(1-\delta_0 )r$} along $\pgot_a$. This contradicts the assumption that $\pgot_a$ is $C$-bi-Lipschitz
since \red{$v/u >\frac{1-\delta_0 }{C\delta_0 } = C$}. It remains to note that since the space is one-ended, every two points on $\pgot_a$
and $\pgot_b$ respectively that are far enough from $c$, are connected by a path outside
$\Ball(c,\delta_0 r)$.

Now we prove that for any $n$ \red{and any $\gamma \geq 4\kappa$} the value $\Dv_{\red{\gamma }}(n ,\delta_0)$ is finite. Take $a,b,c$ such that
$\dist(a,b)\le n$, and let $r=\dist (c, \{a,b\})$. If a geodesic $[a,b]$ does not intersect $\Ball (c,
\delta_0 r)$ then $\dv_{\red{\gamma }}(a,b,c;\delta_0)\leq n$.

Assume that a geodesic $[a,b]$ intersects $\Ball (c, \delta_0 r)$. Then $r\leq \delta_0 r +n$, hence $r\leq
\frac{n}{1-\delta_0 }$.

Without loss of generality we may assume that $a$ is in a fixed compact. Then $b$ and $c$ are in
tubular neighborhoods of this compact, which are other compacts. Each of these compacts is covered by
finitely many balls of radius \red{$\kappa$ and with center on some path $g\pgot_i$  with $g\in Isom (X)$ and $i\in \{1,2,...,n\}$.} Therefore, to finish
the proof, it suffices to \red{note that for $a,b,c$ fixed, all triples
$a',b',c'$ with $\dist (a,a'), \dist(b,b'), \dist (c,c')\le \kappa $ satisfy $$\dv_{\red{4\kappa }}(a',b',c';\delta_0)\leq \dv_0 (a,b,c;\delta_0 ) +2\kappa \, .$$}
\endproof

\begin{remark} \label{rem7}

\red{For a space as in Lemma \ref{rem6} and a fixed $\delta \in (0,1)$, we must require at least that $\gamma$ is larger than $\delta \kappa $,
 otherwise $\dv_{\red{\gamma }}(a,b,c;\delta)$ can be infinite for} $a,b$ arbitrarily far away from each
other. Indeed, consider a \red{Cayley graph $Y$ of a finitely generated one-ended group}, and construct a metric space
$X$ by attaching to each \red{vertex} $y$ in $Y$ a copy $T_y$ of the same finite simplicial \red{ non-trivial tree $T$ with basepoint $v$}. Define a metric on $X$ in the natural way: to get from a point $p\in T_x$ to a point
$q\in T_y$, one needs to first get from $p$ to $x$ inside $T_x$, then from $x$ to $y$ inside $Y$, then
from $y$ to $q$ inside $T_y$. The space $X$ obviously admits a co-compact isometric group action. It
has one end since $Y$ is one-ended and $T$ is finite. \red{Every point in $X$ is within distance $\leq \kappa$ of a bi-infinite geodesic, where $\kappa$ is the Hausdorff distance between $T$ and $\{v\}$. If we take a point $a$ in $T_x$ with $\dist(a,x)=\kappa$,
a point $b$ in $Y$ with $\dist(b,x)\geq \kappa$, $c=x$ and  $\gamma < \delta \kappa $ then there is no path connecting $a$ to $b$ in $X\setminus
\Ball(c,\delta r - \gamma )$.} So $\dv_{\red{\gamma }}(a,b,c;\delta)=\infty$.

There \red{is one} equivalent and equally natural \red{way} to
define divergence function $\Dv_{\red{\gamma }}$ for arbitrary
one-ended length spaces admitting co-compact isometry group actions.
\red{One} chooses $C>0$ and replaces paths in the definition of
$\dv_{\red{\gamma }}(a,b,c;\delta)$ by $C$-paths, as defined in
Section \ref{defgenr}. One can easily check that all the statements
in this section remain true for \red{this more general function}
(after some obvious modifications).
\end{remark}


\me


We now define a new divergence function, closer to the idea of divergence of rays that was inspiring
the notion, in particular closer to the Gersten divergence.

\begin{definition}\label{div}
Let  $\lambda\ge 2$. The {\em small divergence function}
$\dv_{\red{\gamma }}(n;\lambda,\delta)$ is defined as the supremum
of all numbers $\dv_{\red{\gamma }}(a,b,c;\delta)$ with  $\red{0\le
} \dist(a,b)\le n$ and
\begin{equation}\label{lambda}
\lambda\dist(c,\{a,b\})\ge \dist(a,b).
\end{equation}

\end{definition}

\me

The next lemma is obvious.

\begin{lemma}\label{Divdiv}
The following inequalities are true for every geodesic metric space $X$.

\begin{enumerate}
\item If $\delta\le \delta'$, then $\dv_{\red{\gamma }}(n;\lambda,\delta)\le \dv_{\red{\gamma }}(n;\lambda,\delta')$ for every $n,\lambda$ \red{and $\gamma$}.

\item If $\lambda\le \lambda'$, then $\dv_{\red{\gamma
}}(n;\lambda,\delta)\le \dv_{\red{\gamma }}(n;\lambda',\delta)$ for
every $n,\delta$ \red{and $\gamma$}.
\end{enumerate}
\end{lemma}

\me

\begin{proposition}
\begin{itemize}
    \item[(1)] The function $\dv_{\red{\gamma }} (n;\lambda , \delta )$ takes only finite values if and only if for every $n$ there exists $d_n$ such
that if for some $a,b,c$ satisfying (\ref{lambda}),
$\dv_{\red{\gamma }}(a,b,c;\delta)\geq d_n$ then  $\dist(a,b)\geq
n$.

\me
    \item[(2)] If $(X, \dist )$ is \red{a geodesic metric space} quasi-isometric to a \red{finitely generated} group then the existence of some
    $\delta \in (0,1)\red{\, ,\, \gamma \geq 0}$ and $\lambda\ge 2$
     such that the function $\dv_{\red{\gamma }}(\, \cdot\, ;\lambda ,\delta )$ takes only finite values is granted if
    and only if $X$ is one ended.
\end{itemize}
\end{proposition}

\proof The statement in (1) is an easy exercise.

In view of Lemma \ref{lqi} and of (1), it suffices to prove (2) for
$X=G$ finitely generated group with a word metric.

If $G$ has infinitely many ends then for any $0<\delta <1$\red{, $\gamma >0$} and $\lambda \ge 2$, there exist $a,b,c$
satisfying (\ref{lambda}) and with $r=\dist (c,\{a,b\})$ large enough so that
$\dv_{\red{\gamma }}(a,b,c;\delta)=\infty$.

Assume that $G$ is one-ended. Then all $\dv_{\red{\gamma }} (a,b,c;\delta)$ \red{with $\delta \in (0,1/2]$ and $\gamma \geq 2$} are finite.

We rewrite (1) as follows: for every $n$ there exists $d_n$ such that if $\dist (a,b)\in [0,n]$ then
$\dv_{\red{\gamma }}(a,b,c;\delta)\leq d_n$.

Assume then that $\dist (a,b)\in [0,n]$. Without loss of generality we may assume that $a=1$, hence
$b\in \Ball (1,n)$. Take an arbitrary vertex $c$ satisfying (\ref{lambda}), and take $r=\dist (c,\{a,b\})$.
If $\Ball (c, \delta r)$ does not intersect all geodesics joining $a,b$ then $\dv_{\red{\gamma }}(a,b,c;\delta )=\dist
(a,b)\le n$. If it does then $r\leq \delta r+ \dist (a,b)$, hence $r\le \frac{n}{1-\delta }$. Thus
there are finitely many possibilities for $b,c$, hence the supremum of $\dv_{\red{\gamma }}(a,b,c;\delta )$ over all
$a,b,c$ satisfying (\ref{lambda}) and $\dist (a,b)\le n$ is finite.
\endproof

\begin{remark}\label{imends}
When $G$ has infinitely many ends, one may consider another definition of the divergence function.
Indeed, in this case it is known by \red{Stallings'} theorem \cite[Theorems 4.A.6.5 and 5.A.9]{Stallings}
that the group splits as a free product with amalgamation over a finite subgroup. When $G$ is finitely presented \red{it} can be moreover written as the fundamental group of a finite graph of groups with
finite edge groups and one-ended vertex groups \cite{Dunwoody:acces}. Then one can take the divergence
function of $G$ as the supremum function of the divergence functions of the vertex groups, and thus
obtain a finite function.

Note that this new function will provide more information on a group with infinitely many ends, as it
will provide upper bound for divergences of factors. On the other hand it will no longer distinguish
between the group $\Z^2$ and $\Z^2 * \Z^2$, for instance. Thus, it is appropriate only when restricting
to the class of groups with infinitely many ends.
\end{remark}

\me

We need to introduce two more functions, further restricting the choice of $c$. We assume, as before,
that $X$ is a geodesic metric space. For every pair of points $a,b\in X$, we choose and fix a geodesic
$[a,b]$ joining them \red{such that if $x,y$ are points on a geodesic $[a,b]$ chosen to join $a,b$ the sub-geodesic $[x,y] \subseteq [a,b]$ is chosen for $x,y$.} The next \red{lemmas} will show that the definition we are about to give does not depend
on the choice of the geodesic $[a,b]$. We say that a point $c$ is {\em between} $a$ and $b$ if $c$ is
on the fixed geodesic segment $[a,b]$.

We define $\Dv'_{\red{\gamma }}(n;\delta)$ and $\dv'_{\red{\gamma }}(n;\lambda,\delta)$ as $\Dv_{\red{\gamma }}$ and $\dv_{\red{\gamma }}$ before, but restricting $c$
to the set of points between $a$ and $b$. Clearly  $\Dv_{\red{\gamma }}'(n;\delta)\le \Dv_{\red{\gamma }}(n;\delta)$ and
$\dv_{\red{\gamma }}'(n;\lambda,\delta)\le \dv_{\red{\gamma }}(n;\lambda,\delta)$ for every $\lambda,\delta$.

\begin{lemma} \label{shtrih} For every $a,b\in X$, every $\delta\in (0,1)$ and every $\lambda \geq  2$, we have

\begin{equation} \label{div2} \begin{array}{ll} \red{\sup\limits_{\stackrel{c\in [a,b]}{ \lambda\dist(c,\{a,b\})\ge\dist(a,b)}}} \dv_{\red{\gamma }}(a,b,c;\delta/3) &  \le \red{\sup\limits_{\lambda\dist(c,\{a,b\})\ge\dist(a,b)} \dv_{\red{\gamma }}}(a,b,c;\delta/3)\\
& \\
&  \le \red{\sup\limits_{\stackrel{c\in [a,b]}{2\lambda\dist(c,\{a,b\})\ge\dist(a,b)}} } \dv_{\red{\gamma }}(a,b,c;\delta)
+\dist(a,b).\end{array}
\end{equation}

 \end{lemma}

\proof The first inequality in (\ref{div2}) is obvious. We prove only the second inequality in
(\ref{div2}).

Let $c\in X$, $r=\dist(c,\{a,b\})$ and assume that (\ref{lambda}) is satisfied. Suppose first that the
distance between $c$ and the geodesic segment $[a,b]$ is at least $\frac\delta{3}r$. Then $[a,b]$
avoids the ball $\Ball\left(c, \frac\delta{3} r \red{-\gamma } \right)$, and so $\dv_{\red{\gamma }}(a,b,c;\delta/3)=\dist(a,b)$, and
the second inequality in (\ref{div2}) holds.

Suppose now that $\dist(c,[a,b])<\frac\delta{3} r$. Then there exists a point $c'$ on $[a,b]$ at
distance at most $\frac\delta{3} r$ from $c$. Note that the distance $r'$ from $c'$ to $\{a,b\}$ is at
least $\left( 1-\frac\delta{3} \right) r$. Since $\delta<1$, we have
$$\frac23\, \delta<\delta\left( 1-\frac\delta{3} \right).$$
Therefore the ball $B'=\Ball(c', \delta r'\red{-\gamma } )\supseteq \Ball(c', \delta(1-\frac\delta{3})r \red{-\gamma })$ contains the
ball $B=\Ball(c,\frac\delta{3}r \red{-\gamma } )$. Hence every path avoiding the ball $B'$ also avoids the ball $B$.
Hence $\dv_{\red{\gamma }}(a,b,c';\delta)\ge\dv_{\red{\gamma }}(a,b,c;\frac\delta{3})$. This gives the second inequality in
(\ref{div2}).\endproof

\begin{lemma} \label{shtrih2}
\begin{enumerate}
    \item[(a)] For every $a,b\in X$, every $\delta\in (0,1)$\red{, $\gamma \geq 0$} and every $\lambda \geq  2$, we have
\begin{equation}\label{div1}
\red{\sup }_{c\in [a,b]}\dv_{\red{\gamma }}(a,b,c;\delta/3)\le \red{\sup }_{c\in X} \dv_{\red{\gamma }}(a,b,c;\delta/3)\le \red{\sup }_{c\in [a,b]}
\dv_{\red{\gamma }}(a,b,c;\delta)+\dist(a,b);\end{equation}

\begin{equation}\label{div3} \dv_{\red{\gamma }}(n;\lambda,\delta)\le \Dv_{\red{\gamma }}(n;\delta),\; \forall \; \red{n,\lambda \geq 2 ,\delta \in (0,1), \gamma \geq 0  \, .}
\end{equation}

\bigskip

    \item[(b)] \red{ For every $\delta\in (0,1)$ and $\gamma \geq 0$ we have}
    \begin{equation}\label{div4} \Dv_{\red{\gamma }}\red{\left( n ; \frac{\delta }{3} \right)}\le \dv_{\red{\gamma }}(n;2,\delta)+\red{2n \, ,\, \forall \,  n}\, .
\end{equation}

\bigskip

\item[(c)] \red{Let $X$ be a space as in Lemma \ref{rem6}, and let $\delta_0$ and $\gamma_0 = 4\kappa $ be the constants provided by the proof of that lemma. For every $0<\delta'\le \delta \leq \delta_0$ and $\gamma' \ge \gamma \geq \gamma_0$, $\Dv_\gamma (n;\delta) \equiv \Dv_{\gamma'}(n;\delta')$.}
\end{enumerate}

 \end{lemma}

\proof \textbf{(a)}\quad The first inequality in (\ref{div1}) and inequality (\ref{div3}) are obvious.
The second inequality in (\ref{div1}) is proved exactly in the same manner in which it was proved for
(\ref{div2}).

\me

\textbf{(b)}\quad In order to prove (\ref{div4}) it suffices to prove a similar inequality for $\Dv_{\red{\gamma }}'$
and $\dv_{\red{\gamma }}'$ according to (\ref{div2}) and (\ref{div1}). Take $a,b$ with \red{$\dist(a,b)\leq n$}, and take $c\in
[a,b]$. Let $r=\dist (c,\{a,b\})$ and assume that $r=\dist(a,c)$. Let $b'\in [c,b]$ at distance $r$
from $c$.

Then $\dv_{\red{\gamma }} (a,b,c;\delta )\leq \dv_{\red{\gamma }} (a,b',c ; \delta )+n \red{\le \dv_{\red{\gamma }}(n;2,\delta)+n.}$

\me

\red{\textbf{(c)}\quad We prove that $\Dv_\gamma (n;\delta) \preceq \Dv_{\gamma'}(n;\delta')$. As in the proof of Lemma \ref{rem6} we may assume there exist $\pgot_1,...,\pgot_n ,$ bi-infinite $C$-bi-Lipschitz paths such that every point in $X$ is at distance $\leq \kappa$ from a path $g\pgot_i$ with $g\in Isom (X)$ and $i\in \{1,2,...,n\}$; we take $\delta_0=\frac1{C^2+1}$.}

\me

\red{Let $a,b \in X$ be such that $\dist (a,b)\leq n$, and for $c\in X$ let $r=\dist (c,\{ a,b\})=\dist (c,a)$. Assume that $B(c, \delta r-\gamma )$ intersects $[a,b]$ (otherwise $\dv_\gamma (a,b,c,;\delta )\leq n$) whence $r\leq \frac{n}{1-\delta }$.}

\red{The points $a,b$ are at distance $\leq \kappa $ from points $a'$ respectively $b'$ on paths $\pgot_a = g\pgot_i$,  $\pgot_b = g'\pgot_j$. Note that $\dist (a',c)\geq r-\kappa > \delta r - \gamma$ and the same for $\dist (b',c)$. An argument as in the proof of Lemma \ref{rem6} implies that one of the two connected components of $\pgot_a\setminus \{a'\}$
does not intersect $\Ball(c,\delta r -\gamma )$. We denote it by $\cgot_a$. Likewise we obtain $\cgot_b$ ray in $\pgot_b$. Let $r'$ be such that $\delta' r' -\gamma'=\delta r - \gamma$. There exists a point $a_1$ on $[a,a']\cup \cgot_a$ at distance $r'$ from $c$, and a point $b_1$ on $[b,b']\cup \cgot_b$ at distance $\ge r'$ from $c$. Moreover $\dist (a,a_1) \leq r+r'$ and $\dist (b,b_1)\leq r+n+ r'$, whence $\dist (a_1,b_1) \leq 2(r+r')+n \leq D n +D$. Then $\dv_\gamma (a,b,c;\delta ) \leq \dv_{\gamma'} (a_1,b_1,c;\delta' ) + 2\kappa + CDn+CD \leq \Dv_{\gamma'} (Dn+D; \delta') + 2\kappa + CDn+CD$.}
\endproof

Lemmas \ref{shtrih} and \ref{shtrih2} immediately imply.

\begin{cor}\label{cdiv}

\red{Let $X$ be a space as in Lemma \ref{rem6}, and $\delta_0$, $\gamma_0 = 4\kappa $ the constants provided by the proof of that lemma.}

\me

\begin{enumerate}
    \item The functions $\dv_{\red{\gamma }}'(n;\lambda,\delta)$ and $\Dv_{\red{\gamma }}'(n;\delta)$ \red{with $\delta \leq \delta_0,\, \gamma \geq \gamma_0\, ,$} do not depend (up to the equivalence
relation $\equiv$) on the choice of geodesics $[a,b]$ for every pair of points $a,b$.

\me

\item \red{For every $\delta \leq \delta_0$, $\gamma \ge \gamma_0$, and $\lambda \geq 2$}
\begin{equation}\label{eq:equiv}
\red{\Dv_\gamma (n; \delta ) \equiv \Dv_\gamma' (n; \delta ) \equiv \dv_\gamma (n; \lambda ,\delta )\equiv \dv_\gamma' (n; \lambda ,\delta )\, . }
\end{equation}

\red{Moreover all functions in (\ref{eq:equiv}) are independent of  $\delta \leq \delta_0$ and $\gamma \ge \gamma_0$ (up to the equivalence
relation $\equiv$).}

\me

    \item The function $\Dv_{\red{\gamma }}(n;\delta)$ is equivalent to
$\dv_{\red{\gamma }}'(n;2,\delta)$ as a function in $n$. Thus in order to estimate $\Dv_{\red{\gamma }}(n,\delta)$ \red{for $\delta\leq \delta_0$} it is enough to consider points $a,b,c$ where $c$ is the
midpoint of a (fixed) geodesic segment connecting $a$ and $b$.
\end{enumerate}
\end{cor}

\me

We now recall another definition of divergence function, due to S. Gersten
(\cite{Gersten:divergence},\cite{Gersten:divergence3}). Let $X$ be a geodesic metric space, let $x_0$
be a fixed point and let $\rho \in (0,1)$. Denote by $S_r$ the sphere of center $x_0$ and radius $r$
and by $C_r$ the complementary set of the open ball $\Ball (x_0,r)$.

For every $x,y\in S_{r}$ define $\red{\DG}_\rho (x,y)$ to be the shortest path distance in $C_{\rho r}$
between $x,y$. Then define $\red{\DG}_\rho (r)$ as the supremum of $\red{\DG}_\rho (x,y)$, over all $x,y\in
S_r$ that can be connected by a path in $C_{\rho r}$.

The collection of functions $\red{\Delta} =\{\red{\DG}_\rho \mid \rho \in (0,1)\}$ is called \emph{the
divergence} of $X$. Such a collection is considered up to the equivalence relation $\sim$ defined in
what follows. We write $\red{\Delta} \preceq \red{\Delta}'$ if and only if there exist $\epsilon , \epsilon'\in
(0,1)$ and $C\ge 1$ such that for every $\rho < \epsilon$ there exists $\rho' <\epsilon'$ such that
$\red{\DG}_\rho \preceq_C \red{\DG}'_{\rho'}$.

Then we define $\red{\Delta} \sim \red{\Delta}'$ by $\red{\Delta} \preceq \red{\Delta}'$ and $\red{\Delta}' \preceq \red{\Delta}$.

It can be easily checked that the collection of functions $\red{\Delta}$ up to the equivalence relation $
\sim$ is independent of the basepoint $x_0$, and it is a quasi-isometry invariant. Thus we may assume
that $x_0$ is an arbitrary fixed basepoint.

The relationship between Gersten divergence and the small divergence function is the following.

\begin{lemma}\label{divGer}
\begin{enumerate}
    \item Let $X$ be a geodesic metric space. Then for every $\rho \in (0,1)$, $\red{\DG}_\rho (n) \leq \dv_{\red{0}}(\pi n;2,\rho
    )$ for every $n$.
    \item\label{ger2} \red{Let $X$ be a space as in Lemma \ref{rem6}, and $\delta_0$, $\gamma_0 = 4\kappa $ the constants provided by the proof of that lemma.}
    Then for every \red{$0<\rho <\rho'\leq \delta_0$ and $\gamma \geq \gamma_0$}
    $$
\dv_{\red{\gamma}}' (n;2,\rho )\leq \sup_{\red{x\le n}}\red{\DG}_{\rho'} \left( \frac{x}{2}+O(1)\right) + n \red{+O(1)}\, .
    $$
\end{enumerate}
\end{lemma}

We leave the proof to the reader. Note that if $X$ is as in Lemma \ref{divGer}, (\ref{ger2}), and if
the collection of functions $\red{\Delta}$ is equivalent to a non-decreasing function $f$ (i.e. to the
collection $\red{\DG}'_{\rho'}=f$ for every $\rho'$) then the divergence function is equivalent to $f$. In
all the cases when Gersten divergence is computed this is precisely the case.


\me

\Notat\, Let $(d_n), (d_n')$ be two sequences of numbers, and let $\omega$ be an ultrafilter. We write
$d_n'=O_\omega(d_n)$ if for some constant $C>1$, $\frac1Cd_n'<d_n<Cd_n'$ $\omega$-a.s. for all $n$.

\medskip

Recall that a finitely generated group $G$ is called {\em wide} if none of its asymptotic cones has a
cut-point; it is called \emph{unconstricted} if one of its asymptotic cones does not have cut-points.
We say that a closed ball $\bar B=\overline{\Ball}(c,\delta)$ in a metric space $X$ {\em separates }
point $u$ from point $v$ if $u$ and $v$ are in different connected components of $X\setminus \bar B$.

\begin{lemma}\label{lm00}
Let $X$ be a geodesic metric space. Let $\omega$ be any ultrafilter,
and let $(d_n)$ be a sequence of positive numbers such that
$\lim_\omega d_n =\infty$. Let $\calc= {\Con^\omega(X, (o_n),
(d_n))}$, $A=(a_n)^\omega, B=(b_n)^\omega, C=(c_n)^\omega\in
\mathcal{C}$. Let $r=\dist(C, \{A,B\})$. The following conditions
are equivalent for any $0\le \delta<1$.

\begin{itemize}
\item[(i)] The closed ball $\overline{\Ball}(C,\delta)$ in $\calc$ separates $A$ from $B$.
\item[(ii)] For every $\delta'>\delta$ \red{and every (some) $\gamma \geq 0$} the limit $\lim_\omega \frac{\dv_{\red{\gamma }}(a_n,b_n,c_n;\frac{\delta'}{r})}{d_n}$ is $\infty$.
\end{itemize}
\end{lemma}

\proof Suppose that $\frac{\dv_{\red{\gamma }}(a_n,b_n,c_n;\frac{\delta'}{r})}{d_n}$ is bounded by some constant $M$
\uas. This means there exists ($\omega$-a.s.) a path $\pgot_n$ of length $O(d_n)$ connecting $a_n, b_n$
and avoiding the ball $\Ball \left( c_n,\frac{\delta'}{r}r_n \right)$ where
$r_n=\dist(c_n,\{a_n,b_n\}))$. Since $\lio{r_n/d_n}=r$, the $\omega$-limit of these balls in $\calc$ is
the closed ball $\overline{\Ball}(C,\delta')$. The limit $\pgot$ of the paths $\pgot_n$ in $\calc$
exists (because the length of  $\pgot_n$ is $O_\omega(d_n)$), connects $A$ and $B$ and avoids the open
ball $\Ball(C,\delta')$. Hence $\pgot$ avoids every closed ball $\overline{\Ball}(C,\delta)$ for
$\delta<\delta'$. Thus if there exists a closed ball $\overline{\Ball}(C,\delta)$ separating $A$ from
$B$, then for any $\delta'>\delta$ the set of numbers $\frac{\dv_{\red{\gamma }}(a_n,b_n,c_n;\frac{\delta'}{r})}{d_n}$
cannot be bounded \uas. So (i) implies (ii).

Now suppose that property (ii) holds and that $\overline{\Ball}(C,\delta)$ does not separate $A$ from
$B$. Then there exists a path $\pgot$ in $\calc$ connecting $A$ and $B$ and avoiding the ball
$\overline{\Ball}(C,\delta)$. By Lemma \ref{piecewiseg} we can  assume that $\pgot$ is a piece-wise
geodesic, avoiding $\overline{\Ball}(C,\delta')$ for some $\delta'>\delta$, and that $\pgot$ is an
$\omega$-limit of paths $\pgot_n$ in the space $X$. Then the lengths of $\pgot_n$ are $O_\omega(d_n)$
and these paths avoid ($\omega$-a.s.) balls $\Ball(c_n, \delta'' r_n)$ where
$r_n=\dist(c_n,\{a_n,b_n\}))$, $\delta<\delta''<\delta'$. Therefore the set of numbers
$\frac{\dv_{\red{\gamma }}(a_n,b_n,c_n;\delta'')}{d_n}$ is bounded, which contradicts (ii). \endproof

Taking $\delta=0$ in Lemma \ref{lm00}, we obtain the following

\begin{lemma}\label{lm1} Let $X, \omega, (d_n)$ be as in Lemma \ref{lm00}. Then the following conditions are equivalent:

\begin{itemize}
\item[(i)] The asymptotic cone  {$\Con^\omega(X, (o_n), (d_n))$} has a cut-point.

\item[(ii)] There exists a sequence of pairs of points $(a_n,b_n)$ in $X$ with $$\dist(a_n,b_n)=O_\omega(d_n)\mbox{ and }\; \frac{\dist(a_n,o_n)}{d_n}\, ,\, \frac{\dist(b_n,o_n)}{d_n}\mbox{ bounded,} $$
and a sequence of midpoints  $c_n$ of geodesics $[a_n,b_n]$ such that the sequence
    $\dv_{\red{\gamma }}(a_n,b_n,c_n;2,\delta)$, $n\ge 1$, is superlinear $\omega$-a.s. for every $\delta\in (0,1)$ \red{and $\gamma \geq 0$}.

Moreover, if that condition holds, then the point $C=(c_n)^\omega$ is a cut-point in
 {$\Con^\omega(X, (o_n),(d_n))$} separating $A=(a_n)^\omega$ from $B=(b_n)^\omega$.
\end{itemize}
\end{lemma}

\begin{lemma}\label{lm01}
Let $X$ be \red{a periodic geodesic metric space which contains} a bi-infinite quasi-geodesic. Suppose
that one of its asymptotic cones $\calc$ has a closed ball $\overline{\Ball}(C,\delta)$ separating $A$
from $B$ and $\dist(C, \{A,B\})>3\delta$. Then $X$ is not wide, that is one of the asymptotic cones of
$X$ has a cut-point.
\end{lemma}

\proof Let $\overline{\Ball}(C,\delta)$ be a cut-ball,
$C=(c_n)^\omega$,  and let $A=(a_n)^\omega, B=(b_n)^\omega$ be two
points in $\calc$ separated by $\overline{\Ball}(C,\delta)$. Since
$X$ contains a bi-infinite quasi-geodesic and it has a cobounded
action of $Isom (X)$, \red{the cone $\calc$} contains a bi-infinite
bi-Lipschitz line \red{and it} is a homogeneous space\red{.
Therefore} there exist bi-infinite bi-Lipschitz lines containing $A$
and $B$ respectively. We proceed as in Lemma \ref{rem6}. The point
$A$ cuts the bi-Lipschitz line containing it into two bi-Lipschitz
rays. Since $\dist(C, \{A,B\})>3\delta$, one of these rays does not
cross $\overline{\Ball}(C,\delta)$. Let us denote it by $\pgot$.
Similarly, there exists an infinite bi-Lipschitz ray $\pgot'$
starting at $B$ that does not cross $\overline{\Ball}(C,\delta)$.
Clearly every point $A'$ on $\pgot$ and every point $B'$ on $\pgot'$
are separated by $\overline{\Ball}(C,\delta)$. For every $n>1$ let
$A_n$ be a point on $\pgot$ at distance $n$ from $A$, $B_n$ be a
point on $\pgot'$ at distance $n$ from $B$. Note that
$\dist(A_n,B_n)=O(n)$. Indeed otherwise for a big enough $n$, any
geodesic $[A_n,B_n]$ avoids $\overline{\Ball}(C,\delta)$. Also note
that $\dist(C,\{A_n,B_n\})=O(n)$. Now consider the asymptotic cone
$\calc'=\Con^\omega(\calc, (C),(n))$ of the space $\calc$. Applying
Lemma \ref{lm00} to $\calc$, we get that  $(C)^\omega$ is a
cut-point in $\calc'$. Since $\calc'$ is an asymptotic cone of $X$
\cite{DrutuSapir:TreeGraded}, $X$ is not wide.
\endproof

\begin{lemma}\label{lm02}
\begin{itemize}
\item[(i)] \red{Let $X$ be a geodesic metric space. If there exists $\delta \in (0,1)$ and $\gamma \geq 0$ such that the function $\Dv_{\gamma }(n;\delta)$ is bounded by a linear function then $X$ is wide.}
\item[(ii)] \red{Let $X$ be a periodic geodesic metric space which contains a bi-infinite quasi-geodesic. If $X$ is wide then for every $0<\delta <\frac{1}{54}$ and every $\gamma \ge 0$, the function $\Dv_{\gamma }(n;\delta)$ is bounded by a linear function.}
\end{itemize}
\end{lemma}

\proof \red{\textbf{(i)}} \quad Suppose that $X$ is not wide, and
{$\calc=\Con^\omega(X,(o_n),(d_n))$} has a cut-point. Then $\calc$
has a closed ball $\overline{\Ball}(C,0)$ (where $C=(c_n)^\omega$),
separating a pair of points $A=(a_n)^\omega$ and $B=(b_n)^\omega$.
By Lemma \ref{lm00}, for every $\delta >0$ \red{and $\gamma \ge 0$}
the limit $\lim_\omega \frac{\dv_{\red{\gamma }}(a_n,b_n,c_n;\delta
)}{d_n}$ is $\infty $. Since \uass $\dist(a_n,b_n)<\kappa d_n$ for
some $\kappa$, $\Dv_{\red{\gamma }}(\kappa d_n,\delta)$ is not
bounded by any linear function in $d_n$.

\red{\textbf{(ii)}} \quad \red{Suppose that for some $\delta < \frac{1}{54}$ and $\gamma \ge 0$ the function $\Dv_{\red{\gamma }}(n,\delta )$ is superlinear. By (\ref{div2}) and  (\ref{div4})},
the same holds for \red{$\dv_{\gamma }' (n;4,9 \delta )$}, hence there exist $d_n\geq n$ such that $\red{\dv_{\gamma }' (d_n;4, 9\delta
)}\geq 2n d_n$. Consequently, there exists a sequence of triples of points $a_n,b_n, c_n$, where $c_n$
is \red{on} a geodesic $[a_n,b_n]$, \red{$\dist (c_n \{a_n,b_n\})\geq \frac{\dist(a_n,b_n)}{4}$}, such that $\dist(a_n,b_n)\leq d_n$ and
$\frac{\red{\dv_{\gamma }(a_n,b_n,c_n;9\delta )}}{d_n}$ is at least $n$. We can assume without loss of generality that
$\dist(a_n,b_n)=d_n$. In the asymptotic cone $\Con^\omega(X, (c_n), (d_n))$ of $X$, the distance
between $A=(a_n)^\omega$ and $B=(b_n)^\omega$ is $1$ and the ball $\bar B=\overline{\Ball}\left( C,\red{\frac{9\delta }{2}} \right)$
separates $A$ from $B$ (here $C=(c_n)^\omega$) by Lemma \ref{lm00}. By Lemma \ref{lm01}, $X$ is not
wide since \red{$\dist(C,\{A,B\})\ge \frac{1}{4}>3\cdot \frac{9\delta }{2}$}. The lemma is proved.
\endproof

\subsection{Morse quasi-geodesics}

Recall the definition of tree-graded spaces.

\begin{definition}(\cite{DrutuSapir:TreeGraded})\label{deftgr}
Let $\free$ be a complete geodesic metric space and let $\pp$ be a collection of closed geodesic
subsets, called {\it{pieces}}. Suppose that the following two properties are satisfied:
\begin{enumerate}

\item[($T_1$)] Every two different pieces have at most one point in
common.

\item[($T_2$)] Every simple non-trivial geodesic triangle in $\free$
is contained in one piece.
\end{enumerate}
Then we say that the space $\free$ is {\em tree-graded with respect to }$\pp$.

When there is no risk of confusion as to the set $\pp$, we simply say that $\free$ is
\emph{tree-graded}.
\end{definition}

 {The topological arcs starting in a given point and intersecting each piece in at most one point compose a real tree called \emph{transversal tree}. Some transversal trees may reduce to singletons.}

We shall need the following general facts.

\begin{lemma}\label{lm2}\cite[Theorem 3.30]{DrutuSapir:TreeGraded}
Let $(X_n,\pp_n)$ be a sequence of tree-graded spaces, $\omega$ be an ultrafilter. Let
$\free=\lio{X_n,o_n}$ be the $\omega$-limit of $X_n$ with observation points $o_n$. Let $\tilde\pp$ be
the set of $\omega$-limits $\lio{M_n}$ where $M_n\in \pp_n$ (the same ultralimit is counted only once;
the empty ultralimits, corresponding to $M_n\in \pp_n$ such that $\lio{\dist_n(o_n,M_n)}=\infty $, are
not counted). Then $\free$ is tree-graded with respect to $\tilde\pp$.
\end{lemma}

\begin{proposition}\label{lm3} Let $(X_n, \pp_n)$ be a sequence of homogeneous
unbounded tree-graded metric spaces with observation points $o_n$. Let $\omega$ be an ultrafilter. Then
the ultralimit $\lio{X_n, o_n}$ has a tree-graded structure with a non-trivial transversal tree at
every point.
\end{proposition}

\proof By \cite[Lemma 2.31]{DrutuSapir:TreeGraded}, we can first assume that pieces from $\pp_n$ do not
have cut\red{-}points. We can also add all transversal trees to the collection of pieces and assume that
transversal trees of $X_n$ are trivial (each of them consists of one point), that is we assume that  {every arc in $X_n$ intersects at least one piece in a non-trivial sub-arc}. By \cite[Lemma
2.15]{DrutuSapir:TreeGraded} every path-connected subspace of $X_n$ without cut-points is in one of the
pieces of $\pp_n$. Hence every isometry of $X_n$ permutes the pieces of $\pp_n$. If the $\omega$-limit
${\mathcal L}=\lio{X_n, o_n}$ is an $\RR$-tree, there is nothing to prove. So we can assume that one of
the maximal subsets of $\mathcal L$ without cut\red{-}points has non-zero (possibly infinite) diameter
$\tau$. In particular, \uass $X_n$ is not an $\RR$-tree.

We assume that $n>2$. Since $X_n$ is homogeneous and has cut\red{-}points, every point of $X_n$ is a cut
point. Pick a number $c>0$ smaller than the non-zero diameter of a subset of $\mathcal L$ without cut
points. We can assume that $\pp_n$ contains a piece of diameter $\ge c$ and without cut\red{-}points. By
homogeneity, there exists such a piece $M_{n,0}$ in $\pp_n$ containing $o_n$ and such that a geodesic
$\pgot_{n,1}=[o_n,x_{n,1}]$ of length $c/n$ is contained in $M_{n,0}$. Since $x_{n,1}$ is a cut\red{-}point,
$X\setminus \{x_{n,1}\}$ has at least two connected components, one of which contains $M_{n,0}\setminus
\{x_{n,1}\}$. Let $y$ be a point in another connected component at distance at most $c/n$ from
$x_{n,1}$. By homogeneity, there is an isometric copy $M_{n,2}$ of $M_{n,0}$ containing $y$. Let
$x_{n,2}$ be the projection of $x_{n,1}$ onto $M_{n,2}$. {Note that $\dist(x_{n,2}, x_{n,1})\le
c/n$.} We can assume (again by homogeneity) that there exists a geodesic $[x_{n,2}, x_{n,3}]$ in
$M_{n,2}$ of length $c/n$. By induction, we can construct a sequence of points $o_n=x_{n,0}, x_{n,1},
..., x_{n,n}$ such that:
\begin{itemize}
\item $\dist(x_{n,i},x_{n,i+1})\le c/n$,
\item $c\le\sum\dist(x_{n,i}, x_{n,i+1})\le 2c$,
\item If $i$ is even then $[x_{n,i},x_{n,i+1}]$ is contained in a
piece $M_{n,i}$ from $\pp_n$ without cut\red{-}points, all these pieces are isometric and different,
\item If $i$ is odd then $x_{n,i+1}$ is the projection of $x_{n,i}$
onto $M_{n,i+1}$.
\end{itemize}
By \cite[Lemma 2.28]{DrutuSapir:TreeGraded}, this implies that the union of geodesic segments
$\pgot_n=\bigsqcup_{i=0}^{n-1} [x_{n,i},x_{n,i+1}]$ is a geodesic intersecting every piece of $\pp_n$
in a sub-geodesic of length at most $c/n$.

Let $\pgot$ be the $\omega$-limit of $\pgot_n$. Then $\pgot$ is a geodesic in ${\mathcal L}$. Let us
show that $\pgot$ is a transversal geodesic. Suppose that $\pgot$ contains two points
{$A=(u_n)^\omega$, $B=(v_n)^\omega$} belonging to a piece $\lio{N_n}$ where $N_n\in\pp_n$, $u_n,
v_n\in \pgot_n$. Let $l=\dist(A,B)$. That means $\lio{\dist(u_n,N_n)}=\lio{\dist(v_n,N_n)}=0$ while
$\lio{\dist(u_n,v_n)}=l$. Let $u_n'$ (resp. $v_n'$) be the projections of $u_n$ (resp. $v_n$) onto
$N_n$. {These projections exist and are unique by \cite[Lemma 2.6]{DrutuSapir:TreeGraded}.
Applying the strong convexity of pieces in a tree-graded space \cite[Corollary
2.10]{DrutuSapir:TreeGraded}, we can conclude that any geodesic $[u_n',v_n']$ is in $N_n$. Moreover, a
union of three geodesics $[u_n,u_n']\sqcup [u_n',v_n']\sqcup [v_n',v_n]$ is a topological arc. This
follows from the fact that any two consecutive geodesics intersect only in their common endpoint, and
$[u_n,u_n']\cap [v_n,v_n']=\emptyset$ since these geodesics are at distance at least $l/2$ \uas.}

{According to \cite[Proposition 2.17]{DrutuSapir:TreeGraded} the projections $u_n',v_n'$ are on
$\pgot_n$, in between $u_n$ and $v_n$.} Thus \uass $[u_n,v_n]$ intersects $N_n$ in a sub-geodesic of
length at least $l/2$. But by our construction it intersects $N_n$ in a sub-geodesic of length at most
$c/n$ \uas, a contradiction.

Thus $\mathcal L$ contains a non-trivial transversal geodesic. By homogeneity, every transversal tree
of $\mathcal L$ is non-trivial.
\endproof

\begin{remark}\label{Osin}
As noticed by Denis Osin, if the Continuum Hypothesis is true, Proposition \ref{lm3} implies that every
tree-graded asymptotic cone of a finitely generated group has a non-trivial transversal tree at every
point. Indeed, by Corollary 5.5 in \cite{KSTT} every asymptotic cone of a finitely generated group is
isometric to its $\omega$-limit for every ultrafilter $\omega$.
\end{remark}

\begin{definition}\label{midth} Let $X$ be a metric space.
Given a quasi-geodesic $\q :[0,\ell ] \to X$ we call the \emph{middle third} of $\q$ its restriction
to $[\ell/3 , 2\ell/3]$.  {Note that when $\q$ is continuous and of finite length the middle third in the above sense does not in general coincide with the middle third in the arc-length sense.}
\end{definition}

\begin{definition}
A bi-infinite quasi-geodesic $\q$ in a geodesic metric space $(X,\dist)$ is called a \emph{Morse
quasi-geodesic} if for every $L\ge 1, C\ge 0$, every $(L,C)$-quasi-geodesic $\pgot$ with endpoints on
the image of $\q$ stays $M$-close to $\q \, $, where $M$ depends only on $L, C$.
\end{definition}

Note that the above property is equivalent to the fact that $\pgot$ is at uniformly bounded Hausdorff
distance from a sub-quasi-geodesic of $\q$ having the same endpoints.

\blue{For the correct formulation of the next Proposition and for the corrected proof, see Section \ref{s:err}.}

\begin{proposition}[Morse quasi-geodesics]\label{prop3}
Let $X$ be a metric space and for every pair of points $a,b\in X$ let $L(a,b)$ be a fixed set of
$(\lambda,\kappa)$-quasi-geodesics (for some constants $\lambda \geq 1$ and $\kappa \geq 0$) connecting $a$ to $b$.
Let $L=\bigcup_{a,b\in X} L(a,b)$.

{Let $\q$ be a bi-infinite quasi-geodesic $\q$ in $X$, and for every two points $x,y$ on $\q$
denote by $\q_{xy}$ the maximal sub-quasi-geodesic of $\q$ with endpoints $x$ and $y$.}

{The following conditions are equivalent for $\q \, $:}
\begin{itemize}
\item[(1)] In every asymptotic cone of $X$, the ultralimit of $\q$ is
either empty or contained in a transversal tree for some tree-graded structure;
\item[(2)] $\q$ is a Morse quasi-geodesic;
\item[(3)] For every $C\ge 1$ there exists $D\ge 0$ such that every path of length $\le Cn$ connecting two points $a,b$ on $\q$ at distance $\ge n$
crosses the $D$-neighborhood of the middle third of {$\q_{ab}$};
\item[(4)] For every $C\geq 1$ and natural  $k>0$ there exists $D\ge 0$ such that every
 $k$-piecewise $L$ quasi-path $\pgot$ that:
     \begin{itemize}
     \item connects two points $a, b\in \q$,
     \item has quasi-length $\leq C \dist(a,b)$,
     \end{itemize}
crosses the $D$-neighborhood of the middle third of {$\q_{ab}$}.
\item[(5)] for every $C\ge 1$ there exists $D\ge 0$ such that for every $a, b\in \q$,
and every path $\p$ of length $\le C\dist(a,b)$ connecting $a,b$, the {sub-quasi-geodesic
$\q_{ab}$} is contained in the $D$-neighborhood of $\p$.
\end{itemize}
\end{proposition}

\proof We  prove $(1)\Leftrightarrow (2)$ and $(1)\to (5)\to (3)\to (4)\to (1)$.

$\mathbf{(1)\to (2), (1)\to (5).}$ These two implications are proved in a similar manner, so we prove
only the {second one, leaving the first one to the reader.} Assume that in every asymptotic cone
$\calc$, $\lio \q$ is in the transversal tree. Also, by contradiction, assume that for some $C\geq 1$
and every natural $n\ge 1$ there exists a sequence $\pgot_n$ of paths with endpoints $a_n,b_n$ on $\q$
and lengths $\le C \dist (a_n,b_n)$ such that $\q_n=\qgot_{a_nb_n}$ is not in $\onn(\pgot_n,n)$. Then
there exists a point {$z_n\in \qgot_{n}$ at distance $\delta_n\ge n$ from $\pgot_n$. We choose
$z_n\in \q_n$ so that $\delta_n=\dist(z_n, \pgot_n)$ is maximal. Thus $\q_n$ is in
$\onn(\pgot_n,\delta_n)$.}

In the asymptotic cone $\Con^\omega(X,(z_n), (\delta_n))$, {the ultralimits $\q_\omega
=\lio{\q_n}$ and $\pgot_\omega=\lio{\pgot_n}$ are paths such that $\q_\omega$ is transversal, staying
$1$-close to $\pgot_\omega$ and containing a point $z=(z_n)^\omega$ at distance $1$ from
$\pgot_\omega$. Note that since $\pgot_n$ can be parameterized so that $\pgot_n :[0,\dist(a_n,b_n)]\to
X$ is $C$-Lipschitz, $\pgot_\omega$ can also be seen as a $C$-Lipschitz path. Both $\pgot_\omega$ and
$\q_\omega$ are either finite, or infinite or bi-infinite simultaneously, and with the same endpoints,
if any.}

{Assume that both $\pgot_\omega$ and $\q_\omega$ are finite. Possibly by diminishing both, we may
assume that they intersect only in their endpoints $a,b$. By \cite[Corollary
2.11]{DrutuSapir:TreeGraded} applied to the piece $M=T_z$ the transversal tree in $z$, $\pgot_\omega
\setminus\{a,b\}$ projects onto $T_z$ both in $a$ and in $b$. This contradicts the \red{uniqueness} of the
projection point also stated in \cite[Corollary 2.11]{DrutuSapir:TreeGraded}.}

\medskip

{Assume that both $\pgot_\omega$ and $\q_\omega$ are infinite. On the infinite branch $\q'$ of
$\q$ starting at $z$ consider a point $t$ at distance $10$ from $z$, and the sub-path $\q''$ of $\q'$
of endpoints $z$ and $t$. Let $t'$ be a nearest point to $t$ on $\pgot_\omega$, and consider a geodesic
$[t,t']$. Replacing $\q'$ by $\q''$ on $\q_\omega$ and the infinite branch on $\pgot_\omega$ starting
at $t'$ by $[t',t]$, and using for the thus modified paths the argument in the finite case, we obtain
that $z$ must be contained in the modified $\pgot_\omega$. Since $z$ is at distance at least $9$ from
$[t',t]$, it cannot be contained in it, so $z$ must be contained in the path $\pgot_\omega$. This
contradicts the fact that $z$ is at distance $1$ from $\pgot_\omega$.}

{If $\pgot_\omega$ and $\q_\omega$ are bi-infinite then the same operation as before may be
performed on both sides of $z$, obtaining again a contradiction.}

\medskip

$\mathbf{(2)\to (1)}$.  Let $\q$ be a bi-infinite quasi-geodesic satisfying (2). Any asymptotic cone in
which the ultralimit of $\q$ is non-empty equals to an asymptotic cone of the form $\co{X;(x_n), d}$,
where $x_n$ is a sequence of points on $\q$ and $d=(d_n)$ is a sequence of positive numbers with
$\lio{d_n}=\infty \, $. Take one of these asymptotic cones ${\mathcal C}$. Consider any point
$M=(m_n)^\omega$ with $m_n$ on $\q$.

It is enough to show that the two halves of $\lio{\q}$, before $M$ and after $M$, are in two different
connected components of ${\mathcal C}\setminus \{M\}$. {Indeed, we can then consider the
tree-graded structure on ${\mathcal C}$ with the maximal subsets of $\mathcal C$ without cut-points as
pieces \cite[Lemma 2.31]{DrutuSapir:TreeGraded}. The limit $\lio{\q}$ cannot intersect any piece in a
non-trivial arc, hence it is in a transversal tree of this tree-graded structure.}

Suppose there exist two points $A=(x_n)^\omega$, $B=(y_n)^\omega $ in ${\mathcal C}\setminus \{M\}$,
with $x_n,y_n\in \q$ and $m_n\in \q_{x_ny_n}$ such that  $A,B$ can be connected by a path $\pgot$ in
${\mathcal C}\setminus \{M\}$.

Let $2\epsilon$ be the distance from $M$ to $\pgot$. By Lemma \ref{piecewiseg}, we can assume that
$\pgot$ is the concatenation  of a finite number of limit geodesics. We can also assume that $\pgot$ is
simple.

By Lemmas \ref{lemlim0} and \ref{lemlim1}, there exists a constant $C\ge 1$, and a sequence of
$C$-bi-Lipschitz paths $\pgot_n$ connecting $x_n$ with $y_n$ such that $\lio{\pgot_n}$ is in
$\onn_\epsilon (\pgot)$. By Property (2), each path $\pgot_n$ must be contained in  the
$D$-neighborhood of $\q_{x_ny_n}$ for some constant $D$. It follows easily that $\lio{\pgot_n}=
\lio{\q_{x_ny_n}}$, and that $M$ is at distance $\le \epsilon$ from $\pgot$, a contradiction.

\medskip

$\mathbf{(5)\to (3)}$ and $\mathbf{(3)\to (4)}$ are obvious. In  $\mathbf{(3)\to (4)}$ one must use the
fact that every quasi-geodesic is at finite Hausdorff distance from a Lipschitz quasi-geodesic with the
same endpoints and with quasi-length of the same order \cite[Proposition 8.3.4]{Buragos}.

\medskip

$\mathbf{(4)\to (1)}$.  Let $\q$ be a bi-infinite quasi-geodesic satisfying (4). We argue as in $(2)\to
(1)$, and suppose there exist two points $A=(x_n)^\omega$, $B=(y_n)^\omega $ in ${\mathcal C}\setminus
\{M\}$, with $x_n,y_n\in \q$ and $m_n\in \q_{x_ny_n}$ such that  $A,B$ can be connected by a path
$\pgot$ in ${\mathcal C}\setminus \{M\}$, where $M=(m_n)^\omega$.

Let $2\epsilon$ be the distance from $M$ to $\pgot$. By Lemma \ref{piecewiseg}, we can assume that
$\pgot$ is the concatenation  of a finite number of limits of quasi-geodesics in $L$.

{Let $A'=(x_n')^\omega$, $B'=(y_n')^\omega $ with $x_n',y_n'\in \q$ be such that
$\q_{A'B'}=\lio{\q_{x_n'y_n'}}$ contains $M$ in its middle third, and this middle third is of diameter
at most $\epsilon $. Consider $\pgot'=\arct{A'A}\sqcup  \pgot \sqcup \arct{BB'}$, where $\arct{A'A}$
and $\arct{BB'}$ are limits of quasi-geodesics in $L$ joining $A',A$ and $B,B'$, respectively.}

{By Lemma \ref{lemlim1},  (\ref{quasi}), $\pgot' = \lio{\pgot_n}$, where each $\pgot_n$ is a
$k$-piecewise $L$ quasi-path joining $x_n'$ and $y_n'$, moreover each $\pgot_n$ is of quasi-length
$\leq D'\dist(a_n,b_n)$. By Property (4), each path $\pgot_n$ must cross the $D$-neighborhood of the
middle third of $\q_{x_n'y_n'}$ for some constant $D$. Then $\pgot'$ must cross the middle third of
$\q_{A'B'}$. Neither  $\arct{A'A}$ nor $\arct{BB'}$ can cross this middle third. Hence $\pgot$ contains
a point from the middle third of $\q_{A'B'}$, hence at distance $\le \epsilon$ from $M$. This
contradicts the fact that $\pgot$ is at distance $2\epsilon$ of $M$.}\endproof

\subsection{Morse elements}

Recall that an element $h$ of a finitely generated group $G$ is called {\em Morse} if $(h^n)_{n\in \ZZ
}$ is a Morse quasi-geodesic. Notice that the property of being a Morse element obviously does not
depend on the choice of a finite generating set of $G$.

\begin{lemma}\label{lmeasy} Suppose that $h$ is a Morse element in $G$, and $H$ is a finitely generated subgroup in $G$ containing $h$. Then $h$ is Morse in the group $H$ (considered with its own word metric).
\end{lemma}

\proof Without loss of generality, we may assume that the generating set of $H$ contains $h$ and is
inside the generating set of $G$. Since $h$ is Morse in $G$, $\q=\{h^n ; n\in \ZZ\}$ is a Morse
quasi-geodesic in $G$. In particular, $m\ge \dist_G(h^n,h^{n+m})\ge \frac1Lm-C$ for every $n,m$ and
some constants $C, L$. Since $\dist_H(u,v)\ge \dist_G(u,v)$ for every $u,v\in H$, we deduce that $\q$
is a quasi-geodesic in $H$.

Consider an arbitrary path $\pgot$ in $H$ connecting $h^{n}$ with $h^{n+m}$, $\pgot$ of length $\le
C_1m$ for some constant $C_1$. Then $\pgot$ is a path in $G$ having the same length. By Proposition
\ref{prop3}, Part (3), $\pgot$ comes $D$-close (in $G$) to the middle third of $[h^n, h^{m+n}]\subset
\q$, i.e. $\pgot$ contains a point $x$ which is at distance at most $D$ from $h^{n+i}$ where $\frac
m3\le i\le \frac{2m}{3}$. Here $D$ is a constant depending only on $C_1$. Note that the set
$\Ball_G(h^i, D)\cap H$ is finite, so it is contained in some ball $\Ball_H(h^i,D_1)$ in $H$ where
$D_1$ depends only on $D=D(C_1)$ . Hence $\dist_H(x, h^i)\le D_1$. Therefore $\pgot$ comes $D_1$-close
to the middle third of $[h^n, h^{m+n}]$. By Proposition \ref{prop3}, $\q$ is a Morse quasi-geodesic in
$H$.
\endproof

 {Lemma \ref{lmeasy} immediately implies}

\begin{proposition}\label{propsub} Let $H$ be a finitely generated subgroup of a finitely generated group $G$.
Suppose that $H$ does not have its own Morse elements. Then $H$ cannot contain Morse elements of $G$.
\end{proposition}

 {According to Proposition \ref{prop3} a group $H$ such that at least one asymptotic cone of it is
without cut-points (an \emph{unconstricted} group, in the terminology of \cite{DrutuSapir:TreeGraded})
does not contain Morse elements. Hence, Proposition \ref{propsub} can be applied to the cases when $H$
satisfies a non-trivial law or has an infinite cyclic central subgroup, or $H$ is a co-compact lattice
in a semi-simple Lie group of rank $\ge 2$, or $H$ is $\SL_n(\OS)$ for $n\ge 3$, or $H$ is a Cartesian
product of two infinite subgroups. Also if $H$ is a torsion group it cannot contain a Morse element.
Among the groups $G$ where Morse elements exist and play an important part can be counted hyperbolic or
relatively hyperbolic groups, or mapping class groups of a surface.}

\section{Cut-points in asymptotic cones of groups acting on hyperbolic graphs}
\label{sec4}

In this section, we prove that groups acting acylindrically on trees and more general hyperbolic graphs
have cut-points in their asymptotic cones and Morse quasi-geodesics. Recall that the action of a group
$G$ on a graph $X$ is \emph{acylindrical} if there exist $l$ and $N$ such that for any pair of vertices
$x,y$ in $X$ with $\dist (x,y)\geq l$ there are at most $N$ distinct elements, $g\in G$, such that
$gx=x,\, gy=y\, $ (in this case the action is called $l$-acylindrical).

\subsection{Trees}
\label{sec41}

\begin{theorem}\label{th31}  Suppose that a finitely generated group $G$
acts acylindrically on a simplicial tree $X$. Then every element $g\in G$ that does not fix a point in
$X$ is Morse. In particular, every asymptotic cone of $G$ has cut-points and non-trivial transversal
trees.
\end{theorem}
We remark that all the actions on trees we consider are such that there are no inversions of edges. One can always pass to this setting by splitting each edge into two edges by adding a new vertex at the middle of the edge.

Throughout the proof of this theorem, we suppose that $G$ acts $l$-acylindrically on the
simplicial tree~$X$.

Since $X$ is a tree, there exists a bi-infinite geodesic $\pgot$ in $X$ stabilized by $g$. The element
$g$ acts on $\pgot$ with some translation number $\tau$, and for every $x\in X$ we have that $\dist(x,
g\cdot x)\ge \tau$.

Consider a vertex $o\in \pgot$. We can assume that the tree $X$ is the convex hull of $G\cdot o$. The
factor-graph $X/G$ is finite, and $X$ is the Bass-Serre tree of some finite graph of groups $K$ (and
$G$ is the  fundamental group of that graph of groups). For every $h\in G$, we denote $h\cdot o$ by
$\pi(h)$.  { We consider $G$ endowed with a word metric corresponding to a finite generating set $U$ stable with respect to inversion. Let $\Ball (1,R)$ be the ball of radius $R$ in the group $G$ with respect to the word metric. Let $N$ be the upper bound on the cardinality of the stabilizers of arcs of length $l$ in $X$.}

\begin{lemma}\label{lem9}
\begin{enumerate}
    \item\label{card} For every $R\ge 0$ and every pair of vertices $a,b$ in the orbit $G\cdot o$
    the set $V_{a,b}$ of elements $g\in\pi\iv(a)$ such
that $\dist(g,\pi\iv(b))\le R$ is either empty or covered by at most $|\Ball (1,R)|$ right cosets of the
stabilizer of the pair $(a,b)$.
    \item\label{diam} In particular if $\dist(a,b)\geq l$, then
    there exists $D=D(R)$ such that $V_{a,b}$ has diameter at most $D$.
\end{enumerate}
\end{lemma}

\proof Without loss of generality we may assume that $a=o$. Let $S_o$ be the stabilizer of
$o$, and $S_b$ be the stabilizer of $b$. Then $\pi\iv (a)=S_o$, $\pi\iv (b)= g_bS_o$ and $S_b=g_bS_og_b\iv$, for some $g_b$ such
that $g_b\cdot o= b$.

(\ref{card}) Every element $g\in V_{o,b}$ has the property that $g\in S_o$ and for some $w$ in the ball $\Ball (1,R)$ of
$G$, $gw\in \pi\iv (b)$. Therefore
$$
V_{o,b}\subseteq g_bS_o\Ball (1,R) \cap S_o= g_bS_og_b\iv g_b\Ball (1,R) \cap S_o= S_bg_b \Ball (1,R)\cap S_o.
$$
For every $v\in g_b\Ball (1,R)$, either $S_bv\cap S_o$ is empty or $v= h(v)g(v)$ for some $h(v)\in S_b,
g(v)\in S_o$ and $S_bv\cap S_o=(S_b\cap S_o)g(v)$. Therefore the set $V_{o,b}$ is a union of at most
$|\Ball (1,R)|$ right cosets of $S_b\cap S_o$.
\medskip

(\ref{diam})  {Since $\dist(o,b)\ge l$, the stabilizer of $(o,b)$ has cardinality $\le N$, hence by (\ref{card}), the diameter of $V_{o,b}$ is finite. Let $D= D(R)$ be the supremum over all diameters of  $V_{o,b}$ with $b\in \Ball (1,R)\cdot o$, $\dist(o,b)\ge l$. If an arbitrary vertex $b$ with $\dist(o,b)\ge l$ is such that $V_{o,b}$ is non-empty then $b=g w o$ where $g\in S_o$ and $w\in \Ball (1,R)$. Therefore $V_{o,b} = g V_{o,wo}$  has diameter at most $D$.}\endproof

Now we are ready to {\em prove Theorem \ref{th31}.}

We choose a finite generating set $U$ of $G$
such that:

\smallskip

\noindent (Int) for every generator $u$, the geodesic in $X$ joining $o$ with $u\cdot o$ intersects $G\cdot o$ only in its endpoints.

\smallskip

Indeed, consider a finite  generating set $S$ of $G$ and, for every $s\in S$, consider the geodesic  $[o,s\cdot o]$ connecting $o$ and $s\cdot o$, and  $g_1,...,g_m$ a finite sequence of elements in $G$ such that $o, g_1\cdot o,  g_2\cdot o,..., g_{m-1}\cdot o , s\cdot o $ are the consecutive intersections of  $[o,s\cdot o]$ with  $G\cdot o$. Let $h_k=g_{k-1}\iv g_k$ for $k=1,2,...,m$, where we take $g_0=1$ and $g_m =s$. Then $s=h_1h_2...h_m$. It follows that the union of finite sets $\{ h_1^{\pm 1}, ..., h_m^{\pm 1} \}$ for all $s\in S$, composes a finite generating set satisfying (Int).

Let  $\lambda$ be a constant such that for every $u\in U$,
\begin{equation}\label{d1}
\dist_X(o, u\cdot o)\le \lambda \, .\end{equation}

Since $g$ acts on $X$ with translation number $\tau$, the sequence $(g^n)_{n\in \ZZ}$ is a
$(\frac{\lambda }{\tau}|g|, |g|)$-quasi-geodesic, where $\lambda$ is the constant in (\ref{d1}).

By Proposition \ref{prop3}, we need to show that for every $C>0$ there exists $D\ge 0$ such that every
path of length $\le Cn$ in the Cayley graph of $G$ joining $g^{-3n}$ and $g^{3n}$ passes within
distance $D$ from $g^i$ for some $i$ between $-n$ and $n$.

Consider an arbitrary constant $C$, $n>1$, and a path $\q$ of length $<Cn$ connecting $g^{-3n}$ with
$g^{3n}$ in the Cayley graph of $G$ ($n$ is large enough).

Denote the preimage $\pi\iv(\pi(g^i))$ by $Y_i$. Note also that $\dist(\pi(g^{-n}),\pi(g^n))\ge \tau
n$.

Let $h_1,h_2,...,h_m$ be the consecutive vertices of $\q$. Then by (\ref{d1}),
$\dist_X(\pi(h_i),\pi(h_{i+1}))\le \lambda$. Connecting $\pi(h_i)$ with $\pi(h_{i+1})$ by a geodesic in
$X$, we get a path $\pi(\q)$ on $X$ connecting $\pi(g^{-3n})$ with $\pi(g^{3n})$, and intersecting
$G\cdot o$ only in  $\pi(h_i),\, i=1,2,...,m$. Since $X$ is a tree, $\pi(\q)$ must cover the
sub-interval $[\pi(g^{-3n}), \pi(g^{3n})]$ of $\pgot$, hence $\pi (\q) \cap G\cdot o$ must contain all
$\pi (g^i),\, -3n\le i\le 3n$. Therefore $\q$ must cross all $Y_i$, $-3n\le i\le 3n$, on its way from
$g^{-3n}$ to $g^{3n}$.

Let $k-1$ be the integral part of $l/\tau$. We may assume that $n>k$.

The maximal sub-path of the path $\q$ joining a point in $Y_{-n}$ with a point in $Y_n$ can be divided
into sub-paths joining $Y_i$ to $Y_{i+1}$ such that their lengths sum up to the length of $\q$. There
exists $i\in [-n ,n]$ such that the sum of the lengths of the sub-paths of $\q$ between $Y_j$ and
$Y_{j+1} $ with $i\leq j \leq i+k$ is at most $4C k$ (otherwise the total length of $\q$ would be
greater than $Cn$). Note that the distance between $\pi(g^i)$ and $\pi(g^{i+k})$ is at least $k\tau >
l$, and the distance between $g^i\in Y_i$ and $g^{i+k}\in Y_{i+k}$ is at most $|g|k$. Let $R$ be the
maximum of $|g|k$ and $4Ck$. Let $x$ be the start point of the first sub-path of $\pgot$ joining $Y_i$
with $Y_{i+1}$. Applying Lemma \ref{lem9}, we find a constant $D=D(R)$ such that $x$ is at distance at
most $D$ from $g^i$ as required.\hspace*{\fill}$\square$

\begin{remark}\label{rem31} One cannot replace ``simplicial tree" by ``$\RR$-tree" in the formulation of Theorem \ref{th31}.
Indeed, the group $\ZZ^2$ acts freely on the real line but the asymptotic cones of $\ZZ^2$ do not have
cut-points.
\end{remark}

\subsection{ {Uniformly locally} finite hyperbolic graphs}

\label{sec42}

Theorem \ref{th31} can be easily generalized to actions on hyperbolic uniformly locally finite graphs.
Let $G$ be a group acting on a hyperbolic graph $X$. An element $g\in G$ is called {\em loxodromic} if
its translation length is $>0$.

\begin{theorem}\label{th32} Let $G$ be an infinite finitely generated group acting acylindrically on an infinite hyperbolic  uniformly locally finite graph $X$.
Suppose that for some $l>0$ the stabilizer of any pair of points $x,y\in X$ with $\dist(x,y)\ge l$ is
finite of uniformly bounded size. Let $g$ be a loxodromic element of $G$. Then the sequence
$(g^n)_{n\in \ZZ}$ is a Morse quasi-geodesic in $G$. In particular, every asymptotic cone of $G$ has
cut-points.
\end{theorem}

\blue{For the corrected proof see Section \ref{s:err}.}

\proof The proof is similar to the proof of Theorem \ref{th31}. By \cite[Lemma 3.4]
{Bowditch:tightgeod}, some power $g^m$, $m>0$, stabilizes a bi-infinite geodesic $\pgot$ in $X$. Since
we can always replace $g$ by $g^m\ne 1$, we can assume that $g$ stabilizes $\pgot$. {Let $o$ be a vertex on  $\pgot$.  We denote $g\cdot o$ also by
$\pi(g)$, for every $g\in G$.}

The proof of Lemma \ref{lem9} does not use the fact that $X$ is a tree, so it holds in our case too.

 {The fact that $X$ was a simplicial tree was used in the proof of Theorem \ref{th31} twice:}

\begin{itemize}
\item[(P1)] In the choice of a finite generating set of $G$ with property (Int) we used the uniqueness of a geodesic joining two points.

    \smallskip

\item[(P2)] We used that $X$ is a tree to deduce that $\pi(\q)$ must cover the interval $[\pi(g^{-3n}), \pi(g^{3n})]$ of $\pgot$.
\end{itemize}

In this proof we do not need a finite generating set of $G$ with property (Int). Instead of (P2) we can use the fact that $\pgot$ is a Morse quasi-geodesic
 (as is every bi-infinite geodesic in a Gromov hyperbolic space by, say, Proposition \ref{prop3}).
  Then by Proposition \ref{prop3}, part (5), the $D_0$-tubular neighborhood of $\pi(\q)$ must contain the interval $[\pi(g^{-3n}), \pi(g^{3n})]$ (for some constant $D_0=D_0(C)$).

Instead of preimages of points $Y_i=\pi\iv(\pi(g^i))$ let us consider the sets $Y_i'$ which are
$\pi$-preimages of balls of radius {$D_0+\lambda $ around $\pi(g^i)$. The path $\pi(\q)$ visits
each ball $\Ball(\pi(g^i),D_0)$, $-n\le i\le n$, so the path $\q$ must visit each $Y_i', -n\le i\le
n$.}

We need a version of Lemma \ref{lem9}, (\ref{diam}), for pairs of vertices $a,b$ in $G\cdot
o\subset X$ with $\dist (a,b)$ large enough, and for the set $\widetilde{V}_{a,b}$ of elements
$g\in \pi\iv (\Ball (a,D_0+\lambda ))$ at distance at most $R$ from $\pi\iv (\Ball(b,D_0+\lambda))$.

Consider a pair of vertices $a,b$  at distance $\geq l+ 4(D_0+\lambda)$ such that  $\widetilde{V}_{a,b}$ is non-empty. Then there exists $g\in G , w \in \Ball (1,R)$ such that $g\cdot o\in \Ball (a,D_0+\lambda )$ and  $gw\cdot o\in \Ball (b,D_0+\lambda )$. It follows that  $\widetilde{V}_{a,b} \subseteq g \widetilde{V}'$, where $\widetilde{V}'$ is the set of elements $ h\in \pi\iv (B(o, 2(D_0+\lambda)))$ at distance at most $R$ from $\pi\iv (\Ball(w\cdot o, 2(D_0+\lambda)))$. In other words $\widetilde{V}'$ is covered by the sets $V_{x,y}$ with $x\in B(o, 2(D_0+\lambda))$ and $y$ at distance at most $2(D_0+\lambda)$ from a vertex $w\cdot o$ in $B(1,R)\cdot o$ satisfying $\dist (o, w\cdot o)\geq l+ 4(D_0+\lambda)$. Since $X$ is locally finite it follows that there are finitely many pairs of vertices $x,y$ as above, all $V_{x,y}$ are finite, hence $\widetilde{V}'$ has finite diameter.

The rest of the proof of Theorem \ref{th31} carries almost without change.
\endproof


\subsection{Bowditch's graphs}
\label{sec43}

Theorem \ref{th31} can be also generalized to actions on a family of hyperbolic graphs which was first
introduced by Bowditch in \cite[$\S 3$]{Bowditch:tightgeod}, and which we call here Bowditch graphs.
This family of graphs includes the $1$-skeleta of curve complexes of surfaces. As a result, we recover
the theorem of Behrstock \cite{Behrstock:asymptotic} that asymptotic cones of mapping class groups have
cut-points, and for every pseudo-Anosov mapping class $g$, the sequence $(g^n)_{n\in\ZZ}$ is a Morse
quasi-geodesic.

Let us define Bowditch's graphs.

Let $\calgg$ be a $\delta$-hyperbolic graph. For any two vertices $a,b$, choose a non-empty set of
geodesics connecting $a,b$ and call these geodesics {\em tight}.

For every $r\ge 0$, we denote by $T(a,b;r)$ the union of all the points on all tight geodesics
connecting a point in the ball $\Ball(a,r)$ with a point in $\Ball(b,r)$.

\begin{definition}\label{bowd}
We say that a hyperbolic graph $\calgg$ equipped with a collection of tight geodesics as above is a
\emph{Bowditch graph} if it satisfies the following conditions:
\medskip

\begin{enumerate}
\item[(F0)] Every subpath of a tight geodesic is tight.

\medskip

\item[(F1)] For every $r>0$ there exist  $m=m(r)$ and $k=k(r)$ such that for every three vertices $a,b,c$ in $\calgg$ with $\dist(c, \{a,b\})\ge k$, the set $F_c(a,b;r)=\Ball(c,r)\cap T(a,b;r)$ contains at most $m$ points.
\end{enumerate}
\end{definition}

According to \cite[Theorems 1.1 and 1.2]{Bowditch:tightgeod}, the $1$-skeleton of a curve complex of a
surface is a Bowditch graph.

The first statement of the following lemma is \cite[Lemma 3.4]{Bowditch:tightgeod}, the second
statement is obvious.

\begin{lemma}\label{B1}
For every loxodromic isometry $g$ of $\calgg$, there exists a bi-infinite geodesic $\g$ and a natural
number  {$m\le m_0$, where $m_0=m_0(\calgg)$}, such that  $g^m\g=\g$.

If moreover the isometry $g$ preserves tight geodesics then for every vertex $c$ in $\g$,
$gF_c(a,b;r)=F_{g\cdot c}(g\cdot a, g\cdot b;r)$.
\end{lemma}

\begin{definition} Suppose that $G=\la S\ra$ acts on a Bowditch graph $\calgg$ and
 {the set of vertices $V(\calgg)=G\cdot\Delta$ for some finite set of vertices $\Delta$.} Let
$g,h\in G$, $o\in \calgg$. We say that a geodesic $\g$ in $\calgg$ with endpoints $g\cdot o$ and $
h\cdot o$ {\em $\Delta$-shadows} a path $p$ with vertices $p_1=g, p_2, p_3,...,p_m=h$ in $\Cay(G,S)$ if
the sets $p_i\cdot \Delta$, $i=1,...,m$, cover the  {set of vertices of the} geodesic $\g$.
\end{definition}

\begin{definition}
We say that a group $G$ acts {\em tightly} on $\calgg$ if:

\begin{itemize}
\item[(T1)] $G$ stabilizes the set of tight geodesics in $\calgg$;

\medskip

\item[(T2)] For every vertex $o$ of $\calgg$, there exist a finite set of vertices $\Delta$ in $\calgg$, and numbers $\lambda \ge 1, \kappa \ge 0$ such that every $g,h\in G$ can be joined by a
$(\lambda,\kappa)$-quasi-geodesic $\q$ in $\Cay(G,S)$ that is $\Delta$-shadowed by a tight geodesic
connecting $g\cdot o, h\cdot o$.
\end{itemize}

Following the terminology from \cite{MasurMinsky:complex2} for the mapping class group, we call
$(\lambda,\kappa )$-quasi-geodesics $\q$ as in (T2) \emph{hierarchy paths}.
\end{definition}

It follows from \cite{MasurMinsky:complex2} that the mapping class group of a punctured surface acts
tightly on the curve complex of the surface (see Lemma \ref{mmmm} below).

\begin{theorem}\label{th33}
Let $G$ be a finitely generated group acting tightly and acylindrically on a Bowditch graph $\calgg$.
Then every loxodromic element of $G$ is Morse.  In particular, if $G$ has loxodromic elements then every asymptotic cone of $G$ has cut-points.
\end{theorem}

\proof \blue{For a correction of this proof see Section \ref{s:err}.}
The proof is similar to the proofs of Theorems \ref{th31}, \ref{th32}. By Lemma \ref{B1} we can
assume that $g$ stabilizes  a bi-infinite geodesic $\pgot$, and that it acts on it with translation
length $\tau$. Pick a vertex $o$ on $\g$, and let $\Delta$ be the set from (T2).  Without loss of
generality we may assume that $o\in \Delta$. Since $\la g\ra$ acts co-compactly on $\g$, it has a
finite fundamental domain, which we include for convenience in $\Delta$ as well. Hence $\g$ is covered
by the sets $g^i\Delta$, $i\in \Z$.

Let $k$ be a natural number. By Proposition \ref{prop3}, (4), it is enough to show that if $\q$ is an
arbitrary $k$-piecewise hierarchy path connecting $g^{-3n}$, $g^{3n}$, with $n\gg1$, $\q$ at Hausdorff
distance $\le K_0$ from a path joining $g^{-3n}$, $g^{3n}$ and of length $\le K_0 n$, then $\q$ crosses
the $K$-neighborhood of the quasi-geodesic $[g^{-n}, g^n]$ where $K$ depends only on $k$ and on $g$
(but not on $\q$ nor on $n$).

Since $\q$ is a $k$-piecewise hierarchy path, by property (T2) it is shadowed by a $k$-piecewise tight
geodesic $\pi(\q)$ in $\calgg$ of length $\leq K_1n$ (for some constant $K_1$) connecting $g^{-3n}\cdot
o$ and $g^{3n}\cdot o$. The fact that geodesics in a hyperbolic graph are Morse and part (5) of
Proposition \ref{prop3} imply that the sub-arc $[g^{-3n}\cdot o , g^{3n}\cdot o]$ in $\g$ is contained
in the $D$-tubular neighborhood of $\pi(\q)$ for some constant $D$. In particular $[g^{-n}\cdot o ,
g^{n}\cdot o]$ has a sub-arc $\g'$ of length $\geq K_2n$ (for some constant $K_2$) contained in the
$D$-tubular neighborhood of one of the tight geodesic subpaths $\ft$ of $\pi (\q )$. Notice that the
length $|\ft|$ is $\geq K_2n-2D\ge K_3n$ for some constant $K_3$ (since $n\gg 1$).

Since $\g'$ and $\ft$ are two geodesics in a hyperbolic space and $\dist(\g'_-,\ft_-)\le D,
\dist(\g'_+,\ft_+)\le D$, the Hausdorff distance between these two geodesics is at most $K_4$ for some
constant $K_4$.

Since $\ft$ is tight, by (T2), $\ft$ is covered by sets $F_x=F_x(\g'_-, \g'_+; K_4)$ where $x\in \g'$.

Let $\q'$ be the hierarchy sub-path in $\q$ shadowed by $\ft$. As in the proofs of Theorems \ref{th31}
and \ref{th32}, we find a subarc $\q''$ with endpoints $h, h'\in G$ at distance $\le K_5l$ (for some
constant $K_5$) which is shadowed by a sub-arc of $\ft$ of length at least $l+K_6$ of $\ft$ where $K_6$
is any number exceeding the diameter of $\Delta$ plus $2K_4$ (recall that $2K_4$ bounds the diameters
of $F_x, x\in \g'$). Denote the endpoints of that sub-arc of $\ft$ by $x, y$. Then there are vertices
$u, v$ in $\g'$ such that $x\in F_u, y\in F_v$, and there are two powers of $g$, say, $g^i, g^j$,
$-n\le i,j\le n$, such that $u\in g^i\Delta$, $v\in g^j\Delta$ (recall that $g^m\Delta$, $m\in \Z$,
covers $\g$). By Lemma \ref{B1}, the number of elements in $F_u\cup F_v$ is bounded by constants which
do not depend on $n$. Hence $h$ is contained in a union of bounded (independently of $n$) number of
subsets $V_{a,b}$ with $a,b\in \calgg$, $\dist(a,b)\ge l$. It is easy to establish the natural
generalization of Lemma \ref{lem9} to Bowditch graphs. Hence the distance $\dist(h,g^i)$ is bounded by
a number that does not depend on $n$.
\endproof

\begin{lemma}\label{mmmm}
The mapping class group of a compact connected orientable surface acts tightly on the $1$-skeleton of
the curve complex.
\end{lemma}

\proof Let $S$ be a surface of genus $g$ with $p$ punctures. The fact that the action of the mapping
class group $\MCG(S)$ on the curve complex of $S$ satisfies (T1) is proved in
\cite{MasurMinsky:complex2}.

By \cite{MasurMinsky:complex2}, the mapping class group $\MCG(S)$ acts co-compactly on the so called
marking complex $\MM(S)$. Each marking $\mu\in\MM(S)$ consists of a pair of data: a \emph{base
multicurve} denoted $\base(\mu)$, which is a multicurve composed of $3g+p-3$ disjoint curves, and a
collection of \emph{transversal curves}. The 1-skeleton of $\MM(S)$ is a locally finite graph. So there
exists a finite collection of markings $\Phi$ such that $\MCG(S)\cdot \Phi=\MM(S)$. Let $\Delta$ be the
union of all the base curves of all markings in $\Phi$.

If is proved in \cite{MasurMinsky:complex2} that there exists an $\MCG(S)$-equivariant projection $\pi$
of $\MM(S)$ onto the curve complex of $S$ such that for every two markings $\mu, \nu$, there exists an
$(L,C)$-quasi-geodesic in $\MM(S)$ connecting $\mu$ and $\nu$ and $\Delta$-shadowed by a tight geodesic
connecting $\pi(\mu), \pi(\nu)$, where $L, C$ are uniform constants. Using the quasi-isometry of
$\MCG(S)$ and $\MM(S)$, we can pull the set of hierarchy paths of $\MM(S)$ into $\MCG(S)$ and obtain a
collection of hierarchy paths in $\MCG(S)$ satisfying (T2).
\endproof

Lemma \ref{mmmm} and Theorem \ref{th33} immediately imply Behrstock's theorem from
\cite{Behrstock:asymptotic} that every pseudo-Anosov element in the mapping class group is Morse since
pseudo-Anosov elements are loxodromic (see, for example, \cite{Bowditch:tightgeod}).

\section{Lattices in semi-simple Lie groups. The $\Q$-rank one case}
\label{sec5}
\subsection{Preliminaries}

\subsubsection{Horoballs and horospheres}\label{h}

Let $X$ be a CAT(0)-space and $\rho $ a geodesic ray in $X$. {\it The Busemann function associated to
$\rho$ } is the function $ f_\rho:X\to \RR ,\; f_\rho(x)=\lim_{t\to \infty}[d(x,\rho(t))-t]\; .$ A
level hypersurface $H_a(\rho )=\lbrace x\in X \mid f_\rho(x)= a \rbrace$ is called {\it horosphere}, a
level set $Hb_a(\rho ):=\lbrace x\in X \mid f_\rho(x)\leq a \rbrace$ is called {\it closed horoball}
and its interior, $\Hbo_a(\rho ):=\lbrace x\in X \mid f_\rho(x)< a \rbrace$, {\it open horoball}. We
use the notations $H(\rho ),\; Hb(\rho ),\; \Hbo(\rho )$ for the horosphere, the closed and open
horoball corresponding to the value $a=0$.

For two asymptotic rays, their Busemann functions differ by a constant \cite{BridsonHaefliger}. Thus
the families of horoballs and horospheres are the same and we shall call them horoballs and horospheres
{\it of basepoint }$\alpha$, where $\alpha $ is the common point at infinity of the two rays.

\begin{lemma}(\cite[Lemma 2.C.2]{Drutu:Filling})\label{3}
Let  $X$ be a product of symmetric spaces and Euclidean buildings and $\alpha_1 ,\alpha_2 ,\alpha_3$
three distinct points in $\bbi X$. If there exist three open horoballs $\Hbo_i$ of basepoints
$\alpha_i,\; i=1,2,3$, which are mutually disjoint then $\alpha_1 ,\alpha_2 ,\alpha_3$ have the same
projection on the model chamber $\Delta_{mod} $.
\end{lemma}

\subsubsection{Spherical and Euclidean buildings}

Let $Y$ be an Euclidean building.

An Euclidean building is called {\emph{$c$-thick}} if every wall bounds at least $c$ half-apartments
with disjoint interiors.

\begin{lemma}[\cite{KleinerLeeb:buildings}, proof
of Proposition 4.2.1]\label{lflats} Two geodesic rays $r_1,r_2$ in $Y$ are asymptotic respectively to
two rays $r_1',r_2'$ bounding an Euclidean planar sector with angular value the Tits distance between
$r_1(\infty)$ and $r_2(\infty).$
\end{lemma}

Let $\Sigma$ be a spherical building of rank $2$. Then
\begin{itemize}
\item $\Sigma $ is a $CAT(1)$ spherical complex of dimension one, all its
 simplices are isometric (and called chambers), and isometric to an arc of circle of length $\pi/m$;
\item all simplicial cycles are of simplicial length at least $2m$ (so of length at least $2\pi $);
\item $\Sigma$ is of simplicial diameter $m+1$;
\item If two points in $\Sigma$ are at distance smaller than $\pi$ then there exists a unique geodesic joining
them.
\end{itemize}

\begin{definition}
Let $A$ be an apartment in $\Sigma$ and let $x$ be a point outside $A$. We call {\emph{entrance point
for $x$ in $A$}} any point $y$ such that the geodesic joining $x$ with $y$ intersects $A$ only in $y$.
\end{definition}

\begin{lemma}\label{entrv}
Every entrance point in an apartment $A$ for a point outside $A$ is a vertex.
\end{lemma}

\proof If an entrance point $y$ is not a vertex, then it is in the interior of a chamber $\ww$. By the
axioms of buildings $\ww$ and $x$ are contained in one apartment. Obviously one of the vertices of
$\ww$ is on the geodesic joining $x$ to any interior point of $\ww$, in particular with $y$. On the
other hand the fact that the point $y$ is in $A$ implies that $\ww$ is in $A$.\endproof

According to the result in Lemma \ref{entrv}, we shall speak from now on of {\emph{entrance vertices
for a point in an apartment}}.

\begin{lemma}\label{entrpi}
Let $\Sigma$ be a spherical building of rank $2$, let $x$ be a point in
 it, and let $A$ be an apartment not containing $x$.

If the distance from $x$ to $A$ is less than $\pi/2$ then there exists at most one entrance vertex for
$x$ in $A$ which is
  at distance less than $\pi/2$ from $x$.

If the distance from $x$ to $A$ is $\pi/2$ then there exist at most two entrance vertices for $x$ in
$A$ which are
  at distance $\pi/2$ from $x$. Moreover if two such vertices exist, then they must be opposite.
\end{lemma}

\proof In the first case, since two vertices in $A$ \red{are at} distance at most $\pi$, the existence of two
entrance vertices would imply the existence of a cycle in $\Sigma$ of length $<2\pi$.

The second case is proved similarly.\endproof

\subsubsection{Lattices in semi-simple groups}

A {\it lattice} in a Lie group $G$ is a discrete subgroup $\Gamma $ such that $\Gamma \backslash G$
admits a finite $G$-invariant measure. We refer to \cite{Margulis:DiscreteSubgroups} or
\cite{Raghunathan:discrete} for a definition of $\Q$-rank for an arithmetic lattice in a semi-simple
group.

\begin{theorem}[Lubotzky-Mozes-Raghunathan, \cite{LMR:Cyclic},
\cite{LMR:metrics}]\label{lmr} On any irreducible lattice of a semisimple group of rank at least $2$,
the word metrics and the induced metric are bi-Lipschitz equivalent.
\end{theorem}

By means of removing open horoballs one can construct a subspace $X_0$ of the symmetric space $X=G/K$
on which the lattice $\Gamma $ acts with compact quotient. In the particular case of lattices of
$\Q$-rank one, the family of open horoballs have the extra property of being disjoint.

\begin{theorem}[ \cite{Raghunathan:discrete}, \cite{GarlandRaghunathan}]\label{xo}
Let $\Gamma $ be an irreducible lattice of $\Q$-rank one in a semisimple group $G$. Then there exists a
finite set of geodesic rays $\{ \rho_1,\rho_2,\dots ,\rho_k \}$ such that the space $X_0 =X\setminus
\bigsqcup_{i=1}^k\bigcup_{\gamma \in \Gamma} \Hbo(\gamma \rho_i)$ has compact quotient with respect to
$\Gamma$ and such that any two of the horoballs $\Hbo(\gamma \rho_i)$ are disjoint or coincide.
\end{theorem}

Let $p$ be the projection of the boundary at infinity onto the \emph{model chamber} $\Delta_{mod}$.
Lemma \ref{3} implies that $p(\lbrace \gamma \rho_i(\infty )\mid \gamma \in \Gamma ,\;  i\in
\lbrace1,2,\dots k\rbrace \rbrace)$ is a singleton which we denote by $\theta $, and we call {\it the
associated slope of }$\Gamma\, $. We have the following property of the associated slope :

\begin{proposition}[\cite{Drutu:Nondistorsion}, Proposition 5.7]\label{dir}
If $\Gamma $ is an irreducible $\Q$-rank one lattice in a semi-simple group $G$ of $\RR$-rank at least
$2$, the associated slope, $\theta$, is never parallel to a factor of $X= G/K$.
\end{proposition}

In particular, if $G$ decomposes into a product of rank one factors, $\theta $ is a point in
$\mathrm{Int}\; \Delta_{mod}$.

Since the action of $\Gamma $ on $X_0$ has compact quotient, $\Gamma $ with the word metric is
quasi-isometric to $X_0$ with the length metric (the metric defining the distance between two points as
the length of the shortest curve between the two points). Thus, the asymptotic cones of $\Gamma $ are
bi-Lipschitz equivalent to the asymptotic cones of $X_0$. Theorem \ref{lmr}  implies that one may
consider $X_0$ with the induced metric instead of the length metric. We study the asymptotic cones of
$X_0$ with the induced  metric.

\begin{theorem}[\cite{KleinerLeeb:buildings}]\label{cX}
Any asymptotic cone of a product $X$ of symmetric spaces and Euclidean buildings, $X$ of rank $r\geq
2$, is an  Euclidean building ${\mathbf K}$ of rank $r$ which is homogeneous and $\aleph_1$-thick. The
apartments of ${\mathbf K}$ appear as limits of sequences of maximal flats in $X$. The same is true for
Weyl chambers and walls, singular subspaces and Weyl polytopes of ${\mathbf K}$. Consequently, $\bbi
{\mathbf K}$ and $\bbi X$ have the same model spherical chamber and model Coxeter complex.
\end{theorem}

 In the sequel, \red{we use the terminology and notation from \cite{KleinerLeeb:buildings}}. In any asymptotic cone ${\mathbf K}$ of a
product $X$ of symmetric spaces and Euclidean buildings we shall consider the labeling on \red{the boundary at infinity}
$\partial_\infty {\mathbf K}$ induced by a fixed labeling on $\partial_\infty X$. We denote the
projection of $\partial_\infty {\mathbf K}$ on $\Delta_{mod} $ induced by this labeling by $P$ and the
associated Coxeter complex by $\mathsf{S}$.

Concerning the asymptotic cone of a space $X_0$ obtained from a product of symmetric spaces and
Euclidean buildings by deleting disjoint open horoballs, we have the following result

\begin{theorem}[\cite{Drutu:Remplissage}, Propositions 3.10, 3.11]\label{cX0}
Let $X$ be a CAT(0) geodesic metric space and let $\calc=\Con^\omega(X, (x_n),(d_n))$ be an asymptotic
cone of $X$.

(1) If $(\rho_n)$ is a sequence of geodesic rays in $X$ with $\frac{d(x_n,\rho_n)}{d_n}$ bounded and
$\rho =\lbrack \rho_n\rbrack $ is its limit ray in ${\mathbf K}$, then $H(\rho )=\lbrack
H(\rho_n)\rbrack $ and $Hb(\rho )=\lbrack Hb(\rho_n)\rbrack$.

(2) If $X_0=X \setminus \bigsqcup_{\rho \in {\mathcal R}}\Hbo(\rho )$ and $\frac{d(x_n,X_0)}{d_n}$ is
bounded then the limit set of $X_0$ (which is the same thing as the asymptotic cone of $X_0$ with the
induced metric)
 is
$$
\calc_0=\calc \setminus \bigsqcup_{\rho_\omega \in {\mathcal R}_\omega}\Hbo(\rho_\omega )\;
,\leqno(2.1)
$$ where
${\mathcal R}_\omega $ is the set of rays $\rho_\omega =\lbrack \rho_n\rbrack ,\; \rho_n \in {\mathcal
R}$.
\end{theorem}

We note that if $X$ is a product of symmetric spaces and Euclidean buildings, the disjointness of
$\Hbo(\rho ),\; \rho \in {\mathcal R}$, implies by Lemma \ref{3} that $p(\lbrace \rho (\infty )\mid
\rho \in {\mathcal  R}\rbrace )$ reduces to one point, $\theta $, if card ${\mathcal R}\neq 2$. Then
$P(\lbrace \rho_\omega (\infty )\mid \rho_\omega \in {\mathcal R}_\omega \rbrace ) =\theta $.

We also need the following result.

\begin{lemma}\label{platr}
Let $X$ be a product of symmetric spaces and Euclidean buildings, $X$ of rank $r\geq 2$, and
${\mathbf{K}}= {\Con_\omega(X,x_n,d_n)}$ be an asymptotic cone of it. Let $F_\omega $ and $\rho_\omega $ be an
apartment and a geodesic ray in ${\mathbf{K}}$, $F_\omega$ asymptotic to $\rho_\omega $. Let
$\rho_\omega =[\rho_n ]$, where $\rho_n$ have the same slopes as $\rho_\omega $. Then
\begin{itemize}
\item[(a)] $F_\omega$ can be written as limit set $F_\omega =[F_n]$ with $F_n$ asymptotic to $\rho_n$ $\omega $-almost surely ;
\item[(b)] every geodesic segment $[x,y]$ in $F_\omega \setminus \Hbo\, (\rho_\omega)$ may be written as limit set of segments $[x_n,y_n]\subset F_n \setminus \Hbo\, (\rho_n)$.
\end{itemize}
\end{lemma}

\subsection{Lattices in ${\mathbb Q}$-rank one Lie groups of real rank $\ge 2$}

\begin{theorem}\label{thq1}
Let $X$ be a product of symmetric spaces of non-compact type and Euclidean buildings of rank at least
two, and let $\calr$ be a collection of geodesic rays in $X$, with no ray contained in a rank one
factor, and such that if $r,r'$ are two distinct elements of it then the open horoballs
$\Hbo(r),\Hbo(r')$ are disjoint. Then the space $X'=X \setminus \bigsqcup_{r\in \calr } \Hbo(r)$ has
linear divergence $\Dv_{\red{\gamma }}(n,\delta)$ for every $\delta\in (0,1)$.
\end{theorem}

Theorems \ref{xo} and \ref{thq1} immediately imply

\begin{cor}\label{c1} Every lattice of a semi-simple Lie group of $\QQ$-rank 1 and $\RR$-rank $\ge 2$ has linear divergence and no cut-points in its asymptotic cones.
\end{cor}

In what follows $X$ and $\calr$ will always be as in Theorem \ref{thq1}.

\begin{lemma}[Proposition 3.A.1, \cite{Drutu:Filling}]\label{hrf}
Let $Y$ be an Euclidean building of rank at least two, let $r$ be a geodesic ray in it and let $F$ be
an apartment intersecting the horoball $\Hb(r)$. The intersection of  $F$ with $\Hb (r)$ is a convex
polytope whose interior is $F\cap \Hbo(r)$. In particular, if the interior of the polytope is empty
then it is in the horosphere $H(r)$.
\end{lemma}

\begin{lemma}\label{infr}
Let $Y$, $r$ and $F$ be as in Lemma \ref{hrf}. If $F\cap \Hb(r)$ has infinite diameter then the Tits
distance between $\bbi F$ and $r(\infty )$ is at most $\pi/2$.
\end{lemma}

\proof Let $o$ be a fixed point in $F\cap \Hb(r)$. Since $F\cap \Hb(r)$ has infinite diameter and it is
a polytope, it contains a geodesic ray.

According to Lemma \ref{lflats} the rays $\rho$ and $r$ are asymptotic to rays $\rho'$ and $r'$
bounding an Euclidean sector with angular value  the Tits distance between $\rho (\infty )$ and $r
(\infty )$.

Since $\rho $ and $\rho '$ at finite Hausdorff distance it follows that $\rho '$ is contained in $\Hb_a
(r)$ for some $a>0$. The ray $r'$ is asymptotic to $r$, hence $\Hb_a(r)$ coincides with some $\Hb_b
(r')$, hence $\rho'$ is contained in $\Hb_b (r')$. This cannot happen if the angle between $\rho'$ and
$r'$ is larger than $\pi/2$.

It follows that the Tits distance between $\rho (\infty )$ and $r (\infty )$ is at most $\pi /2$, hence
the same holds for the Tits distance between $\bbi F$ and $r(\infty )$.\endproof

\begin{lemma}\label{lrk2}
Let $Y$ be a $4$-thick Euclidean building, of rank $2$, let $r$ be a geodesic ray in $Y$, not contained
in a factor of $Y$, let $F$ be an apartment intersecting $\Hbo (r)$ and let $a,b$ be two distinct
points in $F\cap H(r)$. Then there exists an apartment $F'$ in $Y$ containing $a,b$ and a point from
$\Hbo(r)$, and such that the Tits distance from $\bbi F$ to $r(\infty)$ is larger than $\pi /2$.
\end{lemma}

\proof  By \cite[Proposition 4.2.1]{KleinerLeeb:buildings} the boundary at infinity $\bbi Y$ endowed
with the Tits metric is a spherical building of rank $2$. All its chambers are isometric, and isometric
to an arc of circle of angle $\theta =\pi /m$, with $m\in \N , m\geq 2$. The building $Y$ is reducible
if and only if $m=2$ \cite[Proposition 3.3.1]{KleinerLeeb:buildings}.

Assume that $Y$ is reducible. Then $Y$ is a product of trees $T_1\times T_2$, and $r(t)=(r_1(\sigma
t),r_2(\tau t))$, where $r_i$ is a ray in $T_i$, and $\sigma^2+\tau^2=1$, $\sigma >0$ and $\tau
>0$. Let $a=(a_1,a_2)$, let $b=(b_1,b_2)$. There exists a
geodesic line $L_i$ in $T_i$ containing $[a_i,b_i]$ and not asymptotic to $r_i$. Then the apartment
$L_1 \times L_2$ satisfies the hypothesis. Indeed, the Tits distance from $\bbi (L_1 \times L_2)$ to
$r(\infty)$ is larger than $\pi /2$, because all chambers in $\bbi (L_1 \times L_2)$ are opposite to
the chamber containing $r(\infty )$.

In order to prove that $L_1\times L_2$ contains a point in $\Hbo (r)$,  note first that the value of
the Busemann function $f_r (x_1,x_2)$ is equal to $\sigma f_{r_1}(x_1) +\tau f_{r_2}(x_2)$. Hence
$\sigma f_{r_1}(a_1) +\tau f_{r_2}(a_2) = \sigma f_{r_1} (b_1) +\tau f_{r_2} (b_2) =0$.

If $f_{r_1}(a_1) <f_{r_1}(b_1)$ then $f_{r_2}(a_2) > f_{r_2}(b_2)$ and the point $(a_1,b_2)$ in $L_1
\times L_2$ is in $\Hbo (r)$. If $f_{r_1}(a_1) = f_{r_1}(b_1)$ then $f_{r_2}(a_2) =f_{r_2}(b_2)$ and
either $a_1\neq b_1$ or $a_2\neq b_2$. Assume that $a_1\neq b_1$. The geodesic $[a_1,b_1]$ contains a
point $e_1$ such that $f_{r_1}(e_1) < f_{r_1}(a_1)$. Then the point $(e_1,a_2)$ in $L_1 \times L_2$ is
in $\Hbo (r)$.

Assume that $Y$ is irreducible. Let $F$ be an apartment in $Y$ containing $e\in \Hbo(r)$, and $a\ne
b\in H(r)$. Assume that the Tits distance $\delta$ from $r(\infty )$ to $A=\bbi F$ is smaller than
$\pi/2$. Then we construct an apartment $F'$ containing $a,b,e$ and a point from $\Hbo(r)$ such that
$\bbi F'$ is at Tits distance $\delta +\theta $ from $r(\infty )$.

Indeed, Lemma \ref{entrpi} implies that in this case there exists only one entrance vertex $u$ for
$r(\infty )$ in $A$ at distance $\delta$. All the other vertices in $A$ are at distance at least
$\delta +\theta $ from $r(\infty )$. Let $v,w$ be two opposite vertices in $A \setminus \{ u\}$. The
$0$-sphere $\{v,w\}$ is the boundary at infinity of a singular line $H$, and we may assume that this
line does not separate $\{ a,b,e\}$, but separates the set $\{a,b, e \}$ from a geodesic ray with point
at infinity $u$.

By the hypothesis of thickness there exists a half-apartment $\mathcald$ in $Y$ of boundary $H$ and
with interior disjoint from $F$. Let $\mathcald_i$, $i=1,2$, be the two half-apartments in $F$
determined by $H$ such that $D_1$ contains $u$ (hence $D_2$ contains $\{a,b,e\}$). Note that $D\cup
D_i,i=1,2,$ is an apartment. Lemma \ref{entrpi} applied to the spherical apartment $\bbi (D_1\cup D)$
implies that all the vertices in $\bbi D$ are at Tits distance at least $\delta +\theta$ from
$r(\infty)$. It follows that all the vertices in $\bbi (D_2\cup D)$ are at Tits distance at least
$\delta +\theta$ from $r(\infty)$. Take $F'=D_2\cup D$.

Now we can assume the Tits distance from $r(\infty )$ to $A=\bbi F$ is $\pi/2$. Then we construct an
apartment $F'$ containing $a,b,e$ such that $\bbi F'$ is at Tits distance $\pi/2 +\theta $ from
$r(\infty )$.

Lemma \ref{entrpi} implies that $\bbi F$ contains at most two entrance vertices for $r(\infty )$ in
$A$, and that in case there are two, they must be opposite. Let $H$ be the singular hyperplane in $F$
containing either one or both these entrance vertices in its boundary.

We may moreover assume that $H$ does not separate $a,b,e$. Let $H'$ be a singular hyperplane composing
with $H$ two opposite Weyl chambers, and which also does not separate $a,b,e$. By the irreducibility
assumption on $Y$, $H'$ is not orthogonal to $H$.

The line $H'$ splits $F$ into two half-apartments $D_1,D_2$, with $D_1$ containing $a,b,e$. By
$3$-thickness there exist $D_3$ a half-apartment of the boundary of $H'$ such that $D_1,D_2,D_3$ have
pairwise disjoint interiors.

Assume that $\bbi F$  contains two opposite entrance vertices $x,y\in A$ for  $r(\infty )$ in $A$ with
$x\in \bbi D_1$. If the apartment $\bbi D_1\cup D_3$ contains two entrance vertices at distance $\pi
/2$ then the second entry vertex must be in $D_3$, and opposite to $x$. The same point must be also
opposite to $y$, but since $x,y$ are not symmetric with respect to $\bbi H'$, this gives a
contradiction. Thus $\bbi D_3$ is at Tits distance $\pi/2 +\theta $ from $r(\infty )$, and the
apartment $F'=D_1\cup D_3$ is such that $\bbi F'$
 contains only one entrance vertex $x$ for $r(\infty )$ at distance $\pi /2$.

Thus we reduced to the case when $\bbi F$ contains only one entrance vertex $x$ for $r(\infty )$ at
distance $\pi /2$. The line $H'$ may be chosen such that it splits $F$ into two halves $D_1,D_2$, with
$D_1$ containing $a,b,e$ and $\bbi D_2$ containing $x$. By $4$-thickness it also bounds two half
apartments $D_3,D_4$ s.t. $D_i, i=1,...,4$ have disjoint interiors.

If $\bbi D_3$ is at Tits distance $\pi/2 +\theta$ from $r(\infty)$ then $F'=D_1\cup D_3$ is the
required apartment.

Assume that $\bbi D_3$ is at Tits distance $\pi/2$ from $r(\infty)$. Then the spherical apartment $\bbi
(D_2\cup D_3)$ has two opposite entrance points for $r(\infty )$ at distance $\pi /2$ from $r(\infty)$.
An argument as above implies that $\bbi D_4$ is at distance $\pi/2 +\theta$ from $r(\infty)$. We can
take $F'=D_1\cup D_4$.
\endproof

\begin{lemma}\label{lm000}
Let $[x,y]$ be a segment containing a point $o$ in its interior, and let $r$ be a ray in an Euclidean
building with origin $o$. Then there exists $x'\in [x,o)$, $y'\in (o,y]$ such that $r$ and $[x',y']$
are in the same apartment.
\end{lemma}

\proof For every $t\in [o,x)$ denote by $\theta (t)$ the angle between the segment $[t,x]$ and the ray
of origin $t$ asymptotic to $r$. According to \cite[Lemmas 2.1.5 and 5.2.2]{KleinerLeeb:buildings} the
map $t\mapsto \theta (t)$ is an increasing upper semi-continuous function with finitely many values. It
follows that for some $x'\in [x,o)$ the function is constant on $[x',o]$. Similarly one can find a
point $y'\in (o,y]$ with the angle function constant on $[o,y']$. Since $[x',y']$ is a geodesic,  the
segment $[x',y']$ and the ray $r$ are in the same flat, hence in the same apartment.\endproof

\begin{proposition}\label{qprop0}
Let $Y$ be an Euclidean building, of rank at least $2$, and let $r$ be a geodesic ray not contained in
a factor of $Y$. Let $H(r)$ be the horosphere in $Y$ determined by $r$. Then for every three points
$a,b,c$ in $H(r)$ with $\dist (a,c)=\dist (b,c)=1$ there is a path in $H(r)$ connecting $a$ and $b$ and
avoiding $c$.
\end{proposition}

\proof {\sc{Case 1}}.\quad Assume that $Y$ is of rank $2$. If the points $a,b$ are contained in an
apartment $F$ intersecting $\Hbo (r)$ then by Lemmas \ref{infr} and \ref{lrk2} it can be assumed that
$F\cap H(r)$ is the boundary of a finite convex polytope (i.e. a flat polygon since the rank is 2) with
non-empty interior. Then $F\cap H(r)$ is a simple loop containing $a,b$, and one of the boundary paths
of this loop connecting $a,b$ does not pass through $c$.

Now suppose that $a,b$ are contained in an apartment $F$ such that $F\cap \Hbo (r)$ is empty. Then
$F\cap H(r)=F\cap Hb(r)$ is a convex polytope of dimension 1, so it is a segment, a ray or a line. If
$c$ is not on the segment $[a,b]\subset F\cap H(r)$ then we are done.

Assume that $c\in [a,b]$. By Lemma \ref{lm000} there exists an apartment $F'$ which is asymptotic to
$r$ and which contains a sub-segment $[a',b']$ of $[a,b]$ having $c$ as an interior point.

Note that the apartment $F'$ intersects $\Hbo (r)$. Thus, the first part of the proof can be applied to
show that $a'$ and $b'$ can be connected in $H(r)$ avoiding $c$. This finishes the proof of Case 1.

\medskip

{\sc{Case 2}}.\quad Assume that $Y$ is of rank $n>2$. Consider an apartment $F$ containing $a,b$.

If $F$ also intersects $\Hbo (r)$ then $F\cap \Hb(r)$ is a polytope of non-empty interior and dimension
$n\geq 3$.  If this polytope has two non-parallel co-dimension one faces then its boundary is
connected, and so we can connect $a$ and $b$ by an arc on the boundary of the polytope avoiding $c$. So
suppose that the boundary of the polytope has just two (parallel) co-dimension one faces. We can assume
that $a,b$ belong to different faces. By \cite[Lemma 3.C.2]{Drutu:Filling}, we can find an apartment
$F'$ containing $a,b$ such that $F'\cap H(r)$ contains two non-parallel co-dimension one faces, and we
are done.

If $F$ does not intersect $\Hbo (r)$ then $F\cap H(r) =F\cap Hb(r)$ is a convex polytope of dimension
$d<n$. If $d\geq 2$ then this polytope cannot have cut-points. Assume that $d=1$, hence $F\cap H(r)$ is
a segment, a ray or a line. As above, if $c\not \in [a,b]$ then we are done.

If $c\in [a,b]$ then by Lemma \ref{lm000} there exists an apartment $F'$ asymptotic to $r$ and
containing a sub-segment $[a',b']$ of $[a,b]$ with $c$ in the interior. In particular $[a',b']$ is in
$F'\cap H(r)$, which is a hyperplane of dimension $n-1$ (because it is a horosphere of the flat $F'$).
Clearly $a',b'$ can be joined in this hyperplane by a path avoiding $c$.\endproof

\begin{proposition}\label{qprop1}
Let $X$ be as in Theorem \ref{thq1} and let $r$ be a geodesic ray not contained in a factor of $X$. Let
$H$ be the horosphere in $X$ corresponding to $r$. Then there exists a constant $\lambda \in (0,1)$ and
two positive constants $D$ and $L$ such that for every three points $a,b,c$ in $H$ with $\min \left\{
\dist (a,b),\dist (a , c), \dist(b,c)\right\}\geq D$, there is a path of length at most $L \dist(a,b)$
in $H$ connecting $a$ and $b$ and avoiding the ball of radius $\lambda \min(\dist(a,c), \dist(b,c))$
around $c$.
\end{proposition}

\proof Assume by contradiction that there exists a sequence of triples $a_n,b_n,c_n$ such that the
minimum of $\dist (a_n,b_n),\dist (a_n, c_n), \dist(b_n,c_n) $, denoted by $D_n $, diverges to
infinity, and such that all paths joining $a_n$ and $b_n$ outside the ball of radius $\frac{1}{n}
\min(\dist(a_n,c_n), \dist(b_n,c_n))$ around $c_n$ have length at least $n\dist (a_n,b_n)$. Without
loss of generality we may assume that $\dist(a_n,c_n)=\dist (b_n,c_n)=R_n$. The assumptions imply that
$\dist (a_n,b_n)$ is at least $R_n/2$ otherwise any geodesic $[a_n,b_n]$ stays outside the ball of
radius $R_n/4$ around $c_n$. Thus $2R_n\ge D_n\ge R_n/2$.

The asymptotic cone $X_\omega=\co{X,(c_n), (R_n)}$ is an Euclidean building of rank at least two by
\cite{KleinerLeeb:buildings}. By Theorem \ref{cX0}, (2), the limit points $a=(a_n)^\omega$,
$b=(b_n)^\omega$ and $c=(c_n)^\omega$ are on the horosphere $H(r_\omega)$, where $r_\omega$ is the
limit of the ray $r$; the points $a$ and $b$ are at distance $1$ from $c$. By Lemma \ref{qprop0} there
exists a path $\g$ in
 $H(r_\omega)$ connecting $a$ and $b$ and avoiding $c$. By Lemma \ref{piecewiseg}, we can assume
that $\g$ is a limit of paths $\g_n$ of lengths $O(R_n)$ connecting $a_n, b_n$ in $H(r)$ and avoiding a
ball of radius $O(R_n)$ around $c_n$. This contradicts the assumptions in the previous paragraph.
\endproof

\medskip

{\bf Proof of Theorem \ref{thq1}.} Let $X$ and $\calr$ be as in Theorem \ref{thq1}. If $\calr$ has
cardinality at least three then according to Lemma \ref{3} the set of points $\{ r(\infty )\; ; \; r\in
\calr \}$ projects onto one point on the model chamber of $\bbi X$. This implies that all horospheres
$H(r)$ with $r\in \calr $ are isometric. Let $\lambda , D$ and $L$ be the three constants provided by
Proposition \ref{qprop1} for the above family of isometric horospheres.

If $\calr$ has cardinality at most 2 then take $\lambda$ be the minimum and $D,L$ be the maximum among
the corresponding constants for these horospheres.

Since rays from $\calr$ are not parallel to a rank one factor of $X$, the horospheres corresponding to
them are $M$-bi-Lipschitz embedded into $X$ for some constant $M$ (see \cite[Theorems 1.2 and
1.3]{Drutu:Nondistorsion}).

Since the rank of $X$ is at least $2$, it has linear divergence. Let $\pgot$ be a path of length
$K\dist(a,b)$ connecting $a, b$ in $X$ and avoiding the ball $\Ball(e,R/2)$.

Let $C=12ML/\lambda$, and three points $a,b,e\in X$ where $R=\dist(e,a)=\dist(e,b)=\dist(a,b)/2$. We
want to connect $a, b$ by a path in $X'$ of length $O(\dist(a,b))$ avoiding a ball of radius $R/C$
around $e$. Let $H$ be a horosphere corresponding to some ray $r\in \calr$ crossed by $\pgot$. Let $a'$
and $b'$ be the first and last points on $\pgot\cap H$.

If $\dist(a',b')<R/4M$ then (since the distortion of $H$ is linear) $a'$ and $b'$ can be connected in
$H$ by a path of length at most $M\dist(a',b')<R/4$. That path avoids the ball of radius $R/2$ around
$e$. By replacing all such sub-paths $[a',b']$ in $\pgot$ by the corresponding paths on the
horospheres, we obtain a path of length at most $MK$ connecting $a,b$ and avoiding the ball of radius
$R/4$ around $e$.

So we can assume that $\dist(a',b')\ge R/4M$. Let $e'$ be the projection of $e$ onto $H$ (in a
CAT(0)-space, for every point $x$ and every convex set $P$, there exists at most one point $x'$ in the
set $P$ that realizes the distance from $x$ to $P$ \cite{BridsonHaefliger}). If $\dist(e,e')>R/2C$ then
any path in $H$ joining $a'$ and $b'$ avoids the ball $\Ball(e,R/2C)$.

Assume that $\dist(e,e')\leq R/2C$. Then any path joining $a',b'$ outside the ball $\Ball(e',R/C)$ is
also outside the ball $\Ball(e,R/2C)$. By Proposition \ref{qprop1} $a'$ and $b'$ can be joined in $H$
by a path of length $\leq L\dist (a',b')$ outside the ball $\Ball(e', \lambda R/4M)$. Since $\lambda
R/4M > R/C$ we are done.
\endproof

\section{$\SL_n(\OS)$}

\label{sec6}

Let $\k$ be a number field. Let $\O\subset \k$ be its ring of integers. Let $\calS$ be a finite set of
places of $\k$ containing all the archimedean ones. Let $\OS$ be the ring of $\calS$-integer points of
$\k$. Let us denote $\k_{\calS}=\prod_{\nu\in\calS} \k_{\nu}$. Note that we have a natural diagonal
embedding of $\k$ in $\k_{\calS}$. The image of $\OS$ under this embedding is a cocompact lattice. When
speaking of an action of an element of $\k$ (or more generally of a matrix over $\k$) on $\k_{\calS}$
(respectively on a vector in $\k_{\calS}^{t}$) we will be implicitly using this diagonal embedding.
 For $x\in k$ we denote
 \[| x |_\calS = \max \{ | x|_{\nu}\st \nu\in\calS\}\, .
 \]
 Similarly for vectors $v\in\kS^{t}$ we have
 \[\| v \|_\calS = \max \{ \| v\|_{\nu}\st \nu\in\calS\}\, .
 \]

\begin{theorem}\label{thm:SLd:NoCutPt}
The asymptotic cones of $\Gamma=\SL_d(\OS)$, $d\ge 3$, do not have cut\red{-}points.
\end{theorem}
\begin{remark} When $|\calS|>1$ then one may allow $d=2$ by Theorem \ref{thq1} (because in that case $\SL_2(\OS )$ is of $\Q$-rank one).
\end{remark}

We shall prove the theorem only for $d=3$. The case of $d > 3$ is similar (and easier). In fact one can
easily deduce that case from the case $d=3$ by using various embeddings of $\SL_3$ into $\SL_d$.

\subsection{Notation and terminology}
\begin{itemize}
\item
As usual, for a given set $S$ generating $\Gamma$ we shall denote by $\d_S(\cdot,\cdot)$ the word
metric on $\Gamma$ with respect to $S$.

\item
An entry $a$ of $\gamma\in\SL_d(\OS)$ is called {\em large} if
$$\log(1+|a|_\calS) \ge C \log \sqrt{|\tr\gamma^{*}\gamma|_\calS}$$ for
some fixed constant $C$.
\item
 We shall use the notation
$x\approx y$ to mean that for some constants (which we choose and fix for the given group $\Gamma$) $c_1,c_2>0$ $ c_1 \le(1+ |x|_\calS)/(1+|y|_\calS)
\le c_2$.
\item
Two elements $\alpha,\beta\in\SL_d(\OS)$ are said to be {\em ``of the
  same size''} if
$\d_S(\alpha,e)\approx \d_S(\beta,e)$
\item
Let $\kappa>0$ be fixed. Let $\gamma_1,\gamma_2\in\Gamma$. A {\em $\kappa$-exterior  trajectory} from
$\gamma_1$ to $\gamma_2$ is a path $\omega$ in the Cayley graph $\Cay(\Gamma,S)$ starting at
$\gamma_1$, ending at $\gamma_2$ such that:
\begin{enumerate}
\item
The length of $\omega$ is comparable to $\d_S(\gamma_1,\gamma_2)$, i.e. bounded by some constant (depending only on the group $\Gamma$) times
$\d_S(\gamma_1,\gamma_2)$.
\item The path $\omega$ remains outside a ball of center $e$ and radius
$\kappa\cdot \dist_S(e, \{\gamma_1, \gamma_2 \} )$.
\end{enumerate}
We shall usually omit the constant $\kappa$ and speak about {\em
  exterior trajectory} where $\kappa>0$ is implicit.
\item
Two elements $\gamma_1,\gamma_2\in \Gamma$ will be said to be {\em
  exteriorly connected} if there exist an exterior
  trajectory connecting them.
\item
We shall use the notation $\|\gamma\|_\calS=\sqrt{|\tr\gamma^{*}\gamma|_\calS}$. Note that
$\|\gamma\|_\calS\approx
  \max \{ |a_{ij}|_\calS\st \gamma=(a_{ij})\}$.
\item
Recall (cf. \cite{LMR:Cyclic},\cite{LMR:metrics}) that $\d_S(\gamma,e)\approx \log\|\gamma\|_\calS$.
\item
Let us denote the following subgroups of $\SL_3$:
$$L=\left\{\mat10*01*001\right\}
, N=\left\{\mat1**01*001\right\} , M=\left\{\mat100*10*01\right\} .$$
\end{itemize}

\subsection{Some facts about $\OS$}

Abusing notation, we write $\ZZ_{\calS}$ for the localization of $\ZZ$ with respect to the non
archimedean places in $\calS$ restricted to $\QQ$.

The next lemma contains well known facts about $\OS$ and its ideals.

\begin{lemma}\label{lmo1} \cite{AtiyahMacdonald}
\begin{itemize}
\item[(i)] The ring $\OS$ is a finitely generated $\ZZ_{\calS}$ module.
\item[(ii)] Every non-zero ideal of $\OS$ is a unique product of prime ideals, it is contained in finitely many prime ideals.
\end{itemize}
\end{lemma}

\begin{lemma}\label{lmo3} There exists a constant $C>1$ depending only on $\O$ such that every principal ideal $I$ of $\O$ of norm $k$  is generated by an element $x$ of absolute value $ |x|\le C(k)^{\frac1r}$, where $r$ is the number of archimedean valuations of $\O$. In case $r>1$ we actually have that for
 every $k'\ge k$, $I$  is generated by an element $x$ of absolute value $\frac1C(k')^{\frac1r}\le |x|\le C(k')^{\frac1r}$.

\end{lemma}

\proof Let $r$ be the number of archimedean valuations of $\O$. We distinguish the case where $r=1$
from $r>1$. When $r=1$ we have only  finitely many choices of a generator for the given principal ideal
and all satisfy the assertion. Assume $r>1$. Consider the logarithmic map $\phi\colon \O\backslash
\{0\}\to \RR^r$ that takes $x$ to $(\log |\nu_1(x)|,\ldots,\log |\nu_r(x)|)$. By the Dirichlet theorem,
the image of the group of units $\O^*$ is a lattice in the subspace of $\RR^r$ given by the equation
$x_1+\ldots+x_r=0$. Let $c\O$ be a principal ideal of $\O$. Then the image under $\phi$ of the set of
all generators of $c\O$ is the coset $\phi(c)+\phi(\O^*)$.

Let $\Delta$ be the (closed) fundamental domain of the lattice $\phi(\O^*)$ containing the point on the
diagonal $x_1=x_2=...=x_r$. Let $c'$ be a generator of $c\O$ such that  $\phi(c')$ is in $\Delta$. Then
the difference between the maximal and minimal values of coordinates of $\phi(c')$ does not exceed a
constant $\lambda$ depending only on $\O$ (we can take $\lambda$ to be the diameter of $\Delta$).
Therefore for every $i,j$ between $1$ and $r$, $\frac{|\nu_i(c')|}{|\nu_j(c')|}\le \exp(\lambda)$.
Hence $|c'|_\calS=\max\{|\nu_i(c')|, i=1,\ldots,r\}$ does not exceed $(|\nu_1(c')|\cdots
|\nu_r(c')|)^{\frac{1}{r}}\exp\lambda\le k^{\frac1r}\exp\lambda$ and is at least
$\exp(-\lambda)k^{\frac1r}$.

Suppose now that $k'>k$. Let $|\nu_i(c')|$ be the maximal number among all $|\nu_j(c')|$. By
Dirichlet's theorem, there exists a unit $\epsilon$ of $\O$ (depending only on $\O$) such that
$|\nu_i(\epsilon)|>1$ is the maximal number among all $|\nu_j(\epsilon)|$. Then there exists a constant
$C$ and an integer $u>0$ such that $c''=\epsilon^uc'$ satisfies the desired inequalities
$$\frac{1}{C}(k')^{\frac1r} \le |c''|_\calS\le C(k')^{\frac1r}.$$
\endproof

\begin{lemma}\label{lmo4} Let $a,b,c\in \OS$, $c\neq 0$, and let $P_1,...,P_s$ be distinct prime ideals not containing $a\OS$, but s.t. $c\in P_1,\dots,P_s$. Then there exists $m\in \OS$ such that $b+ma$ is not contained in $P_1,...,P_s$ and $|m|_S$ is bounded by a polynomial in $|a|_\calS, |b|_\calS$, $|c|_\calS$ (the polynomial depends only on $\OS$).
\end{lemma}

\proof  Let $h$ be the class number of $K$. Without loss of generality assume that ideals $P_1,...,P_u$
do not contain $b$ but $P_{u+1},...,P_s$ contain $b$. Since $c\in P_1\cdots P_u$, the norm of
$P_1\cdots P_u$ is smaller than the norm of the ideal $c\OS$ which is bounded by $O(|c|_\calS^r)$ where
$r$ is the number of valuations in $S$.

Then the ideal $P=(P_1\cdots P_u)^h$ is principal and its norm is bounded by $O(|c|_S^{rh})$. The
intersection $P'=P\cap \O$ is also a principal ideal of $\O$ with the same norm as $P$. By Lemma
\ref{lmo3}, there exists a generator $m$ of $P'$ with $|m|_S$ bounded by $O(|c|_S^{rh}) $. This element
$m$ generates the ideal $P$ as well.

If $1\le i\le u$, then $m\in P_i$ but $b\not\in P_i$, hence $b+ma\not\in P_i$. If $u<i\le s$, then
$b\in P_i$ but neither $m$ nor $a$ is in $P_i$, so $ma\not\in P_i$, hence $b+ma\not\in P_i$. So $b+ma$
is not in $P_i$ for any $i$.
\endproof

\begin{lemma}[Effective stable range]\label{lmo2}
For every three elements $a, b, c\in \OS$ such that $a\OS+b\OS+c\OS=\OS$ there exist two elements $k,
m\in \OS$ such that $(ma+b)\OS+(ka+c)\OS=\OS$ and the absolute values $|k|_\calS, |m|_\calS$ are
bounded by a polynomial in $|a|_\calS, |b|_\calS, |c|_S$ (for some polynomial depending only on $\OS$).
\end{lemma}

\proof We can assume that $b, c\ne 0$. Let $P_1,...,P_t$ be all the prime ideals containing $c\OS$. Let
$P_1,...,P_s$ be the ideals that do not contain $a\OS$, $P_{s+1},...,P_t$ be the ideals containing
$a\OS$. By Lemma \ref{lmo4} we can find $m$  such that $(b+ma)\OS$ is not contained in $P_1,...,P_s$
and $|m|_S$ is polynomially bounded in terms of $|a|_S$, $|b|_{\calS}$ and $|c|_S$. We claim that
$c\OS+(ma+b)\OS=\OS$. Indeed, suppose that $c\OS+(ma+b)\OS\ne \OS$. Then there exists a prime ideal $P$
containing $c\OS+(ma+b)\OS$. Since $c\in P$, that prime ideal must be one of the $P_i$, $i=1,...,t$.
Since $ma+b$ is not in $P_1,...,P_s$, we have $i>s$. Then $P$ contains $a$, so $P$ contains $a,b,c$
which contradicts the equality $a\OS+b\OS+c\OS=\OS$. Hence $c\OS+(ma+b)\OS=\OS$.
\endproof

\begin{remark} The proof of Lemma \ref{lmo2} shows that one can always take $k=0$ unless $c=0$ in which case we can take  $k=1$.
\end{remark}

\begin{lemma} \label{module} For every $a, c\in \OS$ there exists an element $a'\in a+c\OS$ with $|a'|_\calS=O(|c|_\calS)$.
\end{lemma}

\proof Let $K_\calS=\bigoplus_{\nu\in \calS} K_\nu$. Fix a closed fundamental set (parallelepiped)
$\mathcal P$ for the lattice $\OS< K_\calS$. Observe that the set $c {\mathcal P}$ contains a full set
of representatives for $\OS/c\OS$. Hence we can find an element $a'\in a+c\OS$ such that $|a'|_\calS
\le |c|_\calS \red{\sup }_{x\in {\mathcal P}} |x|_\calS$.
\endproof


\subsection{Choice of a generating set.}\label{sub:sec:generators}
\label{ssb:generators} Since all the Cayley graphs $\Cay(\SL_3(\OS),T)$ for various finite generating
sets $T$  are quasi-isometric, it will be convenient in the argument to have a sufficiently rich
generating set.

\subsubsection{}
\label{gen12}
 Fix a finite set of generators $e_{\ell}$ , $1\le \ell \le N$, of the $\ZZ$-module $\O$.

Let
\[
S_0= \left\{ E_{i,j}(e_{\ell})\st i\neq j \in\{1,2,3\}\ , 1\le \ell\le N\right\}\, ,
\]
\[
S_{1}= \left\{ \left(\begin{array}{ccc}
 p & 0 & 0\\
 0 & p^{-1}  & 0 \\
 0  & 0 & 1
\end{array}
\right)^{\pm 1}, \left(\begin{array}{ccc}
 p &  0 & 0 \\
0 &1& 0\\
 0  & 0 & p^{-1}  \end{array}
\right)^{\pm 1}, \left(\begin{array}{ccc}
 1 & 0  & 0\\
 0 &p & 0    \\
 0  & 0 & p^{-1}
\end{array}
\right)^{\pm 1} \right\}\, ,
\]
where  $p$ ranges over the primes in $\calS$ restricted to $\ZZ$,
\[
S_{2}=\left\{ \left(
\begin{array}{ccc}
 1 & 0  & 0 \\
 0  & 0  & 1\\
 0  & -1 & 0
\end{array}
\right)^{\pm 1} , \left(
\begin{array}{ccc}
 0  & 1 & 0 \\
 -1& 0  & 0 \\
  0 & 0  & 1
\end{array}
\right)^{\pm 1} , \left(
\begin{array}{ccc}
  0 & 0  & 1\\
  0 & 1 & 0 \\
 -1& 0  & 0
\end{array}
\right)^{\pm 1} \right\}\, .
\]

\subsubsection{}
\label{gen34}

Recall that by the Dirichlet unit theorem the group of units of $\OS$ is a product of some finite \red{Abelian} group and $t=r_1+r_2-1 + r_3$
 cyclic groups, where $r_1$ is the number of places $\nu\in\calS$
 such that $\k_\nu=\RR$, $r_2$ is the number of places $\nu\in\calS$
 such that $\k_\nu=\CC$
 and $r_3$ is the number of non archimedean places in $\calS$. In case $t=0$, i.e., $\OS=\ZZ$, we shall choose a hyperbolic matrix $A\in\SL_{2}(\ZZ)$,
  and denote $\A=\{ A\}$.
 When $t>0$ let $\lambda_i$, $1\le i\le t$, be fixed generators of this product
 of cyclic groups. Let us denote by $\A$ the following set of
 elements
 \[
 \A=\left\{A(\lambda_i)=\left(\begin{array}{cc} \lambda_i & 0 \\ 0 &
 \lambda_i^{-1}\end{array}\right) \st 1\le i \le t\right\}\, .
 \]
 Fix some constant $0< c < 1$. For each $\nu\in \calS$ let us denote by $NC_\nu(\A)$ the
 following  set:
 \[
 NC_\nu(\A)= \{v\in \k_\nu^{2} \st \|  v A\|_{\nu}
> c \|v\|_{\nu} \ \forall A\in\langle\A\rangle\}\cup \{0\}
\]
Observe that for an appropriate $0<c<1$ we can find finitely many elements $\gamma_{i}$, $1\le i\le M$,
such that if we denote $\A^{\gamma_{i}}=\{\gamma_{i}^{-1}A\gamma_{i} \st A\in\A\} $ then
$\bigcup_{i=1}^{M}NC_\nu(\A^{\gamma_{i}}) = \k_\nu^2$, for each $\nu\in\calS$.

 Moreover, we may choose enough $\gamma_{i}$'s so
that for any line in $\k_\nu^2$ we will have some $\A^{\gamma_{i}}$ so that each of its eigenspaces
form an angle  $\pi/3 < \psi < 2\pi/3$ with the given line. Let
\[
S_{3}=\left\{\left(
\begin{array}{cc}
A^{\pm 1}& \begin{array}{c} 0\\ 0 \end{array} \\
\begin{array}{cc}0 & 0\end{array} & 1
\end{array}
\right), \left(
\begin{array}{cc}
1& \begin{array}{cc} 0 & 0\end{array} \\
\begin{array}{c} 0 \\ 0 \end{array} & A^{\pm 1}
\end{array}
\right) \st A\in \A^{\gamma_{i}}, \ 1\le i \le M \right\}\, .
\]

Let $S_4$ be a (finite) set of generators of $\SL_2(\OS)$ embedded into $\SL_3(\OS)$ as the lower right
corner.

\subsubsection{}\label{SFiveSix}
If $K= \QQ[\sqrt{-d}]$ for some $d\ge0$ then we also define $S_5$ and $S_6$  as follows (if $K$ is not
of this form then $S_5=S_6=\emptyset$). Fix a geometrically finite fundamental domain $\mathcal F$ of
the action of $\SL_2(\OS)$ on the corresponding symmetric space, which is the hyperbolic space $\HH^n$
of dimension $n=2,3$, note that in this case $\SL_2(\OS)$ is of real rank 1. For each face of $\mathcal
F$  we include in $S_5$ a generator taking $\mathcal F$ to the neighboring domain. Let
$\Omega_0,...,\Omega_\kappa$ be the points at infinity of $\mathcal F$. Let $P_0,...,P_\kappa$ be the
stabilizers of $\Omega_0,\ldots,\Omega_\kappa$ in $\SL_2(K)$. We assume that $P_0$ is the group of
upper triangular matrices in $\SL_2(K)$. Then for every $i=1,...,\kappa$ there exist $\alpha_i\in
\SL_2(K)$ conjugating $P_0\cap \SL_2(\OS)$ to $P_i\cap \SL_2(\OS)$. Let $\alpha_0= e$. In order to
define $S_6$, we need the following statement.

\begin{lemma}\label{t} There exist a finite set of matrices ${\mathcal T}= \left\{T_1,\ldots, T_\iota\right\}\subset \SL_3(\OS)$ such that for every $L\in \SL_3(\OS)$ there exists $T_j$,$1\le j\le\iota$, such that $\alpha_i LT_j\alpha_i^{-1}\in \SL_3(\OS)$ for every $i=0,\ldots,\kappa$. We may also require that the identity belongs to $\mathcal T$.
\end{lemma}

\proof Every element of $K$ is a fraction with numerator and denominator from $\OS$. Consider the
generic matrix $X=\mt{x_{11}}{x_{12}}{x_{13}}{x_{21}}{x_{22}}{x_{23}}{x_{31}}{x_{32}}{x_{33}}$ and the
matrices $\alpha_iX\alpha_i^{-1}$. The entries of these matrices are linear polynomials in $x_{11},...,
x_{33}$ with coefficients from $K$. Let $D$ be the product of denominators of all these coefficients.
Then for every $i$ and every $L\in \SL_3(\OS)$ the matrix $\alpha_i(1+DL)\alpha_i^{-1}$ belongs to
$\SL_3(\OS)$. Now it is enough to take $T_1,...,T_\iota$ to be representatives of the right cosets of
the congruence subgroup of $\SL_3(\OS)$ corresponding to $D$.
\endproof

\subsubsection{}\label{SSix}
 Now $S_6$ is defined by:
 \[
 S_6 = \left\{ \alpha_i T^{-1} sT'\alpha_i^{-1} \st  0\le i \le \kappa, \
 T, T'\in {\mathcal T}, \ s\in \bigcup\limits_{m=0}^{5} S_m
  \right\}\cap \SL_3(\OS)\, .
  \]
  Note that for any choice of $T\in {\mathcal T}$
  and $s\in \bigcup\limits_{m=0}^{5} S_m$ there exists at least one $T'$
  so that the corresponding element $\alpha_i T^{-1} sT'\alpha_i^{-1}$ belongs to $\SL_3(\OS)$.

Let us fix $S=S_0 \cup ...\cup S_{6}$ as the set of generators of $\SL_3(\OS)$.

\subsection{Proof of Theorem~\ref{thm:SLd:NoCutPt}}

In order to show that asymptotic cones of $\Gamma=\SL_3(\OS)$ do not have cut\red{-}points it suffices to
show that any two elements $\alpha,\beta\in\Gamma$ are exteriorly connected. This follows immediately
from the following two lemmas:

\begin{lemma}\label{lem:gamma:EC:u}
Let $\gamma\in\SL_3(\OS)$. There exists $\alpha\in M= \left(\begin{array}{ccc}
1 & 0  & 0 \\
* & 1 & 0 \\
* & 0  & 1
\end{array}\right)$
such that
\begin{enumerate}
\item
$\d_S(\alpha,e)\approx \d_S(\gamma,e)$, that is $\gamma$ and $\alpha$ are approximately of the same
size.
\item
$\gamma$ is exteriorly connected to $\alpha$.
\end{enumerate}
\end{lemma}

\begin{lemma}\label{lem:u:EC:u}
Given any $\alpha=\left(\begin{array}{ccc}
1 & 0  & 0 \\
u_1 & 1 & 0 \\
u_2   & 0  & 1
\end{array}\right)$ , $\beta=\left(\begin{array}{ccc}
1 & 0  & 0 \\
v_1 & 1 & 0 \\
v_2   & 0  & 1
\end{array}\right)$ with $u_i,v_i\in \OS$ there exists an exterior
trajectory connecting them.
\end{lemma}

A basic tool in proving these lemmas is:

\begin{lemma}\label{lem:gamma:EC:gammaL}
Let $\gamma\in\SL_3(\OS)$ be an element having some large entry in the first column. For any
$\theta=\left(\begin{array}{ccc} 1 & 0 & m\\ 0 & 1 & n \\ 0 & 0 & 1 \end{array}\right)\in L$, $m,n\in
\OS$ there exists an exterior trajectory from $\gamma$ to $\gamma\theta$.
\end{lemma}

\begin{proof}
A path $\omega$ of length $k$ in the Cayley graph connecting $\gamma$ to $\gamma\theta$ corresponds to
a word $\theta=s_1 s_2\dots s_k$ where each $s_i\in S$,  $\omega(i)=\gamma s_1 s_2\dots s_i$, $0\le i
\le k$. Since we want the path to be an exterior trajectory it should satisfy the following conditions:
\begin{enumerate}
\item[$(E_1)$]
$k=\length(\omega) \approx \d_S(\theta,e)$
\item[$(E_2)$]
$\forall 0\le i \le k$, $\d_S(\omega(i),e)\ge \kappa \d_S(\gamma,e)$ for some constant
$\kappa=\kappa(\Gamma)$. (Note that we have that $\d_S(\gamma\theta,e) \ge C\d_S(\gamma,e)$ for some constant $C$
which depends on our notion of an ``entry being large'').
\end{enumerate}

In \cite{LMR:Cyclic} it was shown, in the particular case where $m,n\in \ZZ$ how to construct for any
$\theta=\left(\begin{array}{ccc} 1 & 0 & m\\0 & 1 & n \\ 0 & 0 & 1 \end{array}\right)$ a word of length
$O(\log(n^2+m^2+1))\approx \d_S(\theta,e)$ expressing it in terms of a given generating set. Let us
describe the slightly modified
argument for elements  $\theta=\left(\begin{array}{ccc} 1 & 0 & m\\
0 & 1 & n \\ 0 & 0 & 1 \end{array}\right)\in \SL_3(\OS)$, see also (2.12) in \cite{LMR:metrics}.

Let $H=\langle \A \rangle$ be the subgroup generated by the set $\A$ defined in \ref{gen34}. That is,
either $H$ is a cyclic group generated by some hyperbolic matrix $A\in\SL_2(\ZZ)$ when $\OS=\ZZ$, or
otherwise $$H=\left\{ \left(\begin{array}{cc} \lambda& 0 \\ 0 & \lambda^{-1}\end{array}\right) \st
\lambda \in \OSS\right\}$$ where $\OSS$ is the group of units of $\OS$.
 Consider the group
$\Lambda=H\ltimes\OST$. $\Lambda$ is a finitely generated group which is a cocompact lattice in the
group $H \ltimes \prod_{\nu\in \calS}K_\nu^2$ where $H$ acts on each of the factors of $\prod_{\nu\in
S}K_\nu^2$ via the corresponding embedding of $\OS\in K$ into $K_\nu$. Observe that \cite[$\S 3.15-3.18
$]{LMR:metrics} for each $\nu\in \calS$ the two dimensional vector space $K_\nu^2$ is spanned by
eigenspaces on which $H$ acts with eigenvalues of absolute value strictly bigger than $1$. It follows
as in \cite[section 3]{LMR:metrics} that the restriction to $K_\nu^2$ of the left invariant coarse path
metric on $H \ltimes \prod_{\nu\in S}K_\nu^2$ is such that the distance from the identity of
$\left(\begin{array}{ccc} 1 & 0 & x\\ 0 & 1 & y \\ 0 & 0 & 1
\end{array}\right)$  where $x,y\in K_\nu$ is
$O(\log(|x|_\calS^2+|y|_\calS^2+1))$. If we fix any $\gamma\in\Gamma$, then any  $\theta=\mt10m01n001$
with $m,n\in \OS$ can be expressed as a word  $s_1 s_2 \dots s_k$ of length
$k=O(\log(|n|_\calS^2+|m|_\calS^2+1))$ with respect to a set of generators of the form
\[
S(\gamma_\ell)=\left\{ \mt10{e_j}010001^{\pm 1} , \mt10001{e_j}001^{\pm 1} , \left(
\begin{array}{cc} A  & \begin{array}{c} 0\\ 0\end{array}\\
\begin{array}{cc} 0 & 0 \end{array} & 1\end{array}
\right)^{\pm 1} \right\}
\]
where $\{e_j\}$
 is a finite set generating $\O$ as a $\ZZ$-module
 and $\{A \in \A^{\gamma_\ell}\}$ where $\A$ is as in
 subsection~\ref{gen34} and $1\le \ell \le M$
(see the definition of $S_3$). In particular we have for any $1\le i \le k$ that
\[
s_1 s_2 \dots s_i\in L(H) =
 \left\{\left(\begin{array}{cc}
A    & {\begin{array}{c} s \\ t\end{array}}\\
{\begin{array}{cc}0 & 0\end{array}} & 1
\end{array}
\right) \st A\in H^{\gamma_\ell} \ s,t\in\OS\right\}.
\]
We recall that in our choice of a generating set  $S$ for $\SL_3(\OS)$ we have given ourselves several
possible choices of generators using various (finitely many) conjugates of $H$.

Let us denote $\gamma = \left(\begin{array}{ccc}
a_{11} & a_{12} & a_{13} \\
a_{21} & a_{22} & a_{23} \\
a_{31} & a_{32} & a_{33}
\end{array}
\right)$. By our assumption for some $1\le j\le 3$ we have $a_{j1}$ large. Let $\nu_0\in \calS$ be such
that $|a_{j1}|_{\nu_0} > c_1 |a_{j1}|_\calS$ for some fixed $c_1>0$ depending on $K$ and $\calS$. By
the choice of $S$ in section~\ref{sub:sec:generators} there is some $\gamma_\ell$, $1\le \ell \le M$,
so that $(a_{j1}, a_{j2})\in NC_{\nu_0}(\A^{\gamma_{\ell}})$. We shall use $\A^{\gamma_\ell}$ for
producing a short word representing the element $\theta$. Note that for any $A\in H^{\gamma_{\ell}}$
 $\|(a_{j1}, a_{j2}) A\|_{\nu_0} \ge c_0
\|(a_{j1},a_{j2})\|$
 where $c_0$ is the constant in \ref{sub:sec:generators}. This
 immediately implies that if we use $S(A)\subset S$ for expressing
 $\theta$ then ($E_2$) is satisfied.
\end{proof}
We turn now to the proof of Lemma~\ref{lem:gamma:EC:u}. As we use the right Cayley graph, applying
elements of $S$ on the right corresponds to column operations.

\begin{proof} {\em of Lemma~\ref{lem:gamma:EC:u}.}

\noindent{\bf Step 1.} $\gamma$ is  exteriorly connected to some $\gamma_1=\gamma s$ such that
$\gamma_1$ has  a large
 entry in the first column and $s\in \left\{I,
\left(
\begin{array}{ccc}
  0 & 1 &0  \\
 -1&  0 & 0 \\
  0 &  0 & 1
\end{array}
\right)^{\pm 1} , \left(
\begin{array}{ccc}
  0 & 0  & 1\\
  0 & 1 & 0 \\
 -1&  0 & 0
\end{array}
\right)^{\pm 1} \right\}. $

\noindent{\bf Proof of  1.} If there is no large entry in the first column we can exchange columns (and
reverse sign to keep the determinant $1$). We can also assume (for the sake of simplicity) that the
(1,1)-entry of the matrix is large.

\noindent{\bf Step 2.} $\gamma_1=\left(\begin{array}{ccc}
a & b & c \\
* & * & * \\
* & * & *
\end{array}\right)$
is exteriorly connected to $\gamma_2= \gamma_1 \left(\begin{array}{ccc}
1 & m & k \\
0 & 1 & 0 \\
0 & 0 & 1
\end{array}\right)=
\left(\begin{array}{ccc}
a & b'' & c' \\
* & * & * \\
* & * & *
\end{array}\right)$ so that $b''\OS + c'\OS = \OS$, with $m$ and $n$ polynomially controlled by the
size of $\gamma_1$.

\noindent{\bf Proof of  2.} This follows from Lemma~\ref{lmo2} and Lemma~\ref{lem:gamma:EC:gammaL}
using:
\[
\left(\begin{array}{ccc}
1 & m & k\\
0 & 1 & 0 \\
0 &  0 & 1
\end{array}\right)=
\left(\begin{array}{ccc}
1 & 0 & k\\
0 & 1 & 0 \\
0 & 0  & 1
\end{array}\right)
\left(\begin{array}{ccc}
1 & 0 & 0\\
0 & 0 & -1 \\
0 & 1 & 0
\end{array}\right)
\left(\begin{array}{ccc}
1 & 0 & -m\\
0 & 1 & 0\\
0 & 0 & 1
\end{array}\right)
\left(\begin{array}{ccc}
1 & 0 & 0\\
0 & 0 & 1 \\
0 & -1& 0
\end{array}\right).
\]

\noindent{\bf Step 3.} If $c'=0$ we may jump to step 7 (with clear changes of notation). Otherwise
$\gamma_{2}= \left(\begin{array}{ccc}
a & b'' & c' \\
* & * & * \\
* & * & *
\end{array}\right)$ is exteriorly connected to $\gamma_{3}=\gamma_{2}
\left(\begin{array}{ccc}
1 & 0 & 0\\
0 & 1 & 0 \\
0 & k & 1
\end{array}\right)=
\left(\begin{array}{ccc}
a & b' & c' \\
* & * & * \\
* & * & *
\end{array}\right)$ so that $b'$ is large, and $k$ is polynomially controlled by the size of $\gamma_2$.

\noindent{\bf Proof of  3.} When $c'\neq 0$ we can find $k\in OS$ so that $b'=b''+kc'$ is large and the
size of $k$ is polynomially controlled by the size of $\gamma$. Then apply
Lemma~\ref{lem:gamma:EC:gammaL}, where we use: $\left(\begin{array}{ccc}
1 & 0 & 0\\
0 & 1 & 0 \\
0 & k & 1
\end{array}\right)=
\left(\begin{array}{ccc}
1 & 0 & 0\\
0 & 0 & -1 \\
0 & 1 & 0
\end{array}\right)
\left(\begin{array}{ccc}
1 & 0 & 0\\
0 & 1 & -k \\
0 & 0 & 1
\end{array}\right)
\left(\begin{array}{ccc}
1 & 0 & 0\\
0 & 0 & 1 \\
0 & -1& 0
\end{array}\right)
$.

\noindent{\bf Step 4.} $\gamma_{3}= \left(\begin{array}{ccc}
a & b' & c' \\
* & * & * \\
* & * & *
\end{array}\right)$ is exteriorly connected to $\gamma_4=  \left(\begin{array}{ccc}
b' & -a & c' \\
* & * & * \\
* & * & *
\end{array}\right) = \gamma_3  \left(\begin{array}{ccc}
0 & -1 & 0 \\
1 & 0 & 0 \\
0  & 0 & 1
\end{array}\right)$.

\noindent{\bf Step 5.} $\gamma_4=  \left(\begin{array}{ccc}
b' & -a & c' \\
* & * & * \\
* & * & *
\end{array}\right)$ is  exteriorly connected to $\gamma_5=\gamma_4
\left(\begin{array}{ccc}
1 & u  & 0  \\
0  & 1 & 0  \\
0  & v & 1
\end{array}\right)=
 \left(\begin{array}{ccc}
b' & 1 & c' \\
* & * & * \\
* & * & *
\end{array}\right)$

\noindent{\bf Proof of  5.} Since $\OS= b' \OS + c'\OS$ we have $u',v' \in \OS$ so that $b'u'+c'v'=
a+1$. Observe that for any $u=u'+c'r\in u' +c'\OS$ we have $v=v'- b'r\in\OS$ such that $b'u+c'v=a+1$.
By Lemma~\ref{module} one can choose $u\in u'+c'\OS$ which satisfies $|u|_\calS\le const. |c'|_\calS \,
$, where the constant depends only on $\OS$.

\noindent{\bf Step 6.}
 $\gamma_5=
 \left(\begin{array}{ccc}
b' & 1 & c' \\
* & * & * \\
* & * & *
\end{array}\right)$
is  exteriorly connected to $\gamma_6=\gamma_5 \left(\begin{array}{ccc}
1 & 0  & 0  \\
0  & 1 & -c'  \\
0  & 0  & 1
\end{array}\right) = \left(\begin{array}{ccc}
b' & 1 & 0 \\
* & * & * \\
* & * & *
\end{array}\right).$

\noindent{\bf Proof of 6.} It follows from Lemma~\ref{lem:gamma:EC:gammaL}.

\medskip

\noindent{\bf Step 7.} In step 8 we are going to subtract $b'-1$ times the second column from the
first. To insure that the first column remains large, we first note that  there exists some $x\in \OS$
whose size is polynomially bounded in terms of the size of $\gamma_6$ and such that $\gamma_6$ is
exteriorly connected to $\gamma_7=\gamma_6 \left(
\begin{array}{ccc}
 1 & 0  & 0 \\
 0  & 1 & 0 \\
 0  & x  & 1
\end{array}
\right)$ so that both $\gamma_7$ as well as $\gamma_8=\gamma_7 \left(
\begin{array}{ccc}
 1 & 0  & 0 \\
 1-b'  & 1 & 0 \\
 0 & 0  & 1
\end{array}
\right)$ have large first column.

\noindent{\bf Step 8.} $\gamma_7$ is  exteriorly connected to $\gamma_8=\gamma_7 \left(
\begin{array}{ccc}
 1 & 0  & 0 \\
 1-b'  & 1 & 0 \\
 0 &  0 & 1
\end{array}
\right) =  \left(\begin{array}{ccc}
1 & 1 & 0 \\
* & * & * \\
* & * & *
\end{array}\right)$.

\noindent{\bf Step 9.} $\gamma_8=  \left(\begin{array}{ccc}
1 & 1 & 0 \\
* & * & * \\
* & * & *
\end{array}\right)$ is  exteriorly connected to $$\gamma_9= \gamma_8 \left(
\begin{array}{ccc}
 1 & -1  & 0 \\
 0  & 1 & 0 \\
 0 &  0 & 1
\end{array}
\right) =  \left(\begin{array}{cc}
1 & \begin{array}{cc}0 & 0\end{array} \\
\begin{array}{c} x \\ y\end{array} & B \\
\end{array}\right).$$

\noindent{\bf Step 10.} $\gamma_9$ is  exteriorly connected to $\gamma_{10}= \gamma_9
\left(\begin{array}{cc}
1 & \begin{array}{cc} 0 & 0\end{array} \\
\begin{array}{c} 0 \\ 0 \end{array} & B^{-1}
\end{array}\right)=\mat100x10y01$.


\noindent{\bf Proof of  10.} Let us distinguish two cases:

(i) The rank of $\SL_2(\OS)$ is at least $2$.

(ii) The rank of $\SL_2(\OS)$ is $1$, i.e., $\OS=\QQ(\sqrt{-d})$ for some integer $d \ge 0$,
$\calS=\{\infty\}$.

\noindent We shall abuse terminology and identify $\SL_2$ with its image under embedding into $\SL_3$
as the lower right corner. In the case when $\rank(\SL_2(\OS))\ge 2$ we may express (by the results of
\cite{LMR:metrics}) $\left(\begin{array}{cc}
1 & \begin{array}{cc} 0 & 0\end{array} \\
\begin{array}{c} 0 \\ 0 \end{array} & B^{-1}
\end{array}\right)$ as a short word $s_1 s_2\dots s_n$ using a subset of the
generating set which generates $\LRMat{\SL_2(\OS)}$. For any prefix of this word, we have that the
first column of $\gamma_9 s_1 s_2\dots s_i$ is the same as that of $\gamma_9$.

\medskip

In the case when $\rank(\SL_2(\OS))=1$ we may write $B^{-1}$ as a short word in terms of our generating
set using the following procedure.
 Consider the symmetric space $X=\HH^n$, $n=2,3$ associated with our rank $1$ group $\SL_2(K_\infty)$
 tesselated by fundamental domains for the action of $\SL_2(\OS)$.
We may choose an $\SL_2(\OS)$-invariant collection of horoballs so that  if one removes them from $X$
the resulting subset $X_0$ has a compact quotient modulo $\SL_2(\OS)$. Fix a point $O\in X_0$ and
consider the geodesic $\g$ connecting $O$ to $B^{-1} O$. Following that geodesic we obtain a word $s_1
s_2 \dots s_t$ in the generators $S_5$ (see Section~\ref{sub:sec:generators}) expressing $B^{-1}$.
Combining letters corresponding to parts of the geodesic spent inside various horoballs into sub-words,
we obtain a word of the form $W_1 W_2\dots W_k$, where each $W_i$ is either one of the generators $s_j$
corresponding to the geodesic $\g$ passing between domains outside the collection of horoballs, or
$W_i$ corresponds to the part spent in some horoball. Note that if we manage to replace each sub-word
$W_i$ corresponding to a horoball by a word, $w_i$, whose length is comparable to the length of the
part of the geodesic inside the horoball (namely comparable to $\log(1+\|W_i\|$) we shall obtain a word
of the required length expressing $B^{-1}$. Let us show by induction that we can indeed find for each
$W_i$ a short word $w_i$ representing it and such that the resulting trajectory does not get too close
to the identity. Suppose we have already treated $W_1,W_2,\dots,W_{m-1}$ and denote $\gamma'=\gamma_9
W_1 W_2 \dots W_{m-1}= \left(\begin{array}{cc}
1 & \begin{array}{cc}0 & 0\end{array} \\
\begin{array}{c} x \\ y\end{array} & B' \\
\end{array}\right)$. If $W_m=s_j$ for some $1\le j\le t$, i.e., it does not correspond to passing
through a horoball we let the word representing it to be simply $w_m=s_j$ and clearly the trajectory is
an exterior one (since the first column which contains a large element was not changed). Suppose $W_m$
corresponds to passing via some horoball. Let $\Omega_{k_0}$ be the cusp of the fundamental domain $\F$
corresponding to that horoball.


Conjugating by the element $\alpha = \alpha_{k_0}$ defined in \ref{SFiveSix} we have that
\[
W'_m = \alpha^{-1} W_m \alpha = \mt10001z001 .
\]
Applying lemma~\ref{lem:gamma:EC:gammaL} to $\gamma= \gamma' \alpha$ and $\theta = W'_m$ we obtain a
short word  $r_1 r_2 \dots r_\ell=W'_m$ such that the corresponding trajectory from $\gamma=\gamma'
\alpha$ to $\gamma'\alpha W'_m$ is an exterior trajectory. Notice that actually $\gamma' \alpha$ does
not belong to $\SL_3(\OS)$ but to $\SL_3(K)$ so the notion of ``exterior trajectory" should be
understood with respect to the metric induced from the Riemannian metric on the corresponding symmetric
space. Denote $T_0=e$. For each $r_i$, $1\le i\le \ell-1$,
there is $T_i\in{\mathcal T}$ so that we have a generator
$t_i=\alpha T_{i-1}^{-1} r_i T_i \alpha^{-1}\in S_6$ as in \ref{SSix}. Let $t_\ell= \alpha T_{\ell-1}^{-1}
r_\ell \alpha^{-1}$. Now observe that we have \[ W_m = \alpha W'_m \alpha^{-1} = t_1 t_2 \dots t_\ell\,
.
\]
This in particular implies that $t_\ell$ is indeed an element of $\SL_3(\OS)$ and hence $t_\ell\in
S_6$, and that we obtain an exterior trajectory connecting $\gamma_9$ to
$$\gamma_9 W_1 W_2 \dots W_{m}= \left(\begin{array}{cc}
1 & \begin{array}{cc}0 & 0\end{array} \\
\begin{array}{c} x \\ y\end{array} & B'' \\
\end{array}\right).$$
Repeating this process we shall obtain an exterior trajectory connecting $\gamma_9$ to $$\gamma_{10}=
\left(\begin{array}{cc}
1 & \begin{array}{cc}0 & 0\end{array} \\
\begin{array}{c} x \\ y\end{array} & I \\
\end{array}\right)=\mat100x10y01.$$

\noindent{\bf Remark.} Note that for each $1\le i \le 9$ the sizes of $\gamma_i$ and $\gamma_{i+1}$ are
polynomially comparable.
\end{proof}

\begin{proof}[Proof of Lemma~\ref{lem:u:EC:u}.]
We are given $\alpha=\Umat{u}$, $\beta=\Umat{v}$ where $u,v \in \OST$. Let us denote $w=v-u$. Our goal
is to produce an exterior trajectory connecting $\alpha$ to $\beta$. We shall construct a word using
generators belonging to a subset of $S_0\cup S_3$. The argument proving Lemma~\ref{lem:gamma:EC:gammaL}
allows us to produce a  short word $s_1 s_2 \dots s_n$  representing $\Umat{w}$. This gives us a path
$\p$ from $\Umat{v}$ to $\Umat{v}$ of the required length. However this path may get too close to the
origin. To avoid getting too close to the origin we shall show that we can ``shift the whole path away
from the origin". We fix an archimedean place $\nu_0$ of $K$. Look at the projection of the path $\p$
at this place in $K_{\nu_0}^2$, identified with $\Umat{K_{\nu_0}^2}$. There is a hyperbolic matrix
$A\in \SL_2(\ZZ)$ such that $\LRMat{A}\in S_3$ which has an eigen-direction so that when translating
the path $\p$ in this eigen-direction, $\p$ is moved away from the origin. Choose a word $t_1 t_2 \dots
t_m$ in the generators $\left\{ \LRMat{A}, \mat100110001, \mat100010101\right\}$ which represents an
element of the form $\mat100p10q01$  so that the vector $\left(\begin{array}{c} p \\ q
\end{array}\right)$ is close (actually within uniformly bounded distance) to the eigen-direction of $A$ and
whose size is comparable to the size of $\|u\|_\calS$, see \cite{LMR:Cyclic}. We claim that the
trajectory corresponding to:
\[
\Umat{v} = \Umat{u}  t_1 t_2 \dots t_m s_1 s_2 \dots s_n (t_1 t_2 \dots t_m)^{-1}
\]
gives a trajectory from $\Umat{u}$ to $\Umat{v}$ which never gets too close to the origin. Indeed
observe that as we go along  the path corresponding to $t_1 t_2 \dots t_m$ we get further away at the
$\nu_0$ place from the origin. We might be getting closer to the origin at some other (archimedean)
place but since the rate we move in any other archimedean place is comparable to the rate at which we
move at the $\nu_0$ place and at any non archimedean place moving along this path does not change our
distance from the origin, we conclude that we are always at a distance which is bounded below by a
fixed positive fraction of $\|u\|_\calS$. Once we have completed the path $t_1 t_2 \dots t_m$, the path
along $s_1 s_2 \dots s_n$ is away from the origin by a distance comparable to $\|u\|_\calS$ and
finally, as before, moving back on $t_1 t_2 \dots t_m$ cannot get us too close to the identity.
\end{proof}



\section{The Erratum}
\label{s:err}

The goal of this erratum is to correct Proposition 3.24, and the proofs of Theorems 4.4 and 4.9, relying on that Proposition. We thank the authors of \cite{T}, who kindly pointed out the mistake in the Proposition 3.24 to us.

In what follows we assume that we work in a fixed metric space $(X,\dist)$.

We call {\emph{quasi-geodesic segment in $X$}} a quasi-isometric embedding $\q: [a,b] \to X$, where $a<b$ are two finite real numbers. We call {\emph{bi-infinite quasi-geodesic in $X$}} (or simply {\emph{quasi-geodesic in $X$}}) a quasi-isometric embedding $\q: {\mathbb{R}} \to X$. Given a bi-infinite quasi-geodesic $\q$, its restriction to an interval (finite or infinite) is called a {\emph{sub-quasi-geodesic of $\q$}}.

A quasi-geodesic (segment) is \emph{Morse} if for every $L\ge 1$ and $ C\ge 0$, every $(L,C)$-quasi-geodesic $\pgot$ with endpoints on the image of $\q$ is contained in the $M$-tubular neighborhood of $\q \, $, where $M$ depends only on $L, C$.

The corrected version of Proposition 3.24 is given below.

\begin{proposition}\label{err:prop3}
Let $X$ be a metric space and for every pair of points $a,b\in X$ let $L(a,b)$ be a fixed set of
$(\lambda,\kappa)$-quasi-geodesics (for some constants $\lambda \geq 1$ and $\kappa \geq 0$) of endpoints $a$ and $b$.
Let $L=\bigcup_{a,b\in X} L(a,b)$.

Let $\q$ be a bi-infinite quasi-geodesic in $X$, and for every two points $x,y$ on $\q$
denote by $\q_{xy}$ the maximal sub-quasi-geodesic of $\q$ with endpoints $x$ and $y$.

The following conditions are equivalent:
\begin{itemize}
\item[(1)] In every asymptotic cone of $X$, the ultralimit of $\q$ is
either empty or contained in a transversal tree for some tree-graded structure;
\item[(2)] $\q$ is a Morse quasi-geodesic;
\item[(3)] For every $C\ge 1$ there exists $D\ge 0$ such that every path of length $\le Cn$ connecting two points $a,b$ on $\q$ at distance $\ge n$ crosses the $D$-neighborhood of the middle third of $\q_{ab}$;
\item[(4)] For every $C\geq 1$ and natural  $k>0$ there exists $D\ge 0$ such that every
 $k$-piecewise $L$ quasi-path $\pgot$ that:
     \begin{itemize}
     \item connects two points $a, b\in \q$,
     \item has quasi-length $\leq C \dist(a,b)$,
     \end{itemize}
crosses the $D$-neighborhood of the middle third of {$\q_{ab}$}.
\item[(5)] for every $C\ge 1$ and every $\epsilon>0$ there exists $D\ge 0$ such that for every $a, b\in \q$ with $\dist (a,b) \geq D$,
and every path $\p$ connecting $a,b$ of length $\le C\dist(a,b)$, the sub-quasi-geodesic
$\q_{ab}$ is contained in the $(\epsilon\dist(a,b))$-neighborhood of $\p$.
\end{itemize}
\end{proposition}

\begin{remarks}\label{r:1}
\begin{enumerate}
\item Properties (1) - (4) are as in Proposition 3.24, property (5) is modified. The initial version of (5) stated that, given an arbitrary $C\ge 1$, for every $a, b\in \q$,
and every path $\p$ with endpoints $a,b$ and of length $\le C\dist(a,b)$, the {sub-quasi-geodesic
$\q_{ab}$} would be contained in the $D$-neighborhood of $\p$, with $D$ a constant depending only on $C$. This property is however strictly stronger than the property of being a Morse quasi-geodesic. Indeed, an example of Morse geodesic that does not satisfy the property above, provided in \cite{T}, is an arbitrary bi-infinite geodesic $\q :{\mathbb{R}} \to {\mathbb H}^2$ in the hyperbolic plane ${\mathbb H}^2$: for an arbitrarily large integer $n$, the path $\p_n$ joining $\q (-n)$ to $\q (n)$ obtained as the concatenation of $\q$ restricted to the interval $[-n, -\log n]$, with half of the hyperbolic circle centered in $\q (0)$ and of radius $\log n$, and with $\q $ restricted to the interval $[ \log n , n ]$, has length at most $C n$, for some fixed constant $C$ independent of $n$, yet $\q (0)$ is between the endpoints of $\p_n$ on $\q$, and at distance $\log n$ of $\p_n$.

The mistake is in the proof of the implication $(1)\to (5)$ (see Proposition 3.24), where it is assumed that for a sequence of paths $\p_n$ connecting pairs of points $a_n,b_n$ on the bi-infinite quasi-geodesic $\q$, the $\omega$-limit of $\p_n$ in any asymptotic cone is a rectifiable path. This is not true in general: in the example above, for observation points $x_n$ coinciding with the midpoints of the half-circle, the $\omega$-limit of $\p_n$ in the asymptotic cone
$\Con^\omega(\HH^2, (x_n), (\log n))$ is not rectifiable.

If, on the other hand, one considers sequences of paths $\p_n$ of lengths $\ell_n \to \infty$ and their $\omega$-limits only in asymptotic cones with scaling sequence $\lambda_n = \ell_n$, then these $\omega$-limits, when non-empty, are rectifiable paths.

\medskip

\item The first paper to provide a correct proof of the equivalence $(2)\leftrightarrow (3)$ in Proposition \ref{prop3} is thus \cite{T}, since our initial proof of this implication relied on the wrong version of (5).
\end{enumerate}
\end{remarks}

Proposition \ref{prop3} contains all known characterizations of Morse quasi-geodesics, with one exception that we explain below, to complete the list. Moreover, this characterization plays a central part in the proof of $(2)\leftrightarrow (3)$ in \cite{T}.
We begin with some terminology.

\begin{definition} Let $\q$ be a quasi-geodesic in a metric space $X$, and let $\eta>0$. The \emph{$\eta$--nearest point projection of a point $x\in X$ on $\q$}, denoted by $\proj^\eta_\q(x)$, is the set of points $x'\in \q$ such that $\dist(x,x')\leq \dist(x,\q) + \eta$.
\end{definition}

\begin{definition} [Sublinear contraction \cite{A}] We say that a quasi-geodesic $\q$ in a geodesic metric space $X$ is \emph{uniformly sublinearly contracting} if there exists some constant $\eta >0$ such that for every sub-quasi-geodesic $\q'$ of $\q$, the projection $\proj^\eta_{\q'}$ is uniformly sublinearly contracting: for every $\epsilon>0$ there exists $D= D(\epsilon )$ (independent of the specific sub-quasi-geodesic $\q'$ of $\q$) such that for every $D'\geq D$ and every $x\in X$ with $\dist(x,\q')\geq 2D'$, the union of all nearest point projections $\proj^\eta_{\q'}(y)$ of points $y$ in the ball $\Ball(x,D')$ has diameter at most $\epsilon D'$.
\end{definition}

In what follows, in all arguments using uniform sublinearly contracting properties, we drop the parameter $\eta$ from the notation.

\begin{theorem}[Theorem 1.4, \cite{A}]\label{thm:A}
Let $X$ be a geodesic metric space, and let $\q$ be a quasi-geodesic in it. The following are equivalent:

\begin{enumerate}

\item $\q$ is Morse;

\item $\q$ is uniformly sublinearly contracting.

\end{enumerate}
\end{theorem}


\begin{remark}\label{r:gr}
If $X$ is $\delta$-hyperbolic then every bi-infinite geodesic $\mathfrak g$ of $X$ is uniformly sublinearly contracting, and an even stronger property holds: there exist constants $L=L(\delta)$ and $M=M(\delta )$ such that for every  $\mathfrak g$ geodesic (segment, infinite or bi-infinite), every $D>0$, and every $x\in X$ with $\dist(x,\mathfrak{g})>D+M$, the union of nearest point projections $\proj_{\mathfrak g}$ of points $y$ of the ball $\Ball(x,D)$ to $\mathfrak g$ has diameter at most $L$ ( \cite{Gromov:hyperbolic}, \cite[Proposition 2.1 in Chapter 10]{CDP}).

This clearly extends to $(\lambda ,\kappa )$-quasi-geodesics, modulo increasing the constants $M$ and $L$ with additive constants depending on $(L,C)$.
\end{remark}

\noindent \textit{Proof of Proposition \ref{prop3}.} The equivalence of the five properties follows from the implications $(2)\to (3)\to (4)\to (1)\to (2)$ and the equivalence $(1)\leftrightarrow (5)$.

For $(2)\to (3)$ we refer to \cite{T}.

$(3)\to (4)$ is obvious. The proof of $(4)\to (1)$ is correct.

$(1)\to (2)$. Suppose that there exists $\mu\ge 1, \nu\ge 0$ such that for every $k>1$ there exists a  $(\mu,\nu)$-quasi-geodesic $\p_k$ joining two points on $\q$ and there exists $x_k\in \p_k$ at distance $d_k>k$ from $\q$. We can assume that $d_k$ is the maximal distance from a point of $\p_k$ to $\q$. For any ultrafilter $\omega$, in the asymptotic cone $\Con^\omega(X, (x_k), (d_k))$, the $\omega$-limit $\q_\omega$ of $\q$ is a transversal geodesic, by (1), and the $\omega$-limit $\p_\omega$ of the sequence $(\p_k)$ is either a $\mu$-bi-Lipschitz path with endpoints on $\q_\omega$, or a $\mu$-bi-Lipschitz ray with origin on $\q_\omega$, or a  $\mu$-bi-Lipschitz bi-infinite path. In all three cases  $\p_\omega$ stays $1$-close to  $\q_\omega$, and has one point $x^\omega = (x_k)^\omega$ at distance $1$ from $\q_\omega$. In the latter two cases, we can obtain a simple path with endpoints on $\q_\omega$ by choosing a point $x$ on the ray far enough from the origin (respectively two points $x,y$ far enough from each other), joining them by a geodesic $[x,x']$ (respectively by two geodesics $[x,x']$ and $[y,y']$) to nearest points on $\q_\omega$, and by replacing $x$ (respectively $x,y$) with the farthest from them intersection point between geodesic and $\p_\omega$. We then get a contradiction as in the end of the proof of $(1)\to (5)$ of Proposition \ref{prop3}.

$(1)\to (5).$ Suppose there exist constants $C>1$ and $c>0$, and a sequence of paths $\p_k$ connecting pairs of points $a_k,b_k$ on $\q$ with $d_k=\dist(a_k,b_k)\to \infty$ such that the length of each $\p_k$ is at most $C\dist(a_k,b_k)$ and $\q_{a_kb_k}$ is not in the $cd_k$-neighborhood of $\p_k$. Let $x_k$ be a point on $\q_{a_kb_k}$ such that $\dist(x_k,\p_k )$ is within distance $1$ of the maximal possible value, and consider the asymptotic cone $\cc=\Con^\omega(X, (x_k), (d_k))$. Then the point $(x_k)$ in $\cc$ is at distance $\ge c$ from the $\omega$-limit $\p_\omega$ of the sequence $(\p_n )$. The $\omega$-limit $\p_\omega$ is a path of length at most $C$, in particular, it is rectifiable (see the end of Remark \ref{r:1}, (1)). Therefore the end of the proof of $(1)\to (5)$ of Proposition \ref{prop3} works and we get a contradiction.

$(5)\to (1)$. Suppose that $(5)$ is true, but that the $\omega$-limit $\q_\omega$ of $\q$ in some asymptotic cone $$\cc=\Con^\omega(X, (x_k), (d_k))$$ is not a transverse geodesic. Hence there exists a path $\p_\omega$ connecting two distinct points $a_\omega=(a_k), b_\omega=(b_k)$ on $\q_\omega$ and having no other common points with $\q_\omega$. By approximating $\p_\omega$ well enough with piecewise geodesics (where the geodesic pieces are ultralimits of geodesics), and by eventually replacing $a_\omega$ and $b_\omega$ with two points that are nearer to each other, we can assume that $\p_\omega$ is itself an $\omega$-limit of a sequence $(\p_k )$ of piecewise geodesics, with a uniformly bounded number of geodesic pieces, $\p_k$ with endpoints $a_k$ and $b_k$ on $\q$. Moreover, the hypothesis that $\p_\omega$ intersects $\q_\omega$ only in its distinct endpoints implies that there exists a point $x_\omega=(x_k)^\omega$ on $\q_\omega$ situated between $a_\omega$ and $b_\omega$, and with $c=\dist(x_\omega, \p_\omega)>0$. We have that the lengths of the paths $\p_k$ are at most $C\dist(a_k,b_k)$ $\omega$-almost surely, for some $C>1$, and that $\dist(x_k,\p_k)\ge \frac c2d_k$ $\omega$-almost surely. This contradicts Property (5).\hspace*{\fill}$\Box$

\medskip

As mentioned before, property (5) from Proposition 3.24 is used in the proofs of Theorems 4.4 and 4.9. In the proof of Theorem 4.9, Property (5) is used only in the following paragraph:
\begin{quote}
Since $\q$ is a $k$-piecewise hierarchy path, by property (T2) it is shadowed by a $k$-piecewise tight
geodesic $\proj(\q)$ in $\calgg$ of length $\leq K_1n$ (for some constant $K_1$) connecting $g^{-3n}\cdot
o$ and $g^{3n}\cdot o$. The fact that geodesics in a hyperbolic graph are Morse and part (5) of
Proposition 3.24 imply that the sub-arc $[g^{-3n}\cdot o , g^{3n}\cdot o]$ in $\g$ is contained
in the $D$-tubular neighborhood of $\proj(\q)$ for some constant $D$. In particular $[g^{-n}\cdot o ,
g^{n}\cdot o]$ has a sub-arc $\g'$ of length $\geq K_2n$ (for some constant $K_2$) contained in the
$D$-tubular neighborhood of one of the tight geodesic subpaths $\ft$ of $\proj (\q )$. Notice that the
length $|\ft|$ is $\geq K_2n-2D\ge K_3n$ for some constant $K_3$ (since $n\gg 1$).
\end{quote}
It is easy to see that here ``$D$-neighborhood" can be replaced by ``$o(n)$-neighborhood". Thus the new Property (5) suffices to prove Theorem 4.9.

The proof of Theorem 4.4 requires more modifications. Here is its formulation.

\begin{theorem}[See Theorem 4.4]\label{err:t32} Let $G$ be an infinite finitely generated group acting on an infinite hyperbolic  uniformly locally finite connected graph $X$.
Suppose that for some $\ell>0$ the stabilizer of any pair of points $x,y\in X$ with $\dist(x,y)\ge \ell$ is
finite of uniformly bounded size. Let $g$ be a loxodromic element of $G$. Then the sequence
$(g^n)_{n\in \ZZ}$ is a Morse quasi-geodesic in $G$. In particular, every asymptotic cone of $G$ has
cut-points.
\end{theorem}

\proof
In what follows $n$ is a large enough natural number.

As in the proof of Theorem 4.4,  we can assume that $g$ stabilizes a geodesic $\p$ in $X$ and acts on $\p$ with translation length $\tau>0$. Rescaling the metric in $X$ if necessry, we can assume that $\tau=1$.

Let us fix a point $o$ in $\p$.
Let $\lambda$ be the maximal distance between $o$ and $a\cdot o$ where $a$ is any of the generators in a finite generating set of $G$. Consider the map $\pi$ from $G$ to $X$ defined by $\pi(h)= h\cdot o$.

Take a path $\g$ from $g^{-3n}$ to $g^{3n}$ in the Cayley graph of $G$ such that the length of $\g$ is at most $Cn$ for some $C\ge 1$. We need to show (by Property (3) of Proposition \ref{prop3}) that $\g$ passes boundedly close to one of $g^i$ where $-n\le i\le n$.

Consider the image $\pi(\g)$ of $\g$ in $X$. By connecting consecutive points on $\pi(\g)$ with geodesics we turn it into a path in $X$ which we shall denote by $\g'$. This path connects two points on $\p$, $a_n=\pi(g^{-3n})$ and $b_n=\pi(g^{3n})$. The length of $\g'$ is at most $Cn\lambda$.

By Remark \ref{r:gr}, there exist constants $c, D_0$ such that for all $D>D_0$ if a path $\h$ in $X$ connecting $x$ and $x'$ is of length $\le \frac D2$ and the distance $\dist(x,\p)$ is greater than $D$, then the diameter of $\proj_\p(\h)$ is at most $c$.
Let $k$ be any integer greater than $200$. We can also assume that $\frac c{D_0}<\frac 1 {kC\lambda}$ and $D_0>2l$, $D_0>2\lambda$.

Take the $4D_0$-neighborhood $N_1$ of the axis $\p$. If $\g'\subset N_1$, then we are done. So suppose that $\g'\setminus N_1$ is not empty. Then $\g'\setminus N_1$ is a union of subpaths $\h_i$ connecting points $t_i,t_i'$ of $\g'$ such that $\dist(t_i, \p) = 4D_0=\dist(t_i',\p)$. Dividing each $\h_i$ into subpaths of lengths between $D_0$ and $2D_0$, we conclude that the diameter of $\proj_\p(\h_i)$ is at most $\frac {c}{D_0}|h_i|$ provided $|h_i|\ge 2D_0$.
Let $H$ be the set of those subpaths $\h_i$ whose lengths are at least $2D_0$. Then the projection of the union of the paths $\h\in H$ is covered by a union $Z$ of intervals of total length at most
\begin{equation}\label{e:1}\frac{c}{D_0}\sum_{\h\in H} |\h|\le \frac {c}{D_0} |\g'|\le \frac 1{k\lambda C} Cn\lambda=\frac{n}{k\tau}.\end{equation}

Now let $N_2$ be the $5D_0$-neighborhood of $\p$ and consider the set of maximal subpaths $\h_i'$ of $\g'\setminus N_2$. Note that each $\h_i'$ is inside some $\h_j$, moreover this $\h_j$ must have length at least $2D_0$, so it belongs to $H$.
Therefore the set $\proj_\p(\cup \h_i')$ is covered by $Z$.

Note that the length of the piece $p'$ of the axis $\p$ between $a_n=\pi(g^{-3n})$ and $b_n=\pi(g^{3n})$ is $6n+1$.
%

Let $m=\lfloor 20D_0\rfloor +1$. Consider the arithmetic progression $P={-n, -n+m, -n+2m,....}$ of numbers between $-n$ and $n$. The distance between any two points $\pi(g^j)$, $\pi(g^k)$, $k\ne j\in P$ is $m|k-j|\geq m$. The size of the set $P$ is $> \frac n m$.

For every $i\in P$ such that $\pi(g^i)\in Z$ let $\z_i$ be the maximal subgeodesic of $\p$ containing $\pi(g^i)$, not containing any other $\pi(g^j), j\in P$, and contained in $Z$. Note that paths $\z_i$ may overlap. So the sum of  lengths of all $\z_i$ is at most twice the measure of $Z$. If $i\in P$ is such that $\pi(g^i)\not\in Z$ then let $\z_i$ be just the point $\pi(g^i)$.
Let $u_i$ be the length of $\z_i$. Then by (\ref{e:1})

\begin{equation}\label{e:2} \sum u_i\le \frac {2n}{k}.\end{equation}

Take any positive $\alpha<1-\frac {100}{k}$ (note that $1-\frac {100}k>0$ since $k>200$). If there are fewer than $\alpha |P|$ consecutive pairs $i,i+m\in P$ such that $u_i,  u_{i+m}\le D_0$,
then there are at least $\frac 12 |P| (1-\alpha)$ numbers $u_i$ which are bigger than $D_0$, hence $\sum_{i\in P} u_i> \frac n{2m}(1-\alpha)D_0>\frac n{2\lfloor 20 D_0\rfloor} \frac {100}{k}D_0>\frac{2n}{k}$, a contradiction with (\ref{e:2}). Hence there are at least $\alpha \frac n m$ pairs of consecutive numbers $i,j=i+m\in P$   such that $u_i,u_j\le D_0$. Let $M$ be the set of these pairs of numbers and $(i, i+m)\in M$. For $s=i, i+m$ let $B_s=\Ball(\pi(g^s),6D_0)$,  be the ball in $X$ of radius $6D_0$ around $\pi(g^s)$.  Let $U_i, U_{i+m}$ be $\pi$-preimages of $B_i, B_{i+m}$ respectively.
The path $\g'$ must visit each of the balls $B_i, B_{i+m}$. Hence the path $\g$ must intersect both sets $U_i, U_{i+m}$ at points $w_i, w_{i+m}$ respectively.

Note that the distance between any point from $B_i$ to any point in $B_{i+m}$ is greater than $l$ since $D_0>2l$.

Since the sum of distances $\sum_{(i,j)\in M} \dist(w_i,w_j)$ does not exceed $Cn$ we have that at least one of the distances $\dist(w_i,w_j)$, $(i,j)\in M$, must be smaller than $R'=\frac {Cn}{\alpha \frac n{m}}=\frac {Cm}{\alpha}$ which does not depend on $n$. Let $R$ be the maximum of $R'$ and $m$, the distance between $g^i\cdot o$ and $g^{i+m}\cdot o$.

We need to show that $\dist(w_i, g^i)=|g^{-i}w_i|$ is bounded by a constant not depending on $n$. We have that $g^{-i}w_i\cdot o\in B=\Ball(o,6D_0)$ and there exists $v\in G$ of length at most $R$ such that $g^{-i}w_iv\cdot o\in B'=\Ball(g^m\cdot o, 6D_0)$.

Let $V$ be the set of all $h\in G$ such that $h\cdot o\in B$ and for some $v$, $|v|\le R$, $hv\cdot o\in B'$. Note that $V$ does not depend on $n$, so it is enough to show that $V$ is finite. 

Recall that in Lemma \ref{lem9} we defined the sets $V_{a,b}$, $a,b\in G\cdot o$, as $$V_{a,b}=\{h\in G\mid h\cdot o=a, \exists v\in G, |v|\le R, hv\cdot o=b\}$$ (see Lemma \ref{lem9}). The proof of Lemma \ref{lem9} applies for every metric space $X$ (not only a tree) with $l$-acylindrical action of a group $G$. It shows that if $\dist(a,b)>l$, then $V_{a,b}$ has uniformly bounded diameter (depending only on $R$).

If we denote $a_h=h\cdot o$ and $b_h=hv\cdot o$, then $h\in V_{a_h,b_h}$. Since $a_h\in B, b_h\in B'$, $\dist(a_h,b_h)>l$. Hence  $V_{a_h,b_h}$ is a finite set. The number of possible such pairs $(a_h,b_h)$ does not exceed the size of the direct product $B\times B'$ which is a finite set because $X$ is a locally finite graph. Hence $V$ is finite as  required.\endproof

\providecommand{\bysame}{\leavevmode\hbox to3em{\hrulefill}\thinspace}
\providecommand{\MR}{\relax\ifhmode\unskip\space\fi MR }
\providecommand{\MRhref}[2]{%
  \href{http://www.ams.org/mathscinet-getitem?mr=#1}{#2}
} \providecommand{\href}[2]{#2}

\bigskip

\begin{minipage}[t]{2in}
\noindent Cornelia Dru\c tu\\
Mathematical Institute\\
24-29 St Giles\\
Oxford OX1 3LB\\
United Kingdom.\\
drutu@maths.ox.ac.uk
\end{minipage}
\begin{minipage}[t]{2in}
\noindent Shahar Mozes\\
Institute of Mathematics\\
Hebrew University\\
Jerusalem, Israel\\
{mozes@math.huji.ac.il}
\end{minipage}
\begin{minipage}[t]{2in}
\noindent Mark V. Sapir\\
SC1522\\
Department of Mathematics\\
Vanderbilt University\\
Nashville, TN 37240, U.S.A.\\
m.sapir@vanderbilt.edu
\end{minipage}

\end{document}